\documentclass[a4paper,reqno]{article}
\usepackage[T1]{fontenc}
\usepackage{authblk}
\pdfoutput=1
\usepackage{amsmath}
\usepackage{amsthm}
\usepackage{amssymb}
\usepackage{mathabx}
\usepackage{times}
\usepackage{bm}
\usepackage{amscd}
\usepackage[latin2]{inputenc}
\usepackage{t1enc}
\usepackage[mathscr]{eucal}
\usepackage{indentfirst}
\usepackage{graphicx}
\usepackage{graphics}
\usepackage{pict2e}
\usepackage{epic}
\numberwithin{equation}{section}
\usepackage[margin=1in]{geometry}
\usepackage{epstopdf} 
\usepackage[colorlinks,linkcolor=blue]{hyperref}

\usepackage[title]{appendix}
\usepackage[capitalise,noabbrev]{cleveref}
\usepackage{todonotes}
\crefformat{equation}{(#2#1#3)}
\crefmultiformat{equation}{(#2#1#3)}{ and~(#2#1#3)}{, (#2#1#3)}{ and~(#2#1#3)}
\crefrangeformat{equation}{(#3#1#4) to~(#5#2#6)}

\usepackage[normalem]{ulem}

\allowdisplaybreaks

\theoremstyle{plain}
\newtheorem{Thm}{Theorem}[section]
\newtheorem*{Thm*}{Theorem}
\newtheorem{Lem}[Thm]{Lemma}

\newtheorem{Prop}[Thm]{Proposition}

\theoremstyle{definition}

\newtheorem{Rem}[Thm]{Remark}
\newtheorem{?}[Thm]{Problem}

\newcommand{\p}{\partial}
\newcommand{\R}{\mathbb{R}}
\newcommand{\e}{\varepsilon}
\newcommand{\er}{\mathbf{n}}

\newcommand{\E}{\mathbf{e}}

\newcommand{\Torus}{\mathbb{T}}

\newcommand{\dv}{\text{div}}
\newcommand{\vs}{\mathsf{vs}}
\newcommand{\od}{\flat}
\newcommand{\md}{\sharp}
\newcommand{\lap}{\triangle}
\newcommand{\dnab}{\cdot \nabla}

\newcommand{\rhob}{\bar{\rho}}
\newcommand{\rhos}{\rho^\vs}
\newcommand{\rhot}{\tilde{\rho}}

\newcommand{\ub}{\bar{u}}
\newcommand{\mb}{\bar{m}}

\newcommand{\ut}{\tilde{u}}
\newcommand{\uv}{\mathbf{u}}
\newcommand{\mv}{\mathbf{m}}
\newcommand{\wv}{\mathbf{w}}
\newcommand{\zv}{\mathbf{z}}
\newcommand{\uvb}{\bar{\mathbf{u}}}
\newcommand{\mvb}{\bar{\mathbf{m}}}
\newcommand{\uvt}{\tilde{\mathbf{u}}}

\newcommand{\mvt}{\tilde{\mathbf{m}}}
\newcommand{\mvs}{\mathbf{m}^\vs}

\newcommand{\mt}{\tilde{m}}

\newcommand{\mut}{\tilde{\mu}}
\newcommand{\xp}{x_{\perp}}

\newcommand{\fv}{\mathbf{f}}
\newcommand{\rv}{\mathbf{r}}

\newcommand{\Fv}{\mathbf{F}}

\newcommand{\gv}{\mathbf{g}}
\newcommand{\Gv}{\mathbf{G}}
\newcommand{\Qv}{\mathbf{Q}}
\newcommand{\Zv}{\mathbf{Z}}

\newcommand{\phib}{\bar{\phi}}
\newcommand{\psib}{\bar{\psi}}

\newcommand{\andd}{\quad \text{and} \quad}
\newcommand{\srt}{\sqrt{t+\Lambda}}
\newcommand{\Kc}{\mathcal{K}}
\newcommand{\Nv}{\mathbf{N}}
\newcommand{\Ac}{\mathcal{A}}
\newcommand{\Bc}{\mathcal{B}}
\newcommand{\Ec}{\mathcal{E}}
\newcommand{\Dc}{\mathcal{D}}

\newcommand{\Rc}{\mathfrak{E}}

\newcommand{\abs}[1]{\left\lvert#1\right\rvert}

\newcommand{\norm}[1]{\left\lVert#1\right\rVert}
\newcommand{\jump}[1]{\left\ldbrack#1\right\rdbrack}

\date{}

\begin{document}
	
\begin{sloppypar}
	
		\title{Nonlinear asymptotic stability of compressible vortex sheets with viscosity effects}
		
		\author{Feimin Huang\textsuperscript{1,2}}
		
		\author{Zhouping Xin\textsuperscript{3}}
		
		\author{Lingda Xu\textsuperscript{4,5}}
		
		\author{Qian Yuan\textsuperscript{1}\thanks{ Corresponding author: qyuan@amss.ac.cn. \\ 
		E-Mails:  fhuang@amt.ac.cn (F. Huang), \ zpxin@ims.cuhk.edu.hk (Z. Xin), \ xulingda@tsinghua.edu.cn (L. Xu)}
		}
		
		\affil{
		\footnotesize
		\begin{flushleft}
		\textsuperscript{1}Academy of Mathematics and Systems Science, Chinese Academy of Sciences, Beijing, China. \\
		\vspace{.07cm} 
%		Email: qyuan@amss.ac.cn \\
		\textsuperscript{2}School of Mathematical Sciences, University of Chinese Academy of Sciences, Beijing, China. \\
		\vspace{.07cm}
		\textsuperscript{3}The Institute of Mathematical Sciences \&  Department of Mathematics, The Chinese University of Hong Kong, Shatin, NT, Hong Kong. \\
		\vspace{.07cm}
		\textsuperscript{4}Department of Mathematics, Yau Mathematical Sciences Center, Tsinghua University, Beijing, China. \\
		\vspace{.07cm}
		\textsuperscript{5}Yanqi Lake Beijing Institute of Mathematical Sciences And Applications, Beijing, China.
		\end{flushleft}}
			
%		\affil{\textsuperscript{2}School of Mathematical Sciences, University of Chinese Academy of Sciences, Beijing, China. \\
%		E-mail: fhuang@amt.ac.cn}
%			
%		\affil{\textsuperscript{3}The Institute of Mathematical Sciences \&  Department of Mathematics, The Chinese University of Hong Kong, Shatin, NT, Hong Kong. \\
%		Email: zpxin@ims.cuhk.edu.hk}
%		
%		\affil{\textsuperscript{4}Department of Mathematics, Yau Mathematical Sciences Center, Tsinghua University, Beijing, China. \\
%		E-mail: xulingda@tsinghua.edu.cn}
%		
%		\affil{\textsuperscript{5}Yanqi Lake Beijing Institute of Mathematical Sciences And Applications, Beijing, China.}
		
\maketitle

\begin{abstract} 
	This paper concerns the stabilizing effect of viscosity on the vortex sheets.
	It is found that although a vortex sheet is not a time-asymptotic attractor for the compressible Navier-Stokes equations, a viscous wave that approximates the vortex sheet on any finite time interval can be constructed explicitly, which is shown to be time-asymptotically stable in the $ L^\infty $-space with small perturbations, regardless of the amplitude of the vortex sheet.
	The result shows that the viscosity has a strong stabilizing effect on the vortex sheets, which are generally unstable for the ideal compressible Euler equations even for short time \cite{Miles1958,FM1963,AM1987}.
	The proof is based on the $ L^2 $-energy method. 
%	One novelty of this paper is that the perturbation can keep oscillating around the vortex sheet at the spatial infinity. 
	In particular, the asymptotic stability of the vortex sheet under small spatially periodic perturbations is proved by studying the dynamics of these spatial oscillations.
	The first key point in our analysis is to construct an ansatz to cancel these oscillations.
	Then using the Galilean transformation, we are able to find a shift function of the vortex sheet such that an anti-derivative technique works, which plays an important role in the energy estimates. 
	Moreover, by introducing a new variable and using the intrinsic properties of the vortex sheet, we can achieve the optimal decay rates to the viscous wave.
\end{abstract}


\vspace{.2cm}


%\tableofcontents

%%%%%%%%%%%%%%%%%%%%%%%%%%%%%%%%%%%%%%%%%%%
%% Introduction
%%%%%%%%%%%%%%%%%%%%%%%%%%%%%%%%%%%%%%%%%%%

\section{Introduction}

The three-dimensional (3D) compressible isentropic Navier-Stokes (NS) equations read,  
\begin{equation}\label{NS}
\begin{cases}
\p_t \rho + \dv \mv = 0, \\
\p_t \mv + \dv ( \rho \uv\otimes\uv ) + \nabla p(\rho) = \mu \lap \uv + \big(\mu+\lambda \big) \nabla \dv \uv,
\end{cases} \quad x\in\R^3, t>0, 
\end{equation}
where $ \rho(x,t)>0 $ is the density, $ \mv(x,t) = (\rho \uv)(x,t) \in \R^3 $ is the momentum with $ \uv(x,t) \in\R^3 $ being the velocity, the pressure $ p(\rho) $ is the gamma law, satisfying $ p(\rho) = \rho^\gamma $ with $ \gamma> 1, $ and the viscous coefficients  $ \mu $ and $ \lambda $ are assumed to satisfy
\begin{equation}\label{viscous}
	\mu > 0 \andd \mu + \lambda \geq 0.
\end{equation}
%Note that the model \cref{NS} with \cref{viscous} includes both two-dimensional (2D) and three-dimensional (3D) compressible NS equations.

\vspace{.2cm}

In the case for $ \mu = \lambda =0, $ \cref{NS} is the 3D compressible isentropic Euler equations, 
\begin{equation}\label{Euler}
\begin{cases}
\p_t \rho + \dv (\rho \uv) = 0, & \\
\p_t (\rho \uv) + \dv ( \rho \uv\otimes\uv ) + \nabla p(\rho) = 0, & 
\end{cases} \qquad x\in\R^3, t>0,
\end{equation}
which admit rich wave phenomena such as shock waves, rarefaction waves and contact discontinuities, i.e. vortex sheets. 
A vortex sheet is an inviscid flow in which the velocity field is discontinuous in a tangential direction across a surface. 
In particular, a planar vortex sheet is given by,
\begin{equation}\label{vs-general}
	(\rhob^\vs, \uvb^\vs)(x,t) = \begin{cases}
		(\rhob, \uvb_-), \qquad x_3 < s t, \\
		(\rhob, \uvb_+), \qquad x_3 > s t,
	\end{cases}
\end{equation}
where $ \rhob>0 $, $ \uvb_\pm = (\ub_{1\pm}, \ub_{2\pm},\ub_{3\pm}) $ and $s$ are constants, satisfying the Rankine-Hugoniot (RH) conditions,
\begin{equation}\label{RH}
	\ub_{3+} - \ub_{3-} =0 \andd - s (\ub_{i+} -\ub_{i-} ) + \ub_{3+} \ub_{i+} - \ub_{3-} \ub_{i-} = 0 \quad \text{for } i =1,2,3,
\end{equation}
which implies that $ \ub_{3+} = \ub_{3-} = s. $
%and $ (\ub_{1+}, \ub_{2+}) \neq (\ub_{1-}, \ub_{2-}) $. 

It is well known that both the compressible NS equations \cref{NS} and Euler equations \cref{Euler} are invariant under the Galilean transformation, i.e. if $ (\rho, \uv) $ solves \cref{NS} (resp. \cref{Euler}), then so does
\begin{equation}\label{Galilean}
	(\rho^*, \uv^*)(x,t) = \big(\rho(x-\mathbf{c} t, t), \uv(x-\mathbf{c} t, t) + \mathbf{c}\big)
\end{equation}
for any constant vector $\mathbf{c}\in \R^3.$
Thus, by selecting $ \textbf{c} = - \big(\frac{\ub_{1+} + \ub_{1-}}{2}, \frac{\ub_{2+} + \ub_{2-}}{2}, s\big) $, one can assume without loss of generality that the planar vortex sheet \cref{vs-general} has the form,
%Then rotating the coordinates with respect to $ x_3 $ axis, one can assume without loss of generality that
%\begin{equation}
%	(\ub_{1-}, \ub_{2-}) = -(\ub_{1+}, \ub_{2+}) = (0, -\delta),
%\end{equation}
%where $ \delta>0 $ is a constant equivalent to the amplitude of the vortex sheet. 
\begin{equation}\label{vs}
	(\rhob^\vs, \uvb^\vs)(x,t) = \begin{cases}
		(\rhob, -\uvb), \qquad x_3 <0, \\
		(\rhob, \uvb), \qquad \quad x_3 >0,
	\end{cases}
\end{equation}
where $ \uvb =  (\ub_1, \ub_2, 0) \neq 0 $ is a constant vector.

\vspace{.2cm}

The stability of compressible interfacial waves is an important issue in gas dynamics. Different from the nonlinear shock waves and rarefaction waves, the linear degeneracy makes the compressible contact discontinuities less stable, which may lead to various instabilities such as the Kelvin-Helmholtz instability for vortex sheets and Rayleigh-Taylor instability for entropy waves (the other kind of contact discontinuities in which the velocity and pressure are continuous); see \cite{Ebin2988,GT2011}.
For the 3D compressible Euler equations, the planar vortex sheets are violently unstable (see e.g. \cite{Serre1999}). 
In two dimensions, the rectilinear compressible vortex sheets are also violently unstable if the Mach number $ M $ is less than $\sqrt{2} $, while the supersonic vortex sheets with Mach number $ M>\sqrt{2} $ are shown to be linearly stable; see \cite{Miles1958,FM1963,Serre2000}. 
The nonlinear stability of the 2D supersonic vortex sheets was then established locally in time by \cite{Coulomb2004,Coulomb2008}; also see \cite{MT2008,MTW2019} for similar stability results in the non-isentropic case.  
However, one cannot expect the global nonlinear stability of the compressible vortex sheets for \cref{Euler}, which was observed by \cite{AM1987} through the argument of nonlinear geometric optics.
We also refer to \cite{Trakhinin2005,CW2008,Trakhinin2009,WY2013} and \cite{WX2023} for the studies of the stabilizing effects of the magnetic fields on the compressible vortex sheets  and entropy waves, respectively.

%There have been many works showing that the magnetic effect can help stabilize the compressible vortex sheets in both two and three dimensions, i.e. the current-vortex-sheets are non-linearly stable; see 

\vspace{.1cm}

In this paper, we are concerned about the viscosity effect on the stability of the compressible vortex sheets with arbitrarily large amplitude. Different from the hyperbolic theory stated above and similar to the meta-stability of the 1D entropy waves for the 1D compressible NS equations first observed in \cite{Xin1996}, we are able to obtain a time-asymptotic stability of the viscous waves associated with the vortex sheets for the compressible NS equations by the $L^2$-energy method. 
We remark here that our results hold true in both three and two dimensions; see \cref{Rem-1}.
Due to the viscosity, the discontinuous solution of \cref{Euler} cannot govern the large-time behaviors of the classical solutions of \cref{NS}, that is, the inviscid vortex sheet \cref{vs} is only a meta-stable state for the Navier-Stokes equations. Thus, the first step in our analysis is to investigate the viscous wave associated with the vortex sheet \cref{vs}.
We refer the readers to \cite{Ilin1960,Xin1996,HMX,HXY,XZ2010} for the viscous waves associated with other wave phenomena such as the shock waves and entropy waves.

\vspace{.1cm} 

Given the vortex sheet \cref{vs} and a fixed constant $\Lambda\geq 1$, we consider the Cauchy problem for \cref{NS} on $ \{(x,t): x\in \R^3, t>-\Lambda \} $ with the initial data,
\begin{equation}\label{ic-vs}
	(\rho, \uv)(x,t=-\Lambda) = \begin{cases}
		(\rhob, -\uvb), \qquad x_3 <0, \\
		(\rhob, \uvb), \qquad \quad x_3 >0.
	\end{cases}
\end{equation}
One can observe that the solution (if exists), denoted by $ (\rho^\vs, \uv^\vs), $ is actually independent of the transverse variables $ \xp = (x_1,x_2) $. Thus, regardless of the transverse derivatives, the 3D NS system \cref{NS} reduces to the 1D one,
\begin{equation}\label{system-1d}
	\begin{cases}
		\p_t \rho + \p_3  m_3 = 0, \\
		\p_t m_1 + \p_3 (\rho u_1 u_3) = \mu \p_3^2 u_1,  \\
		\p_t m_2 + \p_3 (\rho u_2 u_3) = \mu \p_3^2 u_2, \\
		\p_t m_3 + \p_3 \left(\rho u_3^2 + p(\rho)\right) = \mut \p_3^2 u_3,
	\end{cases} \quad x_3\in\R, t>-\Lambda,
\end{equation}
where $ \mut := 2\mu+\lambda > 0 $. It is noted that \cref{system-1d} is indeed the standard 1D compressible isentropic NS system for $ (\rho^\vs, u_3^\vs) $,
\begin{equation}\label{decouple-1}
	\begin{cases}
		\p_t \rho^\vs + \p_3 \left(\rho^\vs u_3^\vs\right) = 0, \\
		\p_t m_3^\vs + \p_3 \left(\rho^\vs \abs{u_3^\vs}^2 + p(\rho^\vs)\right) = \mut \p_3^2 u_3^\vs,
	\end{cases} \qquad x_3 \in \R, t>-\Lambda,
\end{equation}
coupled with two scalar parabolic equations for $ u_1^\vs $ and $u_2^\vs$,
\begin{equation}\label{decouple-i}
	\p_t (\rho^\vs u_i^\vs) + \p_3 (\rho^\vs u_i^\vs u_3^\vs) = \mu \p_3^2 u_i^\vs, \qquad x_3\in \R, t>-\Lambda \qquad \text{for } \ i=1,2.
\end{equation}
One can first solve \cref{decouple-1} with the initial data, $ (\rho^\vs, u_3^\vs)(x_3,t=-\Lambda) = (\rhob, 0), $
which gives the unique solution,
\begin{equation}\label{profile-rho-u3}
	(\rho^\vs, u_3^\vs)(x_3,t) = (\rhob, 0), \qquad x_3 \in \R, t\geq -\Lambda.
\end{equation}
Then substituting \cref{profile-rho-u3} into \cref{decouple-i} yields that
\begin{equation}\label{heat}
	\p_t u_i^\vs = \frac{\mu}{\rhob} \p_3^2 u_i^\vs, \qquad x_3 \in\R, t>-\Lambda \qquad \text{for } \ i=1,2.
\end{equation}
By the classical parabolic theory, the Cauchy problem \cref{heat} with the initial data
\begin{equation}\label{ic-us-i}
	u_i^\vs(x_3, t=-\Lambda) = \begin{cases}
		-\ub_i, \qquad x_3 < 0, \\
		\ub_i, \qquad \quad x_3>0,
	\end{cases}
\end{equation}
admits a unique solution, 
\begin{equation}\label{profile-ui}
	u_i^\vs(x_3,t) = \ub_i \theta(x_3,t), \qquad x_3\in\R, t\geq -\Lambda,
\end{equation}
where $ \theta $ is the solution of
\begin{equation}\label{equ-theta}
	\begin{cases}
		\p_t \theta = \frac{\mu}{\rhob}\p_3^2 \theta, \qquad x_3\in\R, t>-\Lambda, \\
		\theta(x_3, t=-\Lambda) = \begin{cases}
			-1, \qquad x_3<0, \\
			1, \qquad \quad x_3>0,
		\end{cases}
	\end{cases}
\end{equation}
and can be computed explicitly by
\begin{equation}\label{theta}
	\theta(x_3,t) = \Theta(\xi) := \frac{2}{\sqrt{\pi}} \int_{0}^{\frac{1}{2} \sqrt{\frac{\rhob}{\mu}} \xi} e^{-\eta^2} d \eta \qquad \text{with } \ \xi = \frac{x_3}{\srt}.
\end{equation}
Note that $ \Theta $ is independent of $\Lambda$, satisfying that
\begin{equation}\label{equ-Theta}
	\begin{cases}
		-2 \mu \Theta''(\xi) = \rhob \xi \Theta'(\xi), \qquad \xi \in \R, \\
		\lim\limits_{\xi \to \pm\infty} \Theta(\xi) = \pm 1.
	\end{cases}
\end{equation}
%Note that one can use \cref{theta,equ-theta} to prove that

\begin{Lem}\label{Lem-theta}
	The solution $ \theta $ of \cref{equ-theta} is odd and strictly increasing with respect to $ x_3 $. Moreover, for any $ j=1,2, \cdots, $ it holds that
	\begin{equation}\label{est-theta}
		\abs{\p_3^j \theta(x_3,t)} \leq  C_j (t+\Lambda)^{-\frac{j}{2}} \exp \Big\{-\frac{ \rhob \abs{x_3}^2}{8\mu (t+\Lambda)} \Big\}  \qquad \forall x_3 \in \R, t\geq 0.
	\end{equation} 
\end{Lem}

\textit{Viscous wave}.
Combining \cref{profile-rho-u3,profile-ui}, we can define the viscous wave associated with the vortex sheet \cref{vs} as
\begin{equation}\label{profile}
	\rho^\vs(x_3,t) := \rhob \andd \uv^\vs(x_3,t) := \theta(x_3,t) \uvb = \Theta\Big(\frac{x_3}{\srt}\Big) \uvb.
\end{equation}
For any constant $\Lambda>0,$ the viscous wave \cref{profile} is a smooth solution (for $ t\geq 0 $) to the compressible NS equations \cref{NS} and connects the constant states of the vortex sheet, $ (\rhob, \pm\uvb) $,  as $ x_3 \to \pm \infty $ for all $ t\geq 0. $ 
Moreover, for any $ p \in [1,+\infty), $ it holds that
\begin{align*}
	C^{-1} \abs{\uvb} \left[\mu \left(t+\Lambda\right)\right]^{\frac{1}{2p}}\leq \norm{(\rho^\vs, \uv^\vs) - (\rhob^\vs, \uvb^\vs)}_{L^p(\R; \text{d}x_3)} \leq C \abs{\uvb} \left[\mu \left(t+\Lambda\right)\right]^{\frac{1}{2p}}
\end{align*}
for some constant $C\geq 1$ that depends on $p.$
Thus, as the viscosity $\mu$ vanishes, the viscous wave \cref{profile} approximates the vortex sheet \cref{vs} on any finite time interval, but not time-asymptotically.
%for any $ a>0, $ it holds that
%\begin{align*}
%	\sup_{ \substack{ x_3\in\R, t\geq 0, \\ \text{with } \abs{x_3} \geq a \srt} } \abs{(\rho^\vs, \uv^\vs) - (\rhob^\vs, \uvb^\vs)} \to 0 \qquad \text{as } \mu \to 0+.
%\end{align*}

\vspace{.2cm}

In this paper, we first study a spatially periodic perturbation of the viscous wave \cref{profile}, i.e. 
\begin{equation}\label{ic}
	(\rho, \mv)(x,0) = (\rhos, \mvs)(x_3,0) +  (v_0, \wv_0 )(x), \quad x \in \R^3,
\end{equation}
where $ (v_0, \wv_0)(x) =  (v_0, w_{01}, w_{02}, w_{03} )(x) $ is a periodic function on $ \Torus^3 = [0,1]^3 $ with zero average,
\begin{equation}\label{zero-ave}
	\int_{\Torus^3} (v_0, \wv_0)(x) dx = 0.
\end{equation} 
Then, we will show in \cref{Sec-local} that the similar arguments can also be applied to study a localized perturbation on the domain $ \Omega := \Torus^2 \times \R$, i.e.
\begin{equation}\label{ic-loc}
	(\rho, \mv)(x,0) = (\rhos, \mvs)(x_3,0) +  (\phib_0, \psib_0 )(x), \quad x \in \Torus^2 \times\R,
\end{equation}
where $(\phib_0,\psib_0)$ is periodic with respect to the transverse variables $\xp=(x_1,x_2)\in\Torus^2$ and belongs to $ H^3(\Torus^2\times\R).$

The compressible vortex sheets are contact discontinuities for the multi-dimensional compressible Euler equations. For the 1D viscous conservation laws with uniform viscosity, \cite{Xin1996} was the first to prove that the inviscid contact discontinuities are meta-stable for the viscous models, and then  \cite{Liu-Xin,XZ2010} constructed asymptotic viscous contact waves and established the pointwise stability via the approximate Green function approach. For the 1D NS equations in Lagrangian coordinates, \cite{HMS,HMX,HXY,HLM} applied the basic $L^2$-energy method to achieve the asymptotic stability of the contact discontinuities.
%However, these works do not include the physical model \cref{NS} (which has partial viscosity) in Eulerian coordinates. 
This paper gives the first stability results of the contact discontinuities for the compressible NS equations in Eulerian coordinates. Moreover, we can achieve the optimal decay rates; see \cref{Rem-opt}.

Note that for the problem \cref{NS}, \cref{ic}, although the perturbation $(\rho,\mv)-(\rho^\vs, \mv^\vs) $ is periodic with respect to $x\in\Torus^3$ initially, it does not remain periodic with respect to $x_3$ for all positive time, due to the non-triviality of the background wave \cref{profile} and the nonlinearity of \cref{NS}. 
Thus, it is more challenging to deal with the oscillating perturbation \cref{ic} than the localized one \cref{ic-loc}.
On the other hand, the theories of the compressible Navier-Stokes equations and Euler equations are connected to each other very closely, and the study of periodic perturbations is an important and interesting topic in the hyperbolic theories.
In particular, for the isentropic compressible Euler equations, the spatially periodic solutions decay to constant states as $t\to+\infty$; see \cite{Lax1957,Glimm1970,Dafe1995,Dafe2013}. However, there appears a resonance phenomenon in the non-isentropic case; as a result, the spatially periodic solutions may oscillate in time simultaneously, which never happens neither in the isentropic case nor for the BV solutions; see \cite{MR1984,TY1996,QX2015,TY2023}.
Recently, the works \cite{XYYsiam,XYYindi,YY2019} also found a new phenomenon of the periodic perturbations in the stability theory of shocks, that is, a new kind of shock shifts is generated by the periodic oscillations, which generally depend on the structures of the equations, while in the case for the localized perturbations, the shock shifts depend only on the initial data; see \cite{Ilin1960,Liu-vshock,SX} for instance.

In this paper, we show that with the initial perturbations in \cref{ic,ic-loc}, the viscous wave associated with the vortex sheet is time-asymptotically stable in the $L^\infty$-norm, where the velocity jump can be arbitrarily large.
The framework is based on an anti-derivative technique and the $L^2$-energy method, that is, we need to study the integrated system of the perturbation in addition. The anti-derivative technique was initiated by \cite{Ilin1960} to prove the shock stability for the 1D scalar conservation laws, and then was widely used for more 1D models to study the stability of both shocks and contact discontinuities; see \cite{MN,G1986,Liu-vshock,SX,Liu-Xin,HMX,HXY,XZ2010} for instance. In the multi-dimensional case, as shown by \cite{Yuan2023n,Yuan2023s}, this method is also effective to prove the stability of planar viscous shock profiles with spatially periodic perturbations for the NS equations. A key observation in \cite{Yuan2023n,Yuan2023s} is that, with the aid of Poincar\'{e} inequality, it suffices to use the anti-derivative technique in the estimate of the zero modes associated with the multi-dimensional perturbations.
It is noted that a premise in the application of the anti-derivative technique is that the perturbations should be of zero masses for all time.
In the previous works concerning localized perturbations, one can just choose constant shifts of the background waves to make the excessive masses of the perturbations vanish, and in the case for the spatially periodic perturbations, the shifts should be some specific functions of time to cancel the influence of the oscillations; see \cite{XYYindi,HY1shock,YY2022,Yuan2023n}.
However, for the vortex sheets with periodic perturbations, there is a new difficulty in determining the shift curves. 
More precisely, the previous analysis relies essentially on the assumption that crossing the interface of the background wave, all the quantities, i.e. the density, momentum and total energy, must be discontinuous. 
However, this assumption fails for the vortex sheets, whose density and normal momentum are continuous across the interface. 
To overcome this technical difficulty, in this paper we have an important observation that for the problem \cref{NS}, \cref{ic}, the two asymptotic behaviors of the solution as $x_3\to - \infty$ and $+\infty$ coincide with each other after a Galilean transformation; see \cref{motiv} and \cref{Lem-coincide}.
Based on this invariance, we are able to find one shift curve of the vortex sheet such that the perturbation is of zero mass for all time.  This is the first novelty of this paper.
The second key element in our analysis is that we introduce a new variable in terms of the anti-derivatives (see \cref{Zv}), which is important in overcoming the difficulty due to the large amplitude of the vortex sheet and achieving the optimal decay rates without the aid of the Green function.

%and using the intrinsic properties of the background vortex sheet, we are able to establish the lower-order estimates and achieve the optimal decay rates of the zero modes. 
%The zero-mode estimates occupy the most important part in our stability analysis. 
%On the other hand, by using the Poincar\'{e} inequality and a similar argument as in \cite{Yuan2023n}, one can prove the lower-order estimates of the non-zero modes with exponential decay rates.
%Afterwards, we return to the original perturbed system \cref{equ-phizeta} to estimate the higher-order derivatives, which can complete the proof of the a priori estimates. 
%It is expected that the analysis of this paper could be developed further to study the stability of more compressible shear flows.

\vspace{.1cm}

Now we state the main results of this paper.
\begin{Thm}\label{Thm}
	Let $ (\rhob^\vs, \uvb^\vs)(x_3,t) $ be the given vortex sheet \cref{vs} and $ (\rho^\vs, \uv^\vs)(x_3,t) $ be the corresponding viscous wave, which is a smooth solution to the 3D compressible Navier-Stokes equations with the form \cref{profile}. 
	Then for the Cauchy problem \cref{NS}, \cref{ic}, there exist $ \Lambda_0 \geq 1 $ and $ \e_0>0 $ such that if 
	\begin{equation}\label{small}
		\Lambda\geq \Lambda_0 \andd \norm{ v_0,\wv_0}_{H^6(\Torus^3)} \leq \e_0 \Lambda^{-\frac{1}{4}},
	\end{equation}
	then there exits a unique classical bounded solution $ (\rho,\uv)(x,t) $  globally in time, which is periodic in the transverse variables $ \xp=(x_1,x_2)\in \Torus^2 $.
	Moreover, the solution satisfies that
%	\begin{align*}
%		
%	\end{align*}
%	satisfies that
	\begin{equation}\label{behavior}
		\begin{aligned}
			\sup_{x_3\in \R} \abs{ (\rho^\od, \uv^\od)(x_3,t) - (\rho^\vs, \uv^\vs)(x_3,t) } & \leq C \sqrt{\e_0}  (t+1)^{-\frac{3}{4}}, \\
			\sup_{x\in \R^3} \abs{(\rho^\md, \uv^\md)(x,t)} & \leq C \sqrt{\e_0} e^{- \alpha_0 t},
%			\sup_{x\in \R^3} \abs{ \nabla \big[  (\rho,\uv)(x,t)- (\rho^\vs, \uv^\vs)(x_3,t) \big] } & \leq C \e^{\frac{1}{2}} (t+1)^{-1}.
		\end{aligned}
	\end{equation}
%	the non-zero mode of the solution, i.e.
%	\begin{align*}
%		
%	\end{align*}
%	satisfies that
%	\begin{equation}\label{beh-md}
%		
%	\end{equation}
	where $(\rho^\od,\uv^\od)$ and $(\rho^\md,\uv^\md)$ denote the zero mode and non-zero mode of the solution, respectively, namely,
	\begin{equation}\label{od-md}
		\begin{aligned}
			(\rho^\od, \uv^\od)(x_3,t) & := \int_{\Torus^2} (\rho, \uv)(\xp, x_3,t) d\xp, \\
			(\rho^\md, \uv^\md)(x,t) & := (\rho,\uv)(x,t) - (\rho^\od, \uv^\od)(x_3,t);
		\end{aligned}
	\end{equation}
	and $ C>0 $ and $ \alpha_0>0 $ in \cref{behavior} are some generic constants.
\end{Thm}

\begin{Thm}\label{Thm-loc}
	If in \cref{Thm}, the initial data \cref{ic} is replaced by \cref{ic-loc}, then there exist $ \Lambda_0 \geq 1 $ and $ \e_0>0 $ such that if 
	\begin{equation}\label{small*}
		\Lambda\geq \Lambda_0 \andd \norm{ \phib_0, \psib_0}_{L^2_{3}(\Omega)}  + \norm{\nabla\phib_0, \nabla\psib_0}_{H^2(\Omega)} \leq \e_0,
	\end{equation}
	where $ \norm{\cdot}_{L^2_{3}(\Omega)}^2 := \int_\Omega (x_3^2 +1 )^{\frac{3}{2}} \abs{\cdot}^2 dx, $
	then for the problem \cref{NS}, \cref{ic-loc}, all the results still hold true, except  the decay rate of the zero mode, which is replaced by
	\begin{equation}\label{behavior-loc}
		\begin{aligned}
			\sup_{x_3\in \R} \abs{ (\rho^\od, \uv^\od)(x_3,t) - (\rho^\vs, \uv^\vs)(x_3,t) } & \leq C \sqrt{\e_0}  (t+1)^{-\frac{1}{2}}.
		\end{aligned}
	\end{equation}
	
%	where $(\rho^\od,\uv^\od)$ and $(\rho^\md,\uv^\md)$ are given by \cref{od-md} and $ C>0 $ and $ \alpha_0>0 $ in \cref{behavior-loc} are some generic constants.
\end{Thm}

\begin{Rem}\label{Rem-1}
	All the conclusions in Theorems \ref{Thm} and \ref{Thm-loc} hold true also in two dimensions, since $\ub_1$ in \cref{vs} can be taken to be zero, and the initial perturbations in \cref{ic,ic-loc} are allowed to be independent of $ x_1 $. 
\end{Rem}

\begin{Rem}
	It should be emphasized that there is no requirement on the size of wave strength (i.e. $ \abs{\uvb_+ -\uvb_-}$) in the above stability theorems. 
	However, our analysis depends crucially on the smallness of the derivatives of the viscous wave, which is provided by the largeness requirement of $ \Lambda. $ We refer to \cite{MN3} for a similar treatment for the asymptotic stability of large-amplitude rarefaction waves.
	In fact, if assuming the smallness of the size of the background flow (i.e. $ \abs{\uvb_+ -\uvb_-}$ is small) instead of the largeness of $ \Lambda $ (for instance, setting $\Lambda=1$ as in \cite{Xin1996,Liu-Xin,XZ2010,HXY}), we obtain all results in this paper by following the analysis here (or even a simpler one).	
\end{Rem}

\begin{Rem}
	To our best knowledge, these are the first results on the effects of viscosities for the compressible vortex sheets. These are also the first multi-dimensional stability results on vortex sheets in a time-asymptotic sense.
\end{Rem}

\begin{Rem}
	1) The zero-average condition \cref{zero-ave} is just assumed without loss of generality. 
	In fact, if the average, $ \int_{\Torus^3} (v_0, \wv_0)(x) dx := (\bar{v}, \bar{\wv}), $ is nonzero, then one can rewrite the initial data \cref{ic} as
	\begin{equation}\label{ic-Rem}
		(\rho_0, \mv_0) = (\rho^\vs + \bar{v}, \mv^\vs + \bar{\wv}) + (v_0-\bar{v}, \wv_0 - \bar{\wv}).
	\end{equation}
	Recall that $ \rho^\vs = \rhob >0 $ is a constant.
	Then using the Galilean transformation \cref{Galilean} with $ \mathbf{c} = - \rhob^{-1} \bar{\wv} $, the initial data \cref{ic-Rem} reduces to 
	\begin{equation*}
		(\rho_0, \mv_0) = (\rho^\vs + \bar{v}, \mv^\vs) + (v_0-\bar{v}, \wv_0 - \bar{\wv}).
	\end{equation*}
	If $ \abs{\bar{v}} $ is suitably small, $ (\rho^\vs + \bar{v}, \mv^\vs) $ is actually the viscous wave associated with the vortex sheet,
	\begin{equation*}
		(\rho,\uv) = \begin{cases}
			(\rhob+\bar{v}, -\frac{\rhob}{\rhob+\bar{v}} \uvb), \qquad x_3 < 0, \\
			(\rhob+\bar{v}, \frac{\rhob}{\rhob+\bar{v}} \uvb), \qquad\quad x_3 > 0.
		\end{cases}
	\end{equation*}
	
	2) In \cref{small*}, the $L^2_{3}(\Omega)$-integrability of $(\phib_0,\psib_0)$ is to ensure that the anti-derivatives of the perturbations belong to $L^2(\R)$; see \cref{L2-Anti} for the details.
\end{Rem}

\begin{Rem}\label{Rem-opt}
	In the special case that the initial perturbation in \cref{ic} or \cref{ic-loc} is independent of the transversal variables, the Cauchy problem for \cref{NS} reduces to a 1D one for \cref{system-1d}, whose associated hyperbolic system, 
	\begin{equation}\label{system-inv}
		\begin{cases}
			\p_t \rho + \p_3 m_3 = 0, \\
			\p_t m_i + \p_3 \big(\frac{1}{\rho} m_3 m_i\big) = 0, \qquad \quad i=1,2, \\
			\p_t m_3 + \p_3 \big(\frac{1}{\rho} m_3^2 + p(\rho) \big) = 0,
		\end{cases}
	\end{equation}
	has two linearly degenerate characteristics with eigenvalues $\lambda_1=\lambda_2 = u_3$ (see \cref{Sec-local}). One can easily verify that the contact discontinuities of \cref{system-inv}, associated with $\lambda_1$ or $\lambda_2$ (see \cite{Lax1957}), just coincide with the zero mode of the planar vortex sheet \cref{vs-general} and satisfy the RH conditions \cref{RH}.
	If the hyperbolic system \cref{system-inv} has uniform viscosity, \cite{XZ2010} constructed an asymptotic ansatz that consists of a viscous contact wave and some diffusion waves propagating in the transversal characteristic fields, and established the pointwise estimates with the optimal decay rate $t^{-\frac{3}{4}}$. 
	Here the viscous contact wave constructed in \cite{XZ2010} is just the same as the zero mode of \cref{profile}.
	Thus, this paper actually extends the stability results of \cite{XZ2010} to a multi-dimensional and physical model. 
	On the other hand, compared to the results in \cite{XZ2010}, one can see that the rates in \cref{behavior,behavior-loc} are both optimal.
	Indeed, the spatially periodic perturbations do not generate diffusion waves in the transversal characteristic fields (see \cref{ansatz}), thus we are able to achieve the rate $t^{-\frac{3}{4}}$; while in the case for the localized perturbations, the rate is at most same as the diffusion waves, that is $t^{-\frac{1}{2}}.$
	
\end{Rem}

\vspace{.2cm}

\textit{Notations}. Throughout this paper, we use the following notations. 
\begin{itemize}
	\item $ \Omega $ denotes the infinitely long nozzle domain
	\begin{equation*}
		\Omega := \Torus^2 \times \R.
	\end{equation*}
%
%	\item Denote $ \delta := \abs{\ub_1} + \abs{\ub_2} $ and $ \e := \norm{(v_0, \wv_0)}_{H^6(\Torus^3)}. $
	
	\item $ \jump{\cdot} $ denotes the difference of the states associated with the vortex sheet \cref{vs}. For instance, $ \jump{\uv} = \uv_+ - \uv_- $ and $ \jump{u_1 m_1} = \ub_{1+} \mb_{1+} - \ub_{1-} \mb_{1-}$, etc.
	
	\item $ C \geq 1 $ is a generic constant. The conventions $ A \lesssim B $, $ A \gtrsim B $, $ A\sim B $ and $ A = O(1) B $ mean $ A \leq CB $, $ A \geq C^{-1} B $, $ C^{-1} B \leq A \leq CB $ and $ \abs{A} \leq C \abs{B} $, respectively.
	
	\item $ \E_i \ (i=1,2,3) $ denotes the $ i $-th column of the $ 3\times 3 $ identity matrix Id$ _{3\times3}, $ and $ \delta_{ij} $ denotes the Kronecker delta function.
	
	\item For any vector $ \mathbf{v}=(v_1,v_2,v_3) $, $ \mathbf{v}_\perp $ denotes
	\begin{equation}\label{perp}
		\mathbf{v}_\perp := (v_1, v_2).
	\end{equation}
	
	\item For any $ f(x) \in L^\infty(\R^3) $ that is periodic in the spatial variables $ \xp = (x_1,x_2) \in \Torus^2, $ $ f^\od $ and $ f^\md $ denote its \textit{zero mode},
	\begin{equation}\label{od}
		f^\od(x_3) := \int_{\Torus^2} f(\xp, x_3) d\xp,
	\end{equation}
	and \textit{non-zero mode},
	\begin{equation}\label{md}
		f^\md(x) := f(x) - f^\od(x_3),
	\end{equation}
	respectively.
\end{itemize}

\vspace{0.3cm}

\textbf{Outline of the paper}. 
In \cref{Sec-ansatz}, we construct the key ansatz in our analysis, which relies on the Galilean transformation and a careful choice of a shift function. In \cref{Sec-prob}, we first formulate the perturbed system, the zero-mode system and the anti-derivative system, respectively, and then give some useful lemmas. \cref{Sec-apriori} is devoted to the a priori estimates, consisting of three subsections showing the lower-order estimates of the zero-modes and non-zero modes, and the higher-order estimates of the original perturbations, respectively. Finally, in \cref{Sec-rate} we prove the optimal decay rate of the perturbations and complete the proof of the main results.

%%%%%%%%%%%%%%%%%%%%%%%%%%%%%%%%%%%%%%%%%%%%%%%%%%%%%
%%%%%%%%%%%%%%%%%%%%%%%%%%%%%%%%%%%%%%%%%%%%%%%%%%%%%

\vspace{0.3cm}

\section{Construction of Ansatz}\label{Sec-ansatz}

As the solution $ (\rho,\mv) $ to the Cauchy problem \cref{NS}, \cref{ic} is periodic in the transverse directions $ \xp\in\Torus^2 $, it suffices to consider the problem on the domain $ \Omega=\Torus^2 \times \R. $ 
However, the perturbation, $ (\rho,\mv) - (\rho^\vs, \mv^\vs), $ keeps oscillating as $ \abs{x_3} \to +\infty $, then in order to use the energy method, it is necessary to construct a suitable ansatz to cancel out the oscillations.
Motivated by \cite{XYYindi,Yuan2023n}, it is plausible that as $ x_3 \to \pm\infty, $
\begin{equation}\label{motiv}
	\abs{(\rho,\mv) - (\rho_\pm, \mv_\pm)}(x,t) \to 0 \qquad \forall \xp \in \R^2, t>0.
\end{equation}
Here $ (\rho_\pm, \mv_\pm) = (\rho_\pm, m_{1\pm}, m_{2\pm},m_{3\pm})(x,t) $ denotes the periodic solution of \cref{NS}, satisfying the periodic initial data 
\begin{equation}\label{ic-per}
	(\rho_\pm, \mv_\pm)(x,0) = (\rhob, \mvb_\pm) + (v_0, \wv_0)(x), \qquad x\in\R^3,
\end{equation}
where $ (\rhob, \mvb_\pm) = (\rhob, \pm\rhob \uvb) $ is the plus or minus constant state of the vortex sheet \cref{vs} and $ (v_0, \wv_0) $ is the periodic perturbation in \cref{ic}. 
The global existence and exponential decay rate of $ (\rho_\pm, \mv_\pm) $ can be found in the following lemma, which has been shown in \cite{HXY2022}.
\begin{Lem}[\cite{HXY2022}]\label{Lem-per}
	Consider the Cauchy problem \cref{NS} with the periodic initial data
	\begin{equation}\label{ic-periodic}
		(\rho, \mv)(x,0) = (\rhob,\mvb) + (v_0, \wv_0)(x), \quad x\in\R^3,
	\end{equation}
	where  $ (\rhob, \mvb) $ is any constant state with $ \rhob > 0 $ and $ (v_0, \wv_0)  $ is any periodic function on $ \Torus^3 = [0,1]^3 $ with zero average.
	Then there exists $ \e_0>0 $ such that if 
	\begin{equation*}
		\norm{v_0,\wv_0}_{H^{k+2}(\Torus^3)} \leq \e_0 \quad \text{for some } \ k\geq 1,
	\end{equation*}
	then the Cauchy problem \cref{NS}, \cref{ic-periodic} admits a unique global periodic solution $ (\rho, \mv) \in C\big((0,+\infty); H^{k+2}(\Torus^3)\big) $, satisfying that
	\begin{equation*}
		\int_{\Torus^3} (\rho - \rhob, \mv - \mvb )(x,t) dx = 0,\qquad t\geq 0,
	\end{equation*}
	and 
	\begin{equation*}
		\norm{(\rho, \mv) - (\rhob, \mvb) }_{W^{k,\infty}(\R^3)} \lesssim \norm{v_0,\wv_0}_{H^{k+2}(\Torus^3)} e^{-c t}, \qquad t\geq 0.
	\end{equation*}
	Here the constant $ c>0 $ is a constant, independent of $ \norm{v_0,\wv_0}_{H^{k+2}(\Torus^3)} $ and $ t. $ 
\end{Lem}
%\red{Need to emphasize the relationship between $\delta$ and $\varepsilon$.}

It follows from \cref{Lem-per} that for $ \e= \norm{v_0, \wv_0}_{H^6(\Torus^3)} \leq \e_0 $ with $ \e_0 $ being suitably small, the periodic solution $ (\rho_\pm, \mv_\pm) $ to the Cauchy problem \cref{NS}, \cref{ic-per} exists in $ C(0,+\infty; W^{4,\infty}(\R^3)) $ with the constant average $ (\rhob, \mvb_\pm) $, and for some generic constant $ \alpha >0, $ the periodic perturbation,
\begin{equation}\label{vw-pm}
	(v_\pm, \wv_\pm, \zv_\pm) := (\rho_\pm,\mv_\pm,\uv_\pm) - (\rhob, \mvb_\pm, \uvb_\pm),
\end{equation}
satisfies that
\begin{equation}\label{exp-per}
	\norm{(v_\pm, \wv_\pm, \zv_\pm) }_{W^{4,\infty}(\R^3)} \lesssim \e e^{- 2\alpha t}, \qquad t\geq 0.
\end{equation}

\vspace{.3cm}

Now we are ready to construct the ansatz. Note that the viscous wave \cref{profile} can be rewritten as 
\begin{equation}\label{profile-1}
	(\rho^\vs, \mv^\vs)(x_3,t) = \frac{1}{2} \Big[ (\rhob, \mvb_-) \Big(1-\Theta\Big(\frac{x_3}{\srt}\Big) \Big) + (\rhob, \mvb_+) \Big(1+\Theta\Big(\frac{x_3}{\srt}\Big) \Big) \Big].
\end{equation}

\textit{Asymptotic Ansatz.} 
Inspired by \cref{motiv}, \cref{profile-1} and \cite{XYYindi,Yuan2023n}
we set the \textit{ansatz} as 
\begin{equation}\label{ansatz}
	\begin{aligned}
		(\rhot, \mvt)(x,t) & := \frac{1}{2} \Big[ (\rho_-, \mv_-)(x,t) \Big(1-\Theta\Big(\frac{x_3-\sigma(t)}{\srt}\Big)\Big) \\
		& \qquad\quad + (\rho_+, \mv_+)(x,t) \Big(1+\Theta\Big(\frac{x_3-\sigma(t)}{\srt}\Big)\Big) \Big],
	\end{aligned}
\end{equation}
%\begin{equation}\label{ansatz}
%	\begin{aligned}
	%		\rhot(x,t) & := \frac{1}{2} \big[ \rho_-(x,t) \big(1-\eta(x_3-\sigma(t),t)\big) + \rho_+(x,t) \big(1+\eta(x_3-\sigma(t),t)\big) \big], \\
	%		\mt_i(x,t) & := \frac{1}{2} \big[ m_{i-}(x,t) \big(1-\eta(x_3-\sigma(t),t)\big) + m_{i+}(x,t) \big(1+\eta(x_3-\sigma(t),t)\big)\big], 
	%	\end{aligned}
%\end{equation}
where $ \sigma = \sigma(t) $ is a curve to be determined later. For convenience, denote
\begin{equation*}
%	\Theta^{(k)}(\xi) := \frac{d^k}{d\xi^k} \Theta(\xi) \andd 
	\Theta_{\sigma} := \Theta\Big(\frac{x_3-\sigma}{\srt}\Big).
\end{equation*}
Define also 
\begin{equation}\label{uvt}
	\uvt := \frac{\mvt}{\rhot},
\end{equation}
which satisfies that
\begin{equation}\label{uvt-1}
	\begin{aligned}
		\uvt & = \frac{1}{2} \big[ \uv_- (1-\Theta_{\sigma}) + \uv_+ (1+\Theta_{\sigma}) \big] \\
		& \quad + \frac{1}{4\rhot} (v_+ - v_-) (\uv_+ - \uv_-) \left(1-\Theta_{\sigma}^2 \right).
	\end{aligned}
\end{equation}
Recall that $ (\rho_\pm, \mv_\pm) $ are solutions of \cref{NS}. Then direct calculations yield that
\begin{equation}\label{equ-t}
	\begin{cases}
		\p_t \rhot + \dv \mvt = f_0, \\
		\p_t \mvt + \dv (\uvt \otimes \mvt ) + \nabla \big(p(\rhot)\big) - \mu \lap \uvt -(\mu+\lambda) \nabla \dv \uvt \\
		\qquad\qquad\quad = \sum\limits_{i=1}^{3} \p_i \Fv_{1,i} + \fv_2 := \gv = (g_1,g_2,g_3),
		%		\p_t \mvt + \dv \big(\frac{\mvt\otimes\mvt}{\rhot} \big) + \nabla p(\rhot) - \mu \lap \uvt -(\mu+\lambda) \nabla \dv \uvt \\
		%		\qquad\qquad\quad = \sum\limits_{i=1}^{3} \p_i \Fvh_{1,i} + \fvh_2,
	\end{cases}
\end{equation}
where $ f_0 $, $ \Fv_{1,i} = (F_{1,i1}, F_{1,i2}, F_{1,i3}) $ and $ \fv_2 := (f_{2,1}, f_{2,2}, f_{2,3}) $ are errors, given by
\begin{equation}\label{f0}
	f_0 = \frac{\Theta'_\sigma}{2\srt} \Big[ - (\rho_+ - \rho_-) \Big(\sigma'(t) + \frac{x_3-\sigma(t)}{2(t+\Lambda)}\Big) + (m_{3+} - m_{3-}) \Big],
	%	+ \sum_{i=1}^2 \p_i (m_{i+} - m_{i-}) (\eta_{\sigma_i} - \eta) \big],
\end{equation}
\begin{equation}\label{F1}
	\begin{aligned}
		\Fv_{1,i} & = \ut_i \mvt - \frac{1}{2} u_{i-} \mv_- (1-\Theta_{\sigma}) - \frac{1}{2} u_{i+} \mv_+ (1+\Theta_{\sigma}) \\
		& \quad + \big[ p(\rhot) - \frac{1}{2} p(\rho_-) (1-\Theta_{\sigma}) - \frac{1}{2} p(\rho_+) (1+\Theta_{\sigma}) \big] \E_i \\
		& \quad - \mu \big[ \p_i \uvt - \frac{1}{2} \p_i \uv_- (1-\Theta_{\sigma})  - \frac{1}{2} \p_i \uv_+ (1+\Theta_{\sigma}) \big] \\
		& \quad - (\mu+\lambda) \big[ \dv \uvt - \frac{1}{2}\dv \uv_- (1-\Theta_{\sigma}) - \frac{1}{2}\dv \uv_+ (1+\Theta_{\sigma}) \big] \E_i, 
	\end{aligned}
\end{equation}
and
\begin{equation}\label{f2}
	\begin{aligned}
		\fv_{2} & := \frac{\Theta_{\sigma}'}{2\srt} \Big\{ -(\mv_+- \mv_-)  \Big(\sigma'(t) + \frac{x_3-\sigma(t)}{2(t+\Lambda)}\Big)   \\
		& \qquad\qquad\qquad  + u_{3+} \mv_+ - u_{3-} \mv_- - \mu \p_3 (\uv_+ - \uv_-) \\
		& \qquad\qquad\qquad + \big[ p(\rho_+) - p(\rho_-) - (\mu+\lambda) \dv (\uv_+ - \uv_-)  \big]  \E_3 \Big\} ,
	\end{aligned}
\end{equation}
respectively. 

\vspace{.2cm}

%
%Note that both the solution $ (\rho,\mv) $ of \cref{NS},\cref{ic} and the ansatz \cref{ansatz} are periodic in the transverse variables $ \xp=(x_1,x_2). $ Integrating the systems \cref{NS} and \cref{equ-t} over the transverse  domain $ \Torus^2 $ yields that
%\begin{equation}\label{equ-od}
%	\begin{cases} 
%		\p_t \rho^\od + \p_3 m_3^\od = 0, \\
%		\p_t \mv^\od + \p_3 \big[ (u_3 \mv)^\od + (p(\rho))^\od \E_3 \big] - \mu \p_3^2 \uv^\od - (\mu+\lambda) \p_3^2 u_3^\od \E_3 = 0,
%	\end{cases}
%\end{equation}
%and
%\begin{equation}\label{equ-t-od}
%	\begin{cases} 
%		\p_t \rhot^\od + \p_3 \mt_3^\od = f_0^\od, \\
%		\p_t \mvt^\od + \p_3 \big[ (\ut_3 \mvt)^\od  + (p(\rhot))^\od \E_3 \big] - \mu \p_3^2 \uvt^\od - (\mu+\lambda) \p_3^2 \ut_3^\od \E_3 \\
%		\qquad\qquad\qquad = \p_3 \Fv_{1,3}^\od + \fv_2^\od = \gv^\od,
%	\end{cases}
%\end{equation}
%respectively.
Recall that $ \Theta \to \pm 1 $ as $ x_3\to \pm\infty. $ Then for all $ \xp \in \Torus^2 $ and $ t\geq 0 $, it holds that
\begin{align*}
	\abs{\rhot - \rho_\pm} + \abs{\mvt - \mv_\pm} + \abs{\uvt - \uv_\pm} \to 0 \quad \text{as } \ x_3\to \pm\infty.
\end{align*}
Then for $ i=1,2,3 $ and all $ \xp \in \Torus^2 $ and $ t\geq 0 $, it holds that
\begin{align*}
	\abs{\Fv_{1,i}(x,t)} \to 0  \quad \text{as } \ x_3\to \pm\infty,
\end{align*}
which, together with the dominated convergence theorem, yields that
\begin{align*}
	\abs{\Fv_{1,i}^{\od}(x_3,t)} \to 0 \quad \text{as } \ x_3\to \pm\infty.
\end{align*}
Thus, by subtracting \cref{equ-t} from \cref{NS} and integrating the resulting equations over $ \Torus^2\times\R $, one can get that
\begin{equation}\label{ode-mass}
	\frac{d}{dt} \Big[\int_\R \big(\rho^\od-\rhot^\od, \mv^\od - \mvt^\od \big) dx_3\Big] = - \int_\R \big(f_0^\od, \fv_2^\od \big)(x_3,t) dx_3, \qquad t>0,
\end{equation}
where one can obtain from \cref{f0,f2} that
\begin{equation}\label{f-0-od}
	\begin{aligned}
		f_0^\od & = \frac{\Theta_\sigma'}{2\srt} \Big[ -(\rho_+^\od - \rho_-^\od)  \Big(\sigma'(t) + \frac{x_3-\sigma(t)}{2(t+\Lambda)}\Big) + \big(m_{3+}^\od - m_{3-}^\od \big) \Big],
	\end{aligned}
\end{equation}
and 
\begin{equation}\label{f-2-od}
	\begin{aligned}
		\fv_2^\od & = \frac{\Theta_\sigma'}{2\srt} \Big\{ - \big(\mv_+^\od - \mv_-^\od\big) \Big(\sigma'(t) + \frac{x_3-\sigma(t)}{2(t+\Lambda)}\Big)   \\
		& \qquad\qquad\qquad + (u_{3+} \mv_+)^\od - (u_{3-} \mv_-)^\od - \mu \p_3 \big(\uv_+^\od - \uv_-^\od\big) \\
		& \qquad\qquad\qquad + \big[ \big(p(\rho_+)\big)^\od - \big(p(\rho_-)\big)^\od - (\mu+\lambda) \p_3 \big(u_{3+}^\od - u_{3-}^\od \big)  \big]  \E_3 \Big\}.
	\end{aligned}
\end{equation}

Using the Galilean transformation \cref{Galilean}, one can prove that the periodic solutions $ (\rho_\pm, \uv_\pm) $ of \cref{NS} satisfy the following invariance. 

\begin{Lem}\label{Lem-coincide}
	Let $ (\rhob, \pm\uvb) = (\rhob, \pm\ub_1, \pm\ub_2, 0) $ be the constant states in the vortex sheet \cref{vs} and assume that \cref{zero-ave} holds. Then the periodic solutions to the Cauchy problem \cref{NS}, \cref{ic-per} satisfy that
	\begin{equation}\label{coincide}
			(\rho_+, \uv_+)(x + \uvb t, t) = (\rho_-, \uv_-)(x - \uvb t, t)  + (0,2 \uvb)
		 \qquad \forall x\in\R^3, t\geq 0.
	\end{equation}
	Moreover, the zero modes on the right-hand sides of \cref{f-0-od,f-2-od} satisfy that
	\begin{equation}\label{coin-pm}
		\begin{aligned}
			(\rho_+^\od, \uv_+^\od) & \equiv (\rho_-^\od, \uv_-^\od) + (0, 2\uvb), \\ \mv_+^\od & \equiv \mv_-^\od + 2 \rho_-^\od \uvb, \\
			(u_{3+} \mv_+)^\od & \equiv (u_{3-} \mv_-)^\od + 2 m_{3-}^\od \uvb, \\
			(p(\rho_+))^\od & \equiv (p(\rho_-))^\od.
		\end{aligned}
	\end{equation}
\end{Lem}

\begin{proof}
	By the Galilean invariance of the NS equations, the pairs $ (\rho_+^*, \uv_+^*) $ and $ (\rho_-^*, \uv_-^*) $ defined by
	\begin{align*}
		(\rho^*_+, \uv_+^*)(x,t) := (\rho_+,\uv_+ - \uvb)(x + \uvb t,t), \\
		(\rho^*_-, \uv_-^*)(x,t) := (\rho_-,\uv_-+\uvb)(x -\uvb t,t),
	\end{align*}
	are both periodic solutions to \cref{NS} with the initial data
	\begin{align*}
		(\rho^*_\pm, \uv_\pm^*)(x,t=0) = (\rhob, 0,0,0) + (v_0, \wv_0)(x).
	\end{align*}
	Thus, it follows from the uniqueness (see \cref{Lem-per}) that
	\begin{align*}
		(\rho^*_+, \uv^*_+)(x, t) = (\rho^*_-, \uv^*_-)(x, t) \qquad \forall x\in\R^3, t \geq 0,
	\end{align*}
	which yields \cref{coincide}.
	
	Then the identities in \cref{coin-pm} can be verified one by one as follows. First, it follows from  \cref{Lem-per,coincide} that 
	\begin{align*}
		(\rho^\od_+, \uv_+^\od)(x_3,t) & = \int_{\Torus^2} (\rho_+,\uv_+)(\xp,x_3,t) d\xp \\
		& = \int_{\Torus^2} (\rho_-, \uv_-)(x_1-2\ub_1 t,x_2 - 2\ub_2t, x_3,t) d\xp + (0, 2\uvb) \notag \\
		& = (\rho^\od_-, \uv_-^\od)(x_3,t) + (0, 2\uvb) \qquad\qquad \forall x_3\in\R,t\geq 0. 
	\end{align*}
	Recall that $ \uvb = (\ub_1, \ub_2, 0). $
%	 Thus, it holds that $ (\rho_+^\od, u_{3+}^\od) \equiv (\rho_-^\od, u_{3-}^\od) $.
	Then one has that
	\begin{align*}
		\mv_+^\od(x_3,t)  & = \int_{\Torus^2} (\rho_+ \uv_+) (x,t) d\xp = \int_{\Torus^2} (\rho_- \uv_-) (x-2\uvb t,t) d\xp + 2\rho_-^\od(x_3,t) \uvb \\
		& = \mv_-^\od(x_3,t) + 2\rho_-^\od(x_3,t) \uvb, \\
		(u_{3+} \mv_+)^\od(x_3,t) & = \int_{\Torus^2} (\rho_+ u_{3+} \uv_+)(\xp,x_3,t) d\xp \\
		& = \int_{\Torus^2} (\rho_- u_{3-} \uv_-)(\xp,x_3,t) d\xp + 2\uvb \int_{\Torus^2} (\rho_- u_{3-})(\xp,x_3,t) \xp \\
		& = (u_{3-} \mv_-)^\od(x_3,t) + 2 m_{3-}^\od(x_3,t) \uvb.
	\end{align*}
	It is direct to prove that $ (p(\rho_+))^\od(x_3,t) = (p(\rho_-))^\od(x_3,t). $ The proof is finished.
\end{proof}

\vspace{.1cm}

Using \cref{coin-pm} on \cref{f-0-od,f-2-od}, respectively, one can get that
\begin{equation}\label{f=0}
	\big(f_0^\od, f_{2,3}^\od\big) \equiv 0,
\end{equation}
and
\begin{equation}\label{f-2-i}
	\begin{aligned}
		f_{2,i}^\od	& = \frac{\ub_i \Theta_\sigma'}{\srt}  \Big[ -\rho_-^\od \Big(\sigma'(t) + \frac{x_3- \sigma}{2(t+\Lambda)} \Big) + m_{3-}^\od \Big], \qquad i=1,2.
	\end{aligned}
\end{equation}
Here recall that $ (\rho_-^\od, m_{3-}^\od) \equiv (\rho_+^\od, m_{3+}^\od). $ 
Then integrating \cref{f-2-i} on $ \R $ yields that
\begin{equation}
	\begin{aligned}
		\int_\R f_{2,i}^\od(x_3,t) dx_3 = \ub_i  \big( - \mathfrak{D}(\sigma,t) \sigma'(t) + \mathfrak{N}(\sigma,t)\big), \qquad i=1,2,
	\end{aligned}
\end{equation}
where
\begin{equation}\label{DN}
	\begin{aligned}
		\mathfrak{D}(\sigma,t) & := \frac{1}{\srt} \int_\R \rho_-^\od(x_3,t) \Theta' \Big(\frac{x_3-\sigma}{\srt} \Big) d x_3, \\
		\mathfrak{N}(\sigma,t) & := \frac{1}{\srt} \int_\R \Big[- \frac{x_3-\sigma}{2(t+\Lambda)} \rho_-^\od(x_3,t) + m_{3-}^\od(x_3,t)\Big] \Theta' \Big(\frac{x_3-\sigma}{\srt} \Big) d x_3.
	\end{aligned}
\end{equation}
Thus, if we choose the curve $ \sigma(t) $ to satisfy that
\begin{equation}\label{ode-shift}
	\sigma'(t) = \frac{\mathfrak{N}(\sigma, t)}{\mathfrak{D}(\sigma, t)}, \qquad t>0,
\end{equation}
then it holds that
\begin{equation}\label{zero-mass-f}
	\int_\R f_{2,1}^\od(x_3,t) dx_3 = 0 \andd \int_\R f_{2,2}^\od(x_3,t) dx_3 = 0 \qquad \forall t\geq 0.
\end{equation}
This, together with \cref{f=0,ode-mass}, yields that
%\begin{equation}\label{0mass-f}
%	\int_\Omega f_0(x,t) dx = 0 \andd \int_\Omega \fv_2 (x,t) dx = 0 \qquad \forall t \geq 0.
%\end{equation}
%Thus, one has that
\begin{align}
	\int_\R (\rho^\od-\rhot^\od, \mv^\od-\mvt^\od)(x_3,t) dx_3 & = \int_{\R} (\rho^\od-\rhot^\od, \mv^\od-\mvt^\od)(x_3,0) dx_3 \notag \\
	& = \int_{\R} \big(0, [\Theta(x_3) - \Theta(x_3-\sigma(0)) ] \mvb  \big) dx_3 \notag \\
	& =  (0, 2\sigma(0) \mvb) \qquad \forall t\geq 0. \label{mass-1}
\end{align}
Thus, if the curve $ \sigma(t) $ solves the problem \cref{ode-shift} with the initial data,
\begin{equation}\label{ic-sigma}
	\sigma(0) = 0,
\end{equation}
then one has that
\begin{equation}\label{zero-mass}
	\int_\R (\rho^\od-\rhot^\od, \mv^\od-\mvt^\od)(x_3,t) dx_3 = 0 \qquad \forall t\geq 0.
\end{equation}
The existence and large time behavior of the curve $ \sigma(t) $ can be found in the following lemma.

\begin{Lem}\label{Lem-shift}
	Under the assumptions of \cref{Lem-coincide}, there exists an $ \e_0>0 $ such that if $\Lambda\geq 1$ and $\e:= \norm{v_0, \wv_0}_{H^6(\Torus^3)} \leq \e_0,$ then the problem \cref{ode-shift}, \cref{ic-sigma},
%	\begin{equation}\label{ode-sigma}
%		\begin{cases}
%			\sigma'(t) = \frac{\mathfrak{N}(\sigma,t)}{\mathfrak{D}(\sigma,t)}, \qquad t>0, \\
%			\sigma(0) = \sigma_0,
%		\end{cases}
%	\end{equation}
	admits a unique global solution $ \sigma=\sigma(t) \in C^1(0,+\infty) $, satisfying that 
	\begin{equation}\label{exp-sigma}
		\abs{\sigma'(t)} \lesssim \e e^{-\alpha t}, \qquad \abs{\sigma(t)} \lesssim \e \Lambda^\frac{1}{2} e^{-\alpha t}, \qquad t \geq 0, 
	\end{equation}
	where $ \alpha $ are the constants in \cref{exp-per,profile}, respectively.
\end{Lem}
The proof of \cref{Lem-shift} is similar to that in \cite{XYYindi,Yuan2023n}. For readers' convenience, we also give a detailed proof in \cref{Sec-app1}.

\vspace{.2cm}

Then based on Lemmas \ref{Lem-per} and \ref{Lem-shift}, the desired ansatz is well constructed by \cref{ansatz}. 
If $ \e \leq \e_0 $ with $ \e_0 $ being suitably small, it holds that 
\begin{equation}\label{bdd-rhot}
	\rhob/2 \leq \inf_{\substack{x\in \R^3 \\ t>0}} \rhot(x,t)  \leq \sup_{\substack{x\in \R^3 \\ t>0}} \rhot(x,t) \leq 2 \rhob,
\end{equation}
and
\begin{equation}\label{small-ans}
	\sup_{\substack{t>0}} \norm{\nabla \rhot, \nabla\uvt}_{W^{3,\infty}(\R^3)} \lesssim \e + \Lambda^{-1/2}.
\end{equation}
%Denote 
%\begin{equation}
%	(\rhob^\vs, \mvb^\vs)(x_3,t):= (\rho^\vs, \mv^\vs)(x_3-\sigma(t),t) = \Big(\rhob, \mvb \Theta\Big(\frac{x_3-\sigma(t)}{\srt}\Big)\Big).
%\end{equation}
Moreover, the difference between the background viscous wave \cref{profile-1} and the ansatz \cref{ansatz} satisfies that
\begin{align}
	\norm{(\rhot, \mvt)-(\rho^\vs,\mv^\vs)}_{W^{4,\infty}(\R^3)} & \lesssim \norm{(\rhot, \mvt)(\cdot,t)-(\rho^\vs,\mv^\vs)(\cdot-\sigma,t)}_{W^{4,\infty}(\R^3)} \notag \\
	& \quad + \norm{(\rho^\vs, \mv^\vs)(\cdot-\sigma,t)-(\rho^\vs,\mv^\vs)(\cdot,t)}_{W^{4,\infty}(\R)} \notag \\
	& \lesssim \norm{(v_\pm, \wv_\pm)(\cdot,t)}_{W^{4,\infty}(\Torus^3)} + \frac{\abs{\sigma(t)}}{\srt} \notag \\
	&  \lesssim \e e^{-\alpha t}. \label{exp-diff}
\end{align}
Similarly, one can show that
\begin{equation}\label{exp-diff-u}
	\norm{\uvt - \uv^\vs}_{W^{4,\infty}(\R^3)} \lesssim \e e^{-\alpha t}. 
\end{equation}
Thus, the ansatz constructed in \cref{ansatz} is time-asymptotically equivalent to the viscous wave associated with the vortex sheet in the $ L^\infty(\R^3) $-norm. It remains to study the large time behavior of the perturbation, $ (\rho,\mv) - (\rhot, \mvt). $

\vspace{0.3cm}

\section{Reformulated Problems}\label{Sec-prob}

In the following part of this paper, we always denote $\Lambda$ as the positive constant in \cref{profile} and $\e := \norm{(v_0, \wv_0)}_{H^6(\Torus^3)}$.
With the ansatz \cref{ansatz}, we denote the perturbation as
\begin{equation}\label{phipsi}
	\phi = \rho - \rhot, \quad \psi = (\psi_1, \psi_2, \psi_3) := \mv-\mvt,
\end{equation}
and
\begin{equation}\label{zeta}
	\zeta = (\zeta_1, \zeta_2, \zeta_3) := \uv - \uvt = \frac{\psi - \uvt \phi}{\rho}.
\end{equation}
It follows from \cref{NS,equ-t} that
\begin{equation}\label{equ-phipsi}
	\begin{cases}
		\p_t \phi + \dv \psi = - f_0, \\
		\p_t \psi + \dv \big(\frac{\mv\otimes\mv}{\rho} - \frac{\mvt\otimes\mvt}{\rhot} \big) + \nabla (p(\rho) - p(\rhot)) \\
		\qquad - \mu \lap \big(\frac{\mv}{\rho} - \frac{\mvt}{\rhot}\big) - (\mu+\lambda) \nabla \dv \big(\frac{\mv}{\rho} - \frac{\mvt}{\rhot}\big) = - \gv.
	\end{cases}
\end{equation}
Due to \cref{ic,ic-per,ansatz,ic-sigma}, the perturbation $ (\phi,\psi) $ satisfies initially
\begin{equation}\label{ic-phipsi}
	(\phi,\psi)(x,0) = (\phi_0, \psi_0)(x) \equiv  0, \qquad x\in \R^3.
\end{equation}
On the other hand, the equivalent system for $ (\phi, \zeta) $ takes the form,
\begin{equation}\label{equ-phizeta}
	\begin{cases}
		\p_t  \phi + \rho \dv \zeta + \uv \cdot \nabla \phi + \dv \uvt \phi + \nabla \rhot \cdot \zeta =-f_0,  \\
		\rho \p_t \zeta + \rho \uv \cdot \nabla \zeta + \nabla \big(p(\rho) -  p(\rhot)\big)  + \rho \zeta \cdot \nabla \uvt + \phi (\p_t \uvt + \uvt\cdot \nabla \uvt)  \\
		\qquad - \mu \lap \zeta - (\mu+\lambda) \nabla \dv \zeta = - \gv + f_0 \uvt,
	\end{cases}
\end{equation}
and \cref{ic-phipsi} implies that
\begin{equation}\label{ic-pertur}
	(\phi,\zeta)(x,0) = (\phi_0,\zeta_0)(x) \equiv 0,\qquad x\in \R^3.
\end{equation}
\vspace{.1cm}
For any $ T>0, $ the solution space for \cref{equ-phizeta} can be taken as
\begin{align*}
	\mathbb{B}(0,T) := \Big\{ (\phi, \zeta): & \ (\phi,\zeta) \text{ is periodic in } \xp=(x_1,x_2) \in \Torus^2, \\
	& \ (\phi,\zeta) \in C(0,T; H^3(\Omega)), \\
	& \ \nabla\phi \in L^2(0,T; H^2(\Omega)), \nabla\zeta \in L^2(0,T; H^3(\Omega)) \Big\}.
\end{align*}

\begin{Thm}\label{Thm-pert}
	Under the assumptions of \cref{Thm}, there exist $ \Lambda_0\geq 1 $ and $ \e_0>0 $ such that if 
	$\Lambda \geq \Lambda_0 $ and $\e \Lambda^{\frac{1}{4}} \leq \e_0, $
	then the Cauchy problem, \cref{equ-phizeta}, \cref{ic-pertur} admits a unique solution $(\phi,\zeta) \in \mathbf{B}(0,+\infty)$, satisfying that
	\begin{equation}\label{behavior-inf}
		\begin{aligned}
			\norm{(\phi^\od, \zeta^\od)(\cdot,t)}_{L^\infty(\R)} & \leq C (\e \Lambda^{\frac{1}{4}})^\frac{1}{2} (t+1)^{-\frac{3}{4}}, \\
			\norm{(\phi^\md, \zeta^\md)(\cdot,t)}_{L^\infty(\R^3)} & \leq C (\e \Lambda^{\frac{1}{4}})^\frac{1}{2} e^{-\alpha_0 t},
%			\norm{(\nabla \phi, \nabla \zeta)(\cdot,t)}_{L^\infty(\R^3)} & \leq C \e^{\frac{1}{2}} (t+1)^{-1},\\
		\end{aligned}
	\end{equation}
	where $ C>0 $ and $\alpha_0\in (0,\alpha)$ are some generic constants.
\end{Thm}

%\begin{Rem}\label{Rem-opt}
%	The rates in \cref{behavior-inf} are optimal. In fact, for the 1D heat function $ \p_t u = \p_3^2 u $ with the initial data $ u(x_3,0) = u_0(x_3) \in L^2(\R). $ Then the solution $ u $ satisfies that
%	
%\end{Rem}

\vspace{.1cm}
We are going to use the $ L^2 $-energy method to prove \cref{Thm-pert}.
We first integrate \cref{equ-phipsi} with respect to $ \xp \in \Torus^2 $ to get the zero-mode system,
\begin{equation}\label{equ-pert-od}
	\begin{cases}
		\p_t \phi^\od + \p_3 \psi_3^\od = 0, \\
		\p_t \psi^\od + \p_3 \big[ \big(\frac{m_3 \mv}{\rho} - \frac{\mt_3 \mvt}{\rhot} \big)^\od \big] + \p_3 \big[ \left(p(\rho) - p(\rhot)\right)^\od \big] \E_3 \\
		\qquad - \mu \p_3^2 \big[ \big(\frac{\mv}{\rho} - \frac{\mvt}{\rhot} \big)^\od \big] - (\mu+\lambda) \p_3^2 \big[ \big(\frac{m_3}{\rho} - \frac{\mt_3}{\rhot} \big)^\od \big] \E_3 = - \gv^\od,
	\end{cases}
\end{equation}
where we have used \cref{f=0}, and it follows from \cref{ic-phipsi} that
\begin{equation}\label{ic-od}
	(\phi^\od,\psi^\od)(x_3,0) = 0, \qquad x_3 \in \R.
\end{equation}
In the linearization of \cref{equ-pert-od}, there appears a bad term, $ \p_3 \uv^\vs \psi_3^\od $, whose coefficient decays at slow rates, which makes it difficult to control the $ L^2 $-estimate. 
Nevertheless, such difficulty can be overcome by considering the integrated system of \cref{equ-pert-od} instead of itself. 
Thanks to \cref{zero-mass}, it is plausible to write $ (\phi^\od, \psi^\od) = \p_3 (\Phi,\Psi), $ where
\begin{equation}\label{Phi-Psi}
	(\Phi, \Psi)(x_3,t) := \int_{-\infty}^{x_3} (\phi^\od, \psi^\od)(y_3,t) dy_3, \qquad x_3 \in\R, t\geq 0.
\end{equation} 
%This anti-derivative technique is motivated by \cite{Yuan2023n} which studies the shock stability for \cref{NS}. 
To make sense of the integrated system, we first make the following a priori assumptions that, for any fixed $T>0,$ assume that $(\phi,\zeta)\in \mathbb{B}(0,T)$ is a solution to \cref{equ-phizeta}, and the associated anti-derivative variable \cref{Phi-Psi} satisfies that
\begin{equation}\label{apr-assum}
	(\Phi,\Psi) \in C(0,T;L^2(\R)).
\end{equation}

\vspace{.1cm}

\textit{Anti-derivative argument.}
%Using \cref{f=0,zero-mass-f}, one has that
%\begin{equation}\label{zero-mass-g}
%	\int_\R \gv^\od(x_3,t) dx_3 = 0 \qquad \forall t\geq 0.
%\end{equation}
With the a priori assumptions, one can integrate \cref{equ-pert-od,ic-od} with respect to $ x_3 $ from $ -\infty $ to $ x_3 $ to get that
\begin{equation}\label{equ-PhiPsi}
	\begin{cases}
		\p_t \Phi + \p_3 \Psi_3 = 0, \\
		\p_t \Psi + \uv^\vs \p_3 \Psi_3 + p'(\rhob) \p_3 \Phi \E_3 - \frac{\mu}{\rhob} \p_3 \big( \p_3 \Psi -  \uv^\vs \p_3 \Phi \big) - \frac{\mu+\lambda}{\rhob} \p_3^2 \Psi_3 \E_3 \\
		\qquad\qquad\quad = -\Gv + \Nv,
	\end{cases}
\end{equation}
and 
\begin{equation}\label{ic-anti}
	(\Phi,\Psi)(x_3,0) = 0, \qquad x_3 \in \R,
\end{equation}
where $ \Gv $ denotes the anti-derivative variable of $ \gv^\od, $ i.e.
\begin{equation}\label{G}
	\Gv(x_3,t) := \int_{-\infty}^{x_3} \gv^\od(y_3,t) dy_3,
\end{equation}
and $ \Nv $ is given by
\begin{equation}\label{Nv}
	\Nv = -\Qv_1^\od - \Qv_2^\od + \p_3 \big(\Qv_3^\od + \Qv_4^\od \big),
\end{equation}
with
\begin{equation}\label{Q}
	\begin{aligned}
		\Qv_1 & := \frac{m_3 \mv}{\rho} - \frac{\mt_3 \mvt}{\rhot} - \frac{\mvt}{\rhot} \psi_3 - \frac{\mt_3}{\rhot} \psi + \frac{\mt_3\mvt}{\rhot^2} \phi  + \big( p(\rho) - p(\rhot)-p'(\rhot) \phi \big) \E_3, \\
		\Qv_2 & := (\uvt - \uv^\vs) \psi_3 + \ut_3 \psi - \ut_3 \uvt \phi + \big(p'(\rhot) - p'(\rhob)\big) \phi \E_3, \\
		\Qv_3 & := \mu \Big(\frac{\mv}{\rho} - \frac{\mvt}{\rhot}-\frac{1}{\rhot} \psi + \frac{\mvt}{\rhot^2} \phi \Big) + (\mu+\lambda)  \Big(\frac{m_3}{\rho} - \frac{\mt_3}{\rhot} - \frac{1}{\rhot} \psi_3 + \frac{\mt_3}{\rhot^2} \phi \Big) \E_3, \\
		\Qv_4 & := \mu \Big[ \Big(\frac{1}{\rhot} - \frac{1}{\rhob}\Big) \psi - \Big(\frac{\uvt}{\rhot} - \frac{\uv^\vs}{\rhob}\Big) \phi \Big] + (\mu+\lambda) \Big[ \Big(\frac{1}{\rhot} - \frac{1}{\rhob}\Big) \psi_3 + \frac{\ut_3}{\rhot} \phi \Big] \E_3.
	\end{aligned}
\end{equation}

\vspace{.1cm}

One can prove that the errors of the ansatz, i.e. the source terms $f_0, \gv$ and $\Gv$ in \cref{equ-phizeta,equ-PhiPsi}, decay exponentially fast in time; see the proof in \cref{Sec-app2}.

\begin{Lem}\label{Lem-F}
	Suppose that the hypotheses of \cref{Thm} hold true.
	Then there exist constants $\Lambda_0\geq 1$ and $ \e_0>0, $ such that if $\Lambda\geq \Lambda_0$ and $ \e \leq \e_0, $ then the source terms $ f_0, \gv $ and $ \Gv $ in \cref{equ-phizeta,equ-PhiPsi} satisfy that 
	\begin{align}
		\norm{\Gv}_{H^3(\R)} + \norm{f_0}_{H^3(\Omega)} + \norm{\gv}_{H^2(\Omega)} \lesssim \e \Lambda^{\frac{1}{4}} e^{-\alpha t}, \qquad t\geq 0, \label{exp-f-g}
	\end{align}
	where $\Lambda$ and $ \alpha $ are the constants in \cref{exp-per,profile}, respectively.
\end{Lem}

\vspace{.1cm}

Now we show the local existence theorem and the a priori estimates.

\begin{Thm}[Local existence]\label{Thm-local}
	Suppose that the hypotheses of \cref{Thm} hold true, and the initial data $ (\phi_0,\psi_0)(x) $ is periodic in $ \xp $, satisfying that
	\begin{align*}
		(\phi_0, \psi_0)\in {H^3(\Omega)} \andd (\Phi_0, \Psi_0):= \int_{-\infty}^{x_3} \int_{\Torus^2} (\phi_0, \psi_0)(\xp,y_3) d\xp dy_3 \in L^2(\R),
	\end{align*}
	with $ \norm{\phi_0, \psi_0}_{H^3(\Omega)} + \norm{\Phi_0, \Psi_0}_{L^2(\R)} \leq \nu_0. $
	Then there exist positive constants $ \Lambda_0\geq 1 $ and $ \e_0 $ such that if $ \Lambda\geq \Lambda_0 $ and $ \e \Lambda^{\frac{1}{4}} \leq \e_0,$
	then the problem \cref{equ-phizeta} with the initial data 
	\begin{equation}\label{ic-pertur*}
		(\phi,\psi)(x,0) = (\phi_0, \psi_0)(x),
	\end{equation}
	admits a unique solution $ (\phi, \zeta) \in \mathbb{B}(0,T_0) $ for some $ T_0>0, $ which depends on $ \nu_0 $ and $ \e_0. $ Moreover, the anti-derivative variable,
	\begin{align*}
		(\Phi,\Psi)(x_3,t) := \int_{-\infty}^{x_3} \int_{\Torus^2} (\phi,\psi)(\xp,y_3,t) d\xp dy_3,
	\end{align*}
	exists and belongs to $ C(0,T_0; L^2(\R)) $, and it holds that
	\begin{equation}\label{est-local}
		\begin{aligned}
			& \sup_{ t\in(0,T_0)} \big(\norm{\Phi, \Psi}_{L^2(\R)}^2 +  \norm{\phi,\zeta}_{H^3(\Omega)}^2\big)\\
			& \qquad \leq C_0 \big(\norm{ \phi_0, \psi_0}_{H^3(\Omega)}^2 + \norm{ \Phi_0, \Psi_0}_{L^2(\R)}^2 + \e^2 \Lambda^{\frac{1}{2}}\big),
		\end{aligned}
	\end{equation}
	where $ C_0 $ is a positive constant, independent of $ \Lambda, \e $ and $ T_0. $
	%, where $ \psi_1 = m_1- \mt_1 = (\rhot+\phi) (\ut_1 + \zeta_1) - \rhot \ut_1 $.
\end{Thm}
The local existence of $(\phi,\zeta) \in \mathbb{B}(0,T_0) $ for some small $T_0>0$ is standard (see \cite{MN1980} for instance). The existence of the anti-derivative variable $(\Phi,\Psi) \in C(0,T_0; L^2(\R))$ can be derived from a linear parabolic theory. 
Roughly speaking, given the local solution $(\phi,\zeta) \in \mathbb{B}(0,T_0) $ to \cref{equ-phizeta}, $\Psi$ is the unique solution to the linear parabolic problem,
\begin{equation}
	\begin{cases}
		\p_t \Psi - \frac{\mu}{\rhob} \p_3^2 \Psi - \frac{\mu+\lambda}{\rhob} \p_3^2 \Psi_3 \E_3 = \mathbf{J}^\od  + \Gv \in C(0,T_0; H^2(\R)), \qquad x_3\in\R, t>0, \\
		\Psi(x_3,0) = \Psi_0(x_3), \qquad x_3 \in\R.
	\end{cases}
\end{equation}
where 
\begin{align*}
	\mathbf{J} & = \mu \Big( \p_3 \zeta - \frac{\p_3 \psi}{\rhob} \Big) + (\mu+\lambda) \big(\p_3 \zeta_3 - \frac{\p_3\psi_3 }{\rhob}\big) \E_3 - \big( u_3\mv - \ut_3 \mvt \big) -   \big(p(\rho) - p(\rhot)\big) \E_3.
\end{align*}
The existence of $\Phi$ can be proved similarly. We refer to \cite[Section 6]{Yuan2023n} for the detailed arguments.

\vspace{.2cm}

\begin{Prop}[A priori estimates]\label{Prop-apriori}
	Suppose that the hypotheses of \cref{Thm} hold true.
	For any $ T>0, $ suppose that $ ( \phi, \zeta) \in C(0,T;H^3(\Omega)) $ is a solution to the problem \cref{equ-phizeta} and satisfies the a priori assumption \cref{apr-assum}.
	Then there exist a large constant $ \Lambda_0\geq 1 $ and small positive constants $ \e_0 $ and $ \nu_0 $ such that, if 
	\begin{equation*}
		\Lambda \geq \Lambda_0, \quad \e \Lambda^{\frac{1}{4}} \leq \e_0,
	\end{equation*}
	and
	\begin{equation}\label{apriori}
		\nu:= \sup\limits_{t\in(0,T)} \big( \norm{\Phi,\Psi}_{L^2(\R)} + \norm{ \phi,\psi}_{H^3(\Omega)} \big) \leq \nu_0,
	\end{equation}
	then it holds that
	\begin{equation}\label{est-apriori}
		\begin{aligned}
			& \sup_{t\in(0,T)} \big( \norm{\Phi,\Psi}^2_{L^2(\R)} + \norm{ \phi,\zeta}_{H^3(\Omega)}^2 \big)+ \int_0^T \Big(\norm{\nabla\phi}_{H^2(\Omega)}^2 +\norm{\nabla \zeta}_{H^3(\Omega)}^2\Big) dt \\
			& \qquad \lesssim \norm{(\Phi, \Psi)(\cdot,0)}_{L^2(\R)}^2 + \norm{ (\phi, \psi)(\cdot,0) }_{H^3(\Omega)}^2 + \e \Lambda^{\frac{1}{4}}.
		\end{aligned}
	\end{equation}
\end{Prop}
The proof of \cref{Prop-apriori} will be given in \cref{Sec-apriori}. 

\vspace{.2cm}

At the end of this section, we present some lemmas to be used later.

\vspace{0.2cm}

Note that the perturbations $ \phi, \psi $ and $ \zeta $ in \cref{phipsi,zeta} are not always of zero averages with respect to $x_1$ or $x_2$. 
As stated in \cite{HY2020},
%For example, for the perturbation $ \phi = \rho - \rhot $, 
%the property that
%\begin{equation}\label{dim}
%	\int_{\Torus} \phi(x,t) dx_1  \equiv 0 \andd \int_{\Torus} \phi(x,t) dx_2 \equiv 0 \quad \text{for all } \ t\geq 0,
%\end{equation} 
these perturbations often do not satisfy the classical 3D Gagliardo-Nirenberg (G-N) inequalities, since the 1D and 2D cases cannot be excluded. 
To overcome this difficulty, we use the following G-N type inequality on the domain $ \Omega = \Torus^2 \times \R $.

\begin{Lem}[\cite{HY2020}, Theorem 1.4 \& Lemma 3.3]\label{Lem-GN}
	Assume that $ u(x) $ is a bounded function that is periodic in $ \xp = (x_1, x_2) $. Then there exists a decomposition $ u(x) = \sum\limits_{k=1}^{3} u^{(k)}(x) $ such that 
	\begin{itemize}
		\item[i) ] \begin{equation}\label{decomp}
			\begin{aligned}
				u^{(1)} & \equiv \int_{\Torus^2} u(\xp,x_3) d\xp = u^{\od}(x_3),\\
				u^{(2)} + u^{(3)} & \equiv u(x)- \int_{\Torus^2} u(\xp,x_3) d\xp = u^\md(x);
			\end{aligned}
		\end{equation}
		
		\item[ii)] if $\nabla^l u $ belongs to $ L^p(\Omega)$ with an order $l\geq 0 $ and $p\in[1,+\infty]$, then each $u^{(k)}$ satisfies that
		\begin{equation}\label{ineq-cauchy}
			\norm{\nabla^l u^{(k)}}_{L^p(\Omega)} \lesssim \norm{\nabla^l u}_{L^p(\Omega)};
		\end{equation} 
		
		\item [iii)] each $ u^{(k)} $ satisfies the classical k-dimensional Gagliardo-Nirenberg inequality, i.e. 
		\begin{equation}\label{G-N-type-1}
			\norm{\nabla^j u^{(k)}}_{L^p(\Omega)} \lesssim \norm{\nabla^{m} u^{(k)}}_{L^{r}(\Omega)}^{\theta_k} \norm{u^{(k)}}_{L^{q}(\Omega)}^{1-\theta_k},
		\end{equation}
		where $ 0\leq j< m $ is any integer, $ 1\leq p \leq +\infty $ is any number and $\theta_k \in [\frac{j}{m}, 1)$ satisfies 
		$$ \frac{1}{p} = \frac{j}{k} + \Big(\frac{1}{r}-\frac{m}{k}\Big) \theta_k + \frac{1}{q}\big(1-\theta_k\big). $$
	\end{itemize}
%	Assume that $ u(x) $ is periodic in $ \xp = (x_1, x_2) $. Then there exists a decomposition $ u(x) = \sum\limits_{k=1}^{3} u^{(k)}(x) $ such that if $ u\in L^q(\Omega) $ and $ \nabla^m u \in L^r(\Omega), $ with $ 1\leq q,r\leq +\infty $ and $ m\geq 1 $, then
%	\begin{equation}\label{decomp}
%		\begin{aligned}
%		 	u^{(1)} & = u^{\od}(x_3) = \int_{\Torus^2} u(\xp,x_3) d\xp,\\
%		 	u^{(2)} + u^{(3)} & = u^\md(x) = u(x)-u^\od(x_3);
%		\end{aligned}
%	\end{equation}
%	and each $ u^{(k)} $ satisfies the k-dimensional Gagliardo-Nirenberg inequality, namely,
%	\begin{equation}\label{G-N-type-1}
%		\norm{\nabla^j u^{(k)}}_{L^p(\Omega)} \lesssim \norm{\nabla^{m} u^{(k)}}_{L^{r}(\Omega)}^{\theta_k} \norm{u^{(k)}}_{L^{q}(\Omega)}^{1-\theta_k},
%	\end{equation}
%	where $ 0\leq j< m $ is any integer, $ 1\leq p \leq +\infty $ is any number and $\theta_k \in [\frac{j}{m}, 1]$ satisfies 
%	$$ \frac{1}{p} = \frac{j}{k} + \Big(\frac{1}{r}-\frac{m}{k}\Big) \theta_k + \frac{1}{q}\big(1-\theta_k\big). $$
%	Moreover, if $\nabla^l u \in L^p(\Omega)$ for some $l\geq 0$ and $p\in[1,+\infty],$ then each $u^{(k)}$ satisfies that
%	\begin{equation}\label{ineq-cauchy}
%		\norm{\nabla^l u^{(k)}}_{L^p(\Omega)} \lesssim \norm{\nabla^l u}_{L^p(\Omega)}.
%	\end{equation}
\end{Lem}

\vspace{.1cm}

\begin{Lem}\label{Lem-poin}
	Suppose that $ u(x) $ belongs to $ W^{1,p}(\Omega) $ with $ p\in [1,+\infty]. $ Then its non-zero mode $ u^\md $ satisfies that
%	\begin{equation*}
%		\begin{aligned}
%			\norm{u^\od}_{L^p(\R)} + \norm{u^\md}_{L^p(\Omega)} & \lesssim \norm{u}_{L^p(\Omega)}, \\
%			\norm{u^\md}_{L^p(\Omega)} & \lesssim \norm{u}_{L^p(\Omega)} + \norm{u^\od}_{L^p(\R)} \lesssim \norm{u}_{L^p(\Omega)},
%		\end{aligned}
%	\end{equation*}
%	and 
	\begin{equation}\label{poincare}
		\norm{u^\md}_{L^p(\Omega)} \lesssim \norm{\nabla_{\xp} u^\md}_{L^p(\Omega)} \lesssim \norm{\nabla u}_{L^p(\Omega)}.
	\end{equation}
	%	where the implicit constants depend on $ p. $
\end{Lem}
\begin{proof}
	Note that $u^\md$ has zero average on the transverse torus $ \Torus^2 $.
	Then \cref{poincare} follows directly from the Poincar\'{e} inequality.
\end{proof}
%
%\vspace{.1cm}

%Let $ \kappa=\kappa(x_3,t) $ denote the heat kernel,
%\begin{equation}\label{kernel}
%	\kappa = \frac{1}{2\sqrt{\pi(t+1)}} \exp\Big\{ - \frac{x_3^2}{4(t+1)}\Big\}.
%\end{equation}
At last, we present an inequality in \cite{HLM}, which is helpful in obtaining the optimal decay rate.
Set
\begin{equation}\label{omega}
	\kappa(x_3,t) :=  \frac{1}{\srt} \exp \left\{- \frac{\rhob x_3^2}{32\mu(t+\Lambda)}\right\},
\end{equation}
and $ \Kc(x_3,t)=\int_{-\infty}^{x_3} \kappa(y_3,t) dy_3, $
which satisfies that $\p_t \Kc = \frac{8\mu}{\rhob} \p_3 \kappa.$

\begin{Lem}[\cite{HLM}, Lemma 1]\label{Lem-HLM}
	For $0<T\leq+\infty$, suppose that $ h= h(x_3,t)$ satisfies
	\begin{equation*}
		h \in L^2(0,T; H^1(\R)) \andd \p_t h \in L^2(0,T;H^{-1}(\R)).
	\end{equation*}
	Then it holds that
	\begin{equation}\label{ineq-HLM}
		\begin{aligned}
			& \int_{\R} \abs{h}^2 \kappa^2 dx_3 + \frac{\rhob}{8\mu} \frac{d}{dt} \Big( \int_\R \abs{h}^2 \Kc^2 dx_3 \Big) - \frac{\rhob}{4\mu} \int_\R \p_t h \cdot h \Kc^2 dx_3 \lesssim \norm{\p_3 h}_{L^2(\R)}^2.
		\end{aligned}
	\end{equation}
\end{Lem}

\vspace{.3cm}

\section{A Priori Estimates}\label{Sec-apriori}

To prove \cref{Prop-apriori}, we first establish the $ H^2 $-estimates of the anti-derivatives and the $ H^1 $-estimates of the non-zero modes of the perturbations in \cref{Sec-od,Sec-md}, respectively. 
Afterwards, we go back to the original system \cref{equ-phizeta} in \cref{Sec-orn} to deal with the higher-order derivatives to complete the proof.

\vspace{0.2cm}

Denote
\begin{equation}\label{nu}
	\begin{aligned}
		\nu := \sup\limits_{t\in(0,T)} \Big\{ & \norm{\Phi,\Psi}_{L^2(\R)} +  \norm{ \phi,\zeta}_{H^3(\Omega)} \Big\}.
	\end{aligned}
\end{equation}
It follows from \cref{Lem-GN} that for any $h\in H^2(\Omega),$ 
\begin{equation}\label{Sob-inf}
	\norm{h}_{L^\infty(\Omega)} \lesssim \norm{\nabla h}^{\frac{1}{2}}_{L^2(\Omega)} \norm{h}^{\frac{1}{2}}_{L^2(\Omega)} + \norm{\nabla^2 h}_{L^2(\Omega)}^{\frac{1}{2}} \norm{h}_{L^2(\Omega)}^{\frac{1}{2}} + \norm{\nabla^2 h}^{\frac{3}{4}}_{L^2(\Omega)} \norm{ h}^{\frac{1}{4}}_{L^2(\Omega)}.
\end{equation}
Then one has that
\begin{align}
	\sup_{t\in(0,T)} \norm{ \phi,\zeta}_{W^{1,\infty}(\Omega)} 
%	& \lesssim \sup_{t\in(0,T)} \Big\{ \norm{\nabla( \phi,\zeta)}^{\frac{1}{2}}_{H^1(\Omega)} \norm{ \phi,\zeta}^{\frac{1}{2}}_{H^1(\Omega)} + \norm{\nabla ( \phi,\zeta)}_{H^1(\Omega)} \notag \\
%	& \qquad\qquad\quad + \norm{\nabla^2 (\phi,\zeta)}^{\frac{3}{4}}_{H^1(\Omega)} \norm{ \phi,\zeta}^{\frac{1}{4}}_{H^1(\Omega)} \Big\} \notag \\
	& \lesssim \sup_{t\in(0,T)} \norm{ \phi,\zeta}_{H^3(\Omega)} \lesssim \nu. \label{small-phizeta}
\end{align}
Similarly, the perturbation of momentum, $ \psi = \rho \zeta + \phi \uvt $, satisfies that
\begin{equation*}
	\sup_{t\in(0,T)} \norm{\psi}_{W^{1,\infty}(\Omega)} \lesssim \sup_{t\in(0,T)} \norm{\psi}_{H^3(\Omega)} \lesssim \sup_{t\in(0,T)} \norm{ \phi,\zeta}_{H^3(\Omega)} \lesssim \nu.
\end{equation*}
It follows from \cref{bdd-rhot,small-ans,small-phizeta} that if the positive constants $ \Lambda^{-1}, $ $ \e $ and $ \nu $ are small, then
\begin{align*}
	\rhob/4 \leq \inf_{\substack{x\in\Omega \\ t\in(0,T)}} \rho(x,t)  \leq \sup_{\substack{x\in\Omega \\ t\in(0,T)}} \rho(x,t) \leq 4 \rhob,
\end{align*}
and
\begin{equation}\label{small-uv}
	\sup_{t\in(0,T)} \norm{\nabla\rho(x,t), \nabla\uv(x,t)}_{L^{\infty}(\Omega)} \lesssim  \Lambda^{-\frac{1}{2}} + \e + \nu.
\end{equation}
Besides, it follows from the Sobolev inequality and \cref{Lem-GN} that 
\begin{equation}\label{nu-od}
	\begin{aligned}
		\sup_{t\in(0,T)} \norm{\phi^\od,\psi^\od}_{W^{2, \infty}(\R)} \lesssim \sup_{t\in(0,T)} \norm{\phi^{\od},\psi^{\od}}_{H^3(\R)} \lesssim \sup_{t\in(0,T)} \norm{\phi,\psi}_{H^3(\Omega)} \lesssim \nu,
	\end{aligned}
\end{equation}
and 
\begin{equation}\label{nu-md}
	\begin{aligned}
		\sup_{t\in(0,T)} \norm{\phi^\md, \zeta^\md}_{W^{1,\infty}(\Omega)} \lesssim \sup_{ t\in(0,T)} \norm{\phi^\md, \zeta^\md}_{H^3(\Omega)}  \lesssim \sup_{t\in(0,T)} \norm{\phi, \zeta}_{H^3(\Omega)} \lesssim \nu.
	\end{aligned}
\end{equation}
Then it holds that
\begin{align}
	\sup_{t\in(0,T)} \norm{\Phi,\Psi}_{W^{3, \infty}(\R)} & \lesssim \sup_{t\in(0,T)} \norm{\Phi,\Psi}_{H^4(\R)} \notag \\
	& \lesssim \sup_{t\in(0,T)} \big( \norm{\Phi,\Psi}_{L^2(\R)} + \norm{\phi^{\od},\psi^{\od}}_{H^3(\R)} \big) \notag \\
	& \lesssim \nu. \label{small-PhiPsi}
\end{align}
Moreover, note that $\kappa \lesssim (t+\Lambda)^{-\frac{1}{2}}$ and it follows from \cref{Lem-theta} that 
\begin{equation}\label{bdd-theta-1}
	\abs{\p_3 \theta}^2 \lesssim \Lambda^{-\frac{1}{2}} \kappa, \quad \abs{\p_3^2 \theta} \lesssim \kappa^2, \quad \abs{\p_t \p_3 \theta} \lesssim \Lambda^{-\frac{1}{2}} \kappa^2, 
\end{equation}
and meanwhile,
\begin{equation}\label{bdd-theta}
	\begin{aligned}
		\abs{\p_3^j \theta} \lesssim (t+\Lambda)^{-\frac{j-1}{2}} \kappa \qquad \text{for } j=1,2,\cdots.
	\end{aligned}
\end{equation}

\vspace{.2cm}

To carry out the $L^2$-estimate for \cref{equ-PhiPsi}, we may encounter a difficulty due to the large amplitude of the vortex sheet that the linear terms such as $\uv^\vs \p_3 \Psi_3$ and $\frac{\mu}{\rhob} \uv^\vs \p_3^2 \Phi $ are difficult to get controlled. Nevertheless, we can overcome this difficulty by introducing a new variable.

\textit{New variable}.
Set
\begin{equation}\label{Zv}
	\Zv = (Z_1,Z_2,Z_3)(x_3,t) := \Psi - \uv^\vs \Phi.
\end{equation}
%where we have used \cref{profile}.
It holds that
\begin{equation}\label{small-Zv}
	\begin{aligned}
		\sup_{t\in(0,T)} \norm{\Zv}_{W^{3,\infty}(\R)} \lesssim \sup_{t\in(0,T)} \norm{\Zv}_{H^4(\R)} \lesssim \sup_{t\in(0,T)} \norm{\Phi, \Psi}_{H^4(\R)} \lesssim \nu.
	\end{aligned}
\end{equation}
Note that $ \uv^\vs = \uvb \theta $ and $ \p_t \theta = \frac{\mu}{\rhob} \p_3^2 \theta. $ Then it follows from \cref{equ-PhiPsi} that
\begin{equation}\label{equ-Zv}
	\begin{cases}
		\p_t \Phi + \p_3 Z_3 = 0, \\
		\p_t \Zv + \big( p'(\rhob) \E_3 - \frac{\mu}{\rhob} \p_3 \theta \uvb \big) \p_3 \Phi - \frac{\mu}{\rhob} \p_3^2 \Zv -\frac{\mu+\lambda}{\rhob} \p_3^2 Z_3 \E_3 = - \Gv + \Nv.
	\end{cases}
\end{equation}
With the aid of \cref{Zv}, we can successfully get rid of the bad linear terms in \cref{equ-PhiPsi}\textsubscript{2}; and the linear term, $-\frac{\mu}{\rhob} \p_3 \theta \uvb \p_3 \Phi$  in \cref{equ-Zv}\textsubscript{2}, has smallness due to the largeness of the constant $\Lambda$ in \cref{profile}. 
%Moreover, the introduction of \cref{Zv} plays an important role in achieving the optimal decay rate of the zero mode, $(\phi^\od,\zeta^\od)$.
Recall the notation $``_\perp"$ in \cref{perp} and the fact that $ \uvb = (\uvb_\perp, 0) $.
For a later use, we also decompose \cref{equ-Zv}$ _2 $ into
\begin{equation}\label{equ-Zp}
	\p_t \Zv_\perp  - \frac{\mu}{\rhob} \p_3 \theta \p_3 \Phi \uvb_\perp- \frac{\mu}{\rhob} \p_3^2 \Zv_\perp = - \Gv_\perp + \Nv_\perp,
\end{equation}
and
\begin{equation}\label{equ-Z3}
	\p_t Z_3 + p'(\rhob) \p_3 \Phi - \frac{\mut}{\rhob} \p_3^2 Z_3 = - G_3 + N_3 \qquad \text{with } \ \mut = 2\mu +\lambda.
\end{equation}

\vspace{.1cm}

Then we give two lemmas that will be needed in the energy estimates later. The readers may skip them at first reading.

\begin{Lem}\label{Lem-Q}
	Under the assumptions of \cref{Prop-apriori}, there exist a large constant $ \Lambda_0\geq 1 $ and small positive constants $ \e_0 $ and $ \nu_0 $ such that, if $ \Lambda \geq \Lambda_0 $, $ \e \leq \e_0 $ and $ \nu \leq \nu_0, $ then the remainders \cref{Q} satisfy that
	\begin{equation}\label{est-Q}
		\begin{aligned}
			\abs{\Qv_1 } + \abs{\Qv_3 } & \lesssim \abs{\phi}^2 + \abs{\psi}^2, \\
			\abs{\p_3 \Qv_1 } + \abs{\p_3 \Qv_3 } & \lesssim \big[\e e^{-\alpha t}+(t+\Lambda)^{-\frac{1}{2}}\big] \big(\abs{\phi}^2 + \abs{\psi}^2 \big) \\
			& \quad  + \big(\abs{\p_3\phi} + \abs{\p_3\psi} \big) \left(\abs{\phi} + \abs{\psi}\right), \\
			\abs{\p_3^k \Qv_2 } +\abs{\p_3^k \Qv_4 }  & \lesssim \sum_{j=0}^{k}  \e e^{-\alpha t} \big(\abs{\p_3^j \phi} + \abs{\p_3^j \psi} \big), \qquad k=0,1.
		\end{aligned}
	\end{equation}
\end{Lem}

%\begin{proof}
%	The estimates \cref{est-Q}\textsubscript{1} and \cref{est-Q}\textsubscript{3} follows from \cref{Q,exp-diff,exp-diff-u} directly. It suffices to estimate $\p_3 \Qv_1$ and $\p_3 \Qv_3$.
%	
%	By \cref{Q}, direct calculations yield that
%	\begin{align*}
%		\p_3 \Qv_1 	& = \p_3 \mt_3 \Big( \frac{\mv}{\rho}- \frac{\mvt}{\rhot} - \frac{1}{\rhot} \psi + \frac{\mt}{\rhot^2} \phi \Big) + \p_3 \mvt \Big( \frac{m_3}{\rho} - \frac{\mt_3}{\rhot} - \frac{\psi_3}{\rhot} + \frac{\mt_3}{\rhot^2} \phi \Big) \\
%		& \quad - \p_3 \rhot \Big( \frac{m_3 \mv}{\rho^2} - \frac{\mt_3 \mvt}{\rhot^2} - \frac{\mvt}{\rhot^2} \psi_3 - \frac{\mt_3}{\rhot^2} \psi + \frac{2\mt_3 \mvt}{\rhot^3} \phi \Big) \\
%		& \quad + \p_3 \psi_3 \zeta + \zeta_3  \p_3 \psi - \p_3 \phi (u_3 \uv - \ut_3 \uvt) \\
%		& \quad + \big(p'(\rho)-p'(\rhot)\big) \p_3 \phi \E_3 + \big( p'(\rho) - p'(\rhot) - p''(\rhot) \phi \big) \p_3 \rhot \E_3.
%	\end{align*}
%	Using \cref{exp-diff,exp-diff-u}, one can get that
%	\begin{align*}
%		\abs{\p_3 \Qv_1} & \lesssim \e e^{-\alpha t} \big(  \abs{\phi}^2 + \abs{\psi}^2 \big) + (t+\Lambda)^{-\frac{1}{2}} \big(  \abs{\phi}^2 + \abs{\psi_3}^2 \big)  \\
%		& \quad  + \abs{\zeta} \abs{\p_3 \psi} +  \abs{\p_3 \phi} \abs{\zeta} + \abs{\phi} \abs{\p_3\phi},
%	\end{align*}
%	which yields the estimate of $\p_3 \Qv_1$ in \cref{est-Q}. The estimate of $\p_3 \Qv_3$ in \cref{est-Q} can be proved similarly.
%\end{proof}

\vspace{0.1cm}

\begin{Lem}\label{Lem-rel}
	Under the assumptions of \cref{Prop-apriori}, there exist a large constant $ \Lambda_0\geq 1 $ and small positive constants $ \e_0 $ and $ \nu_0 $ such that, if $ \Lambda\geq \Lambda_0 $, $ \e \leq \e_0 $ and $ \nu \leq \nu_0, $ then
%	\begin{itemize}
%		\item[1) ] 
	\begin{align}
		\pm \norm{\zeta^\od}_{L^2(\R)} 
		& \lesssim \pm \norm{\p_3 \Zv}_{L^2(\R)} + \norm{\Phi \kappa}_{L^2(\R)} + \nu \norm{\nabla \zeta^\md}_{L^2(\Omega)} + \e \nu e^{-\alpha t}, \label{rel-0} \\
		\pm \norm{\p_3 \zeta^\od}_{L^2(\R)} 
		& \lesssim \pm \norm{\p_3^2 \Zv}_{L^2(\R)} + \nu \norm{\p_3^2 \Phi}_{L^2(\R)} +  (t+\Lambda)^{-\frac{1}{2}} \norm{\p_3 \Phi, \Phi \kappa}_{L^2(\R)} \notag \\
		& \quad  + \nu \norm{\nabla \phi^\md, \nabla \zeta^\md}_{L^2(\Omega)} + \e \nu e^{-\alpha t}, \label{rel-1} \\
		\pm \norm{\p_3^2 \zeta^\od}_{L^2(\R)} & \lesssim \pm \norm{\p_3^3 \Zv}_{L^2(\R)} +  (t+\Lambda)^{-\frac{1}{2}} \norm{\p_3^2 \Phi}_{L^2(\R)} \notag \\
		& \quad  + (t+\Lambda)^{-1} \norm{\p_3 \Phi, \Phi \kappa}_{L^2(\R)} + \nu \norm{\nabla^2 \phi,\nabla^2 \zeta^\md }_{L^2(\Omega)} + \e \nu e^{-\alpha t}, \label{rel-2}
	\end{align}
	and 
	\begin{align}
		\pm \norm{\nabla \psi^\md}_{L^2(\Omega)} & \lesssim \pm \norm{\nabla \zeta^\md}_{L^2(\Omega)} + \norm{\nabla \phi^\md}_{L^2(\Omega)} + \e \nu e^{-\alpha t}, \label{rel-3} \\
		\pm \norm{\nabla^3 \psi}_{L^2(\Omega)}
		& \lesssim \pm \norm{\nabla^3 \zeta}_{L^2(\Omega)} + \norm{\nabla^3\phi}_{L^2(\Omega)} + \sum_{j=1}^{2} (t+\Lambda)^{\frac{j-4}{2}} \norm{\p_3^{j} \Phi}_{L^2(\R)} \notag \\
		& \quad  + (\Lambda^{-\frac{1}{2}} +\nu) \norm{\nabla^2\phi}_{L^2(\Omega)} + \e \nu e^{-\alpha t}. \label{rel-4}
	\end{align}
\end{Lem}
The proof of \cref{Lem-rel} is given in \cref{Sec-app4}.

\vspace{.2cm}

For convenience, in the following energy estimates, we write $ ``\p_3(\cdots)" $ and $ ``\dv(\cdots)" $ to denote the terms which vanish after integration with respect to $ x_3 \in \R $ and $ x\in\Omega, $ respectively. 
Besides, we use $ c_{i,\flat}, c_{i,\neq} $ and $ c_i $ for $ i=0,1,2,\cdots $ to denote some positive generic constants, which are independent of $ x,t,\e,\Lambda $ and $ \nu. $

\vspace{.2cm}

\subsection{Estimates of Zero-Modes.}\label{Sec-od}

In this section, we establish the $H^2(\R)$-estimates of the anti-derivatives, which is the key part in the a priori estimates.

\begin{Lem}\label{Lem-est1}
	Under the assumptions of \cref{Prop-apriori}, there exist a large constant $ \Lambda_0\geq 1 $ and small positive constants $ \e_0 $ and $ \nu_0 $ such that, if $ \Lambda \geq \Lambda_0 $, $ \e \leq \e_0 $ and $ \nu \leq \nu_0, $ then
	\begin{equation}\label{est1}
		\begin{aligned}
			\frac{d}{dt} \Ac_{0,\od}^{(1)} + \frac{\mu}{\rhob}\norm{\p_3 \Zv}_{L^2(\R)}^2 & \lesssim (\Lambda^{-\frac{1}{2}}+\nu)  \norm{\p_3 \Phi, \Phi\kappa, \Zv \kappa}_{L^2(\R)}^2\\
			& \quad + \nu  \norm{\nabla \phi^\md, \nabla \zeta^\md}_{L^2(\Omega)}^2 + \e \nu \Lambda^{\frac{1}{4}} e^{-\alpha t}.
		\end{aligned}
	\end{equation}
	where 
	\begin{equation}\label{Ac0-1}
		\Ac_{0,\od}^{(1)} := p'(\rhob) \norm{\Phi}_{L^2(\R)}^2 + \norm{\Zv}_{L^2(\R)}^2 + \frac{2\mu}{\rhob p'(\rhob)} \int_\R \p_3 \theta \uvb_\perp \cdot \Zv_\perp Z_3 dx_3.
	\end{equation}
\end{Lem}

\begin{proof}
%	The proof will decomposed into two parts
%		\begin{equation}\label{est0}
%		\begin{aligned}
%			& \sup_{t\in(0,T)} \norm{\Phi,\Psi}_{L^2(\R)}^2 + \int_0^T  \norm{\p_3 \Phi,\p_3 \Psi}_{L^2(\R)}^2  dt\\
%			 \lesssim& \norm{\Phi_0,\Psi_0}_{L^2(\R)}^2+ \delta\int_0^T\int_\R(\Phi^2+\Psi_1^2+\Psi_3^2)\kappa^2\ dx_3dt +\int_0^T\nu\|(\phi,\psi)\|_{H^1(\Omega)}^2dt+ \nu\varepsilon.
%		\end{aligned}
%	\end{equation}
%and
%		\begin{equation}\label{estomega}
%	\begin{aligned}
%				\int_0^t\int_\R(p'\Phi^2&+\Psi^2)\kappa^2dx_3dt\lesssim\sup_{t\in(0,T)}{\big\Vert(\Phi,\Psi)(\cdot,t)\big\Vert}_{L^2(\R)}^2+ \nu\varepsilon+\int_0^T\nu\|(\phi,\psi)\|_{H^1(\Omega)}^2dt.
%	\end{aligned}
%\end{equation}

Multiplying $ p'(\rhob)\Phi $ and $ \Zv $ on \cref{equ-Zv}$ _1 $ and \cref{equ-Zv}$ _2 $, respectively, and adding the resulting two equations together, one has that
\begin{equation}\label{eq1}
	\begin{aligned}
		& \p_t \Big[ \frac{1}{2} \big( p'(\rhob) \Phi^2 + \abs{\Zv}^2\big) \Big] + \frac{\mu}{\rhob} \abs{\p_3 \Zv}^2 + \frac{\mu+\lambda}{\rhob} \abs{\p_3 Z_3}^2 \\
		& \qquad = \p_3 (\cdots) + \underbrace{(-\Gv + \Nv) \cdot \Zv}_{I_1} + \underbrace{ \frac{\mu}{\rhob} \p_3 \theta \p_3 \Phi \uvb_\perp \cdot \Zv_\perp}_{I_2}.
	\end{aligned}
\end{equation}
It follows from Lemmas \ref{Lem-F} and \ref{Lem-Q} that
\begin{equation}\label{tool-G}
\int_\R \abs{\Gv} \abs{\Zv} dx_3 \lesssim \norm{\Gv}_{L^2(\R)} \norm{\Zv}_{L^2(\R)} \lesssim \e \nu \Lambda^{\frac{1}{4}} e^{-\alpha t},
\end{equation}
and
\begin{equation}\label{tool-Q}
\begin{aligned}
	\abs{\int_\R \Nv \cdot \Zv dx_3} & \lesssim \abs{\int_\R (\Qv_1^\od + \Qv_2^\od)  \cdot \Zv dx_3} + \abs{\int_\R (\Qv_3^\od + \Qv_4^\od) \cdot \p_3 \Zv dx_3} \\
	& \lesssim \norm{\Zv}_{H^1(\R)} \norm{\Qv_2, \Qv_4}_{L^2(\Omega)} + \norm{\Zv}_{W^{1,\infty}(\R)} \norm{\Qv_1, \Qv_3}_{L^1(\Omega)} \\
	& \lesssim \nu \norm{\phi, \psi}_{L^2(\Omega)}^2 + \e \nu^2 e^{-\alpha t}.
	%	& \lesssim \e \nu^2 e^{-\alpha t} + \nu 	\norm{\p_3 \Phi,\p_3 \Zv}_{L^2(\R)}^2 + \delta\nu \int_\R \Phi^2 \kappa^2 dx_3 + \nu \norm{\nabla \phi, \nabla \psi}_{L^2(\Omega)}^2.
\end{aligned}
\end{equation}
Note that $ \phi^\od = \p_3 \Phi $ and $ \psi^\od = \p_3 \Psi. $ Then it follows from \cref{Lem-poin} that
\begin{align}
\norm{\phi}_{L^2(\Omega)}^2 & = \norm{\phi^\od}_{L^2(\R)}^2 + \norm{\phi^\md}_{L^2(\Omega)}^2 \lesssim \norm{\p_3 \Phi}_{L^2(\R)}^2 + \norm{\nabla \phi^\md}_{L^2(\Omega)}^2, \label{ineq-poin1}
\end{align}
and similarly,
\begin{align}
	\norm{\psi}_{L^2(\Omega)}^2 & \lesssim \norm{\p_3 \Psi}_{L^2(\R)}^2 + \norm{\nabla\psi^\md}_{L^2(\Omega)}^2. \label{ineq-poin2}
\end{align}
By \cref{Zv,bdd-theta}, one can verify easily that
\begin{equation}\label{ineq-1}
\begin{aligned}
	%		 \abs{\norm{\Psi}_{L^2(\R)}^2 - \norm{\Zv}_{L^2(\R)}^2 } & \lesssim  \delta \norm{\Phi}_{L^2(\R)}^2, \\
	\norm{\p_3 \Psi}_{L^2(\R)}^2  &  \lesssim \norm{\p_3 \Zv,\p_3 \Phi, \Phi \kappa}_{L^2(\R)}^2.
\end{aligned}
\end{equation}
Then it follows from \cref{ineq-poin1,ineq-poin2,ineq-1,rel-3} that
\begin{equation}\label{ineq-poin3}
\begin{aligned}
	\norm{\phi, \psi}_{L^2(\Omega)}^2 & \lesssim \norm{\p_3 \Phi,\p_3 \Zv,\Phi\kappa}_{L^2(\R)}^2 + \norm{\nabla \phi^\md, \nabla \zeta^\md}_{L^2(\Omega)}^2 + \e \nu e^{-\alpha t}.
\end{aligned}
\end{equation}
This, together with \cref{rel-3}, yields that
\begin{equation}\label{est-I1}
	\int_\R I_1 dx_3 \lesssim \e \nu \Lambda^{\frac{1}{4}} e^{-\alpha t} + \nu \norm{\p_3 \Phi,\p_3 \Zv,\Phi\kappa}_{L^2(\R)}^2 + \nu \norm{\nabla \phi^\md, \nabla \zeta^\md}_{L^2(\Omega)}^2.
\end{equation}

Now we estimate $I_2$. 
Using \cref{equ-Z3}, one can get that
\begin{align*}
	\frac{\rhob p'(\rhob)}{\mu} I_2 & = \p_3 \theta \uvb_\perp \cdot \Zv_\perp \big(-G_3+N_3 \big) - \p_3 \theta \uvb_\perp \cdot \Zv_\perp \p_t Z_3 +  \frac{\mut}{\rhob} \p_3 \theta \uvb_\perp \cdot \Zv_\perp \p_3^2 Z_3 \\
	& := I_{2,1} + I_{2,2} + I_{2,3}.
\end{align*}
First, similar to the estimate of $I_1$ and noting that $\abs{\p_3 \theta} \lesssim \Lambda^{-\frac{1}{2}}$, one can prove that
\begin{equation}\label{est-I2-1}
	\int_\R I_{2,1} dx_3 \lesssim \e \nu e^{-\alpha t} + \nu \norm{\p_3 \Phi,\p_3 \Zv,\Phi\kappa}_{L^2(\R)}^2 + \nu \norm{\nabla \phi^\md, \nabla \zeta^\md}_{L^2(\Omega)}^2.
\end{equation}
For $I_{2,2}$, it holds that
\begin{equation}\label{I2-2}
	I_{2,2} = \p_t \big( - \p_3 \theta \uvb_\perp \cdot \Zv_\perp Z_3\big) + \p_3\p_t \theta \uvb_\perp \cdot \Zv_\perp Z_3 + \underbrace{ \p_3 \theta \uvb_\perp \cdot \p_t \Zv_\perp Z_3}_{I_{2,4}}.
\end{equation}
Using \cref{equ-Zp}, one has that
\begin{equation}\label{I2-4}
\begin{aligned}
	I_{2,4} & =  \p_3 \theta Z_3 \uvb_\perp \cdot (-\Gv_\perp + \Nv_\perp )  + \p_3 \Big( \frac{\mu}{\rhob} \p_3 \theta Z_3 \uvb_\perp \cdot \p_3 \Zv_\perp\Big) \\
	& \quad - \frac{\mu}{\rhob} \p_3 \theta \p_3 Z_3 \uvb_\perp \cdot \p_3 \Zv_\perp - \frac{\mu}{\rhob} \p_3^2 \theta Z_3 \uvb_\perp \cdot \p_3 \Zv_\perp  + \frac{\mu}{\rhob} \abs{\uvb}^2 \abs{\p_3 \theta}^2 Z_3 \p_3 \Phi.
\end{aligned}	
\end{equation}
Collecting \cref{I2-2,I2-4}, one can use \cref{bdd-theta-1}, \cref{bdd-theta} and a similar proof of \cref{est-I1} to get that
\begin{equation}\label{est-I2-2}
	\begin{aligned}
		& \int_\R I_{2,2} dx_3 + \frac{d}{dt} \Big(\int_\R \p_3\theta \uvb_\perp \cdot \Zv_\perp Z_3 dx_3 \Big) \\
		& \qquad \lesssim  \big(\Lambda^{-\frac{1}{2}}+\nu\big) \norm{\p_3 \Phi,\p_3 \Zv,\Phi \kappa, \Zv \kappa}_{L^2(\R)}^2 + \e \nu e^{-\alpha t} + \nu \norm{\nabla \phi^\md, \nabla \zeta^\md}_{L^2(\Omega)}^2.
	\end{aligned}
\end{equation}
%Using \cref{bdd-theta} and the fact that $ \abs{\uvb} \lesssim \delta $, it holds that
%\begin{equation}\label{ineq-I4}
%	\int_\R \abs{\frac{\mu}{\rhob} \p_3 \theta \p_3 \Phi \uvb_\perp \cdot \Zv_\perp} dx_3 \lesssim \delta \int_\R  \abs{\Zv_\perp}^2 \kappa^2 dx_3 + \delta \norm{\p_3 \Phi}_{L^2(\R)}^2.
%\end{equation} 
%\begin{equation}\label{ineq-I4}
%	\int_\R \abs{\frac{\mu}{\rhob} \p_3 \theta \p_3 \Phi \uvb_\perp \cdot \Zv_\perp} dx_3 \lesssim \norm{\Zv_\perp \kappa}_{L^2(\R)} \norm{\p_3 \Phi}_{L^2(\R)}.
%\end{equation}
Moreover, it holds that
\begin{align*}
	I_{2,3} & = \p_3 \Big( \frac{\mut}{\rhob} \p_3 \theta \uvb_\perp \cdot \Zv_\perp \p_3 Z_3 \Big) - \frac{\mut}{\rhob} \p_3^2 \theta \uvb_\perp \cdot \Zv_\perp \p_3 Z_3 - \frac{\mut}{\rhob} \p_3 \theta \uvb_\perp \cdot \p_3 \Zv_\perp \p_3 Z_3,
\end{align*}
which, together with \cref{bdd-theta-1}, yields that
\begin{equation}\label{est-I2-3}
	\int_\R I_{2,3} dx_3 \lesssim \Lambda^{-\frac{1}{2}} \norm{\p_3 \Zv, \Zv_\perp \kappa}_{L^2(\R)}^2.
\end{equation}
Then collecting \cref{est-I1,est-I2-1,est-I2-2,est-I2-3}, one can obtain \cref{est1} if $ \Lambda^{-1}>0 $ and $ \nu>0 $ are suitably small.

\end{proof}

\vspace{0.3cm}

Set $ \mathfrak{K}(x_3,t) := \int_{-\infty}^{x_3} \kappa^2(y,t) dy. $ It holds that
\begin{equation}\label{A-bdd}
	\norm{\mathfrak{K}}_{L^\infty(\R)} \lesssim (t+\Lambda)^{-\frac{1}{2}} \andd \norm{\p_t \mathfrak{K}}_{L^\infty(\R)} \lesssim (t+\Lambda)^{-\frac{3}{2}}, \qquad t\geq 0.
\end{equation}
	
\begin{Lem}\label{Lem-weight}
	Under the assumptions of \cref{Prop-apriori}, there exist a large constant $ \Lambda_0\geq 1 $ and small positive constants $ \e_0 $ and $ \nu_0 $ such that, if $ \Lambda \geq \Lambda_0 $, $ \e \leq \e_0 $ and $ \nu \leq \nu_0, $ then
	\begin{equation}\label{est-weight}
		\begin{aligned}
			& \frac{d}{dt} \Ac_{0,\od}^{(2)} + p'(\rhob) \norm{\Phi \kappa}_{L^2(\R)}^2 + \norm{\Zv \kappa}_{L^2(\R)}^2 \\
			& \qquad \lesssim \norm{\p_3 \Zv}_{L^2(\R)}^2 + (\Lambda^{-\frac{1}{2}} + \nu)  \norm{\p_3 \Phi}_{L^2(\R)}^2 + \nu  \norm{\nabla \phi^\md, \nabla \zeta^\md}_{L^2(\Omega)}^2 \\
			& \qquad \quad  +  \e \nu \Lambda^{\frac{1}{4}} e^{-\alpha t} + (t+\Lambda)^{-\frac{3}{2}} \norm{\Phi,\Zv}_{L^2(\R)}^2,
		\end{aligned}
	\end{equation}
	where 
	\begin{equation}\label{Ac0-2}
		\Ac_{0,\od}^{(2)} = \frac{\rhob}{8\mu}\int_\R \abs{\Zv_\perp}^2 \Kc^2 dx_3 - 4 \int_\R Z_3\Phi \mathfrak{K} dx_3 - \frac{1}{4p'(\rhob)} \int_\R \p_3 \theta \uvb_\perp \cdot \Zv_\perp Z_3 \Kc^2 dx_3.
	\end{equation}

\end{Lem}

\begin{proof}
	
Multiplying  $\Phi \mathfrak{K}$ on \cref{equ-Z3} and integrating the resulting equation on $ \R $, one has that
\begin{equation}\label{ineq-wa}
	\begin{aligned}
		& \int_\R \frac{1}{2} \big(p'(\rhob) \Phi^2 +  Z_3^2 \big) \kappa^2 dx_3 - \frac{d}{dt} \Big( \int_\R Z_3\Phi \mathfrak{K} dx_3 \Big) \\
		& \qquad =  - \int_\R Z_3 \Phi \p_t \mathfrak{K} dx_3 + \frac{\mut}{\rhob} \int_\R \p_3 Z_3 \big(\p_3 \Phi \mathfrak{K} + \Phi \kappa^2 \big) dx_3 + \int_\R (G_3 - N_3) \Phi \mathfrak{K} dx_3,
	\end{aligned}
\end{equation}
here we have used the fact that $ Z_3 \p_t \Phi \mathfrak{K} = - Z_3 \p_3 Z_3 \mathfrak{K} = -\p_3 \big(\frac{1}{2} Z_3^2 \mathfrak{K} \big) + \frac{1}{2} Z_3^2 \kappa^2. $ 
Similar to the proof of \cref{est-I1}, one can show that
\begin{equation}\label{ineq-wb}
	\begin{aligned}
		&  \int_\R (G_3 - N_3) \Phi \mathfrak{K} dx_3 \lesssim \e \nu e^{-\alpha t}  + \nu \norm{\p_3 \Phi,\p_3 \Zv, \Phi \kappa}_{L^2(\R)}^2 + \nu \norm{\nabla \phi^\md, \nabla \zeta^\md}_{L^2(\Omega)}^2.
	\end{aligned}
\end{equation}
Thus, by \cref{A-bdd,ineq-wb}, integrating \cref{ineq-wa} over $ (0,T) $ yields that
\begin{equation}\label{ineq-wc}
	\begin{aligned}
		& p'(\rhob) \norm{\Phi \kappa}_{L^2(\R)}^2 + \norm{ Z_3\kappa}_{L^2(\R)}^2 - \frac{d}{dt} \Big(2 \int_\R Z_3\Phi \mathfrak{K} dx_3 \Big) \\
		& \qquad \lesssim (t+\Lambda)^{-\frac{3}{2}} \norm{\Phi, Z_3}_{L^2(\R)}^2  + (\Lambda^{-\frac{1}{2}} + \nu) \norm{\p_3 \Phi,\p_3 \Zv}_{L^2(\R)}^2 \\
		& \qquad \quad + \nu \norm{\nabla \phi^\md, \nabla \zeta^\md}_{L^2(\Omega)}^2 +  \e \nu e^{-\alpha t}. 
	\end{aligned}
\end{equation}

To estimate $ \int_0^T \norm{\Zv_\perp \kappa}_{L^2(\R)}^2 dt, $ it firstly follows from \cref{Lem-HLM} with $ h = \Zv_\perp $ that
\begin{equation}\label{ineq-wd}
	\begin{aligned}
		\norm{\Zv_\perp \kappa}_{L^2(\R)}^2 +  \frac{d}{dt} \Big(\frac{\rhob}{8\mu} \int_\R \abs{\Zv_\perp}^2 \Kc^2 dx_3 \Big) - \frac{\rhob}{4\mu} \int_{\R} \p_t \Zv_\perp\cdot \Zv_\perp \Kc^2\ dx_3 & \lesssim  \norm{\p_3\Zv_\perp}_{L^2(\R)}^2.
	\end{aligned}
\end{equation}
Using \cref{equ-Zp}, one can get that
\begin{align*}
	\frac{\rhob}{4\mu} \int_{\R} \p_t\Zv_\perp \cdot \Zv_\perp \Kc^2\ dx_3 & = \frac{1}{4} \underbrace{ \int_\R \p_3 \theta \p_3 \Phi \uvb_\perp \cdot \Zv_\perp \Kc^2 dx_3}_{I_3} - \frac{1}{2} \int_\R \p_3 \Zv_\perp \cdot \Zv_\perp \Kc \kappa dx_3 \\
	& \quad - \frac{1}{4} \int_\R \abs{\p_3 \Zv_\perp}^2 \Kc^2 dx_3 - \frac{\rhob}{4\mu} \int_{\R} \big(\Gv_\perp - \Nv_\perp \big) \cdot \Zv_\perp \Kc^2\ dx_3.
\end{align*}
Similar to the estimates of $I_{2,i}, i=1,2,3,4$ in the proof of \cref{Lem-est1}, one can use \cref{equ-Z3} to prove that
\begin{align*}
	& I_3 +  \frac{d}{dt} \Big( \frac{1}{ p'(\rhob)}\int_\R \p_3 \theta \uvb_\perp \cdot \Zv_\perp Z_3 \Kc^2 dx_3 \Big) \\
	& \qquad \lesssim \big(\Lambda^{-\frac{1}{2}} + \nu\big) \norm{\Phi\kappa, \Zv\kappa,\p_3 \Phi,\p_3 \Zv}_{L^2(\R)}^2  + \nu \norm{\nabla \phi^\md, \nabla \zeta^\md}_{L^2(\Omega)}^2 + \e \nu e^{-\alpha t}.
\end{align*}
This, together with \cref{bdd-theta}, \cref{ineq-wd} and a similar proof of \cref{est-I1}, yields that
\begin{equation}\label{ineq-we}
	\begin{aligned}
		& \norm{\Zv_\perp \kappa}_{L^2(\R)}^2  +  \frac{d}{dt} \Big(\frac{\rhob}{8\mu} \int_\R \abs{\Zv_\perp}^2 \Kc^2 dx_3 - \frac{1}{4p'(\rhob)} \int_\R \p_3 \theta \uvb_\perp \cdot \Zv_\perp Z_3 \Kc^2 dx_3 \Big) \\
		& \qquad \lesssim \norm{\p_3 \Zv}_{L^2(\R)}^2 + \big(\Lambda^{-\frac{1}{2}} + \nu\big) \norm{\p_3 \Phi, \Phi \kappa, Z_3\kappa}_{L^2(\R)}^2 \\
		& \qquad \quad + \nu \norm{ \nabla \phi^\md, \nabla \zeta^\md}_{L^2(\Omega)}^2 + \e \nu \Lambda^{\frac{1}{4}} e^{-\alpha t}.
	\end{aligned}
\end{equation}
Collecting \cref{ineq-wc,ineq-we}, one can finish the proof.

\end{proof}

\vspace{.3cm}

\begin{Lem}\label{Lem-est2}
		Under the assumptions of \cref{Prop-apriori}, there exist a large constant $ \Lambda_0\geq 1 $ and small positive constants $ \e_0 $ and $ \nu_0 $ such that, if $ \Lambda \geq \Lambda_0 $, $ \e \leq \e_0 $ and $ \nu \leq \nu_0, $ then it holds that
		\begin{equation}\label{est2}
			\begin{aligned}
				& \frac{d}{dt} \Ac^{(3)}_{0,\od} + p'(\rhob) \norm{\p_3 \Phi}_{L^2(\R)}^2 \\ 
				& \qquad \lesssim \norm{\p_3 \Zv}_{L^2(\R)}^2 + \nu \norm{\Phi \kappa}_{L^2(\R)}^2 + \nu \norm{\nabla \phi^\md, \nabla \zeta^\md}_{L^2(\Omega)}^2 + \e \nu \Lambda^{\frac{1}{4}} e^{-\alpha t},
			\end{aligned}
		\end{equation}
	where 
	\begin{equation}\label{Ac0-3}
		\Ac_{0,\od}^{(3)} := \frac{\mut}{\rhob} \norm{\p_3 \Phi}_{L^2(\R)}^2 + 2 \int_\R Z_3 \p_3 \Phi dx_3.
	\end{equation}
\end{Lem}
\begin{proof}

	Multiplying $ \frac{\mut}{\rhob} \p_3 \Phi $ and $ \p_3 \Phi $ on $ \p_3 $\cref{equ-Zv}$ _1 $ and \cref{equ-Z3}, respectively, and adding the resulting two equations together, one has that
	\begin{equation}\label{eq2}
		\p_t \Big( \frac{\mut}{2\rhob} \abs{\p_3 \Phi}^2 + Z_3 \p_3 \Phi \Big) + p'(\rhob) \abs{\p_3 \Phi}^2 = \p_3 (\cdots) + \abs{\p_3 Z_3}^2 + (-G_3 + N_3) \p_3 \Phi.
	\end{equation}
	Similar to the proof of \cref{est-I1}, with the a priori assumption that $\norm{\p_3\Phi}_{H^1(\R)} + \norm{\p_3 \Phi}_{W^{1,\infty}(\R)} \lesssim \nu, $ one can prove that
	\begin{align*}
		& \int_\R (-G_3 + N_3) \p_3 \Phi dx_3 \lesssim \nu \norm{\p_3 \Phi,\p_3 \Zv,\Phi\kappa}_{L^2(\R)}^2 + \nu \norm{\nabla \phi^\md, \nabla \zeta^\md}_{L^2(\Omega)}^2 + \e \nu \Lambda^{\frac{1}{4}} e^{-\alpha t}.
	\end{align*}
 	This, together with \cref{eq2}, can yield \cref{est2}.
\end{proof}

Collecting Lemmas \ref{Lem-est1}--\ref{Lem-est2}, one has that
\begin{equation}\label{est2b}
	\begin{aligned}
		\frac{d}{dt } \Ac_{0,\od} + c_{0,\od} \Bc_{0,\od} \lesssim \nu \norm{\nabla \phi^\md, \nabla \zeta^\md}_{L^2(\Omega)}^2 + \e \nu \Lambda^{\frac{1}{4}} e^{-\alpha t} + (t+\Lambda)^{-\frac{3}{2}} \norm{\Phi,\Zv}_{L^2(\R)}^2,
	\end{aligned}
\end{equation}
where $ \Ac_{0,\od} $ denotes a linear combination of $ \Ac_{0,\od}^{(1)} $, $ \Ac_{0,\od}^{(2)} $ and $ \Ac_{0,\od}^{(3)}, $ satisfying that 
\begin{equation}\label{Ac0}
	\Ac_{0,\od} \sim \norm{\Phi}_{H^1(\R)}^2 + \norm{\Zv}_{L^2(\R)}^2;
\end{equation}
and 
\begin{equation}\label{Bc0}
	\Bc_{0,\od} := \norm{\p_3 \Phi, \p_3 \Zv, \Phi\kappa, \Zv\kappa}_{L^2(\R)}^2.
\end{equation}
Then integrating \cref{est2b} with respect to $ t $ yields that
\begin{equation}\label{est2a}
	\begin{aligned}
		& \sup_{t\in(0,T)} \big(\norm{\Phi}_{H^1(\R)}^2 + \norm{\Zv}_{L^2(\R)}^2 \big) + \int_0^T  \norm{\p_3 \Phi, \p_3 \Zv,\Phi \kappa, \Zv \kappa}_{L^2(\R)}^2   dt \\
		& \qquad \lesssim \norm{\Phi_0}_{H^1(\R)}^2 + \norm{\Psi_0}_{L^2(\R)}^2 + \nu \int_0^T \norm{\nabla\phi^\md, \nabla \zeta^\md}_{L^2(\Omega)}^2 dt + \e \nu \Lambda^{\frac{1}{4}}.
	\end{aligned}
\end{equation}
Here we have used the fact that 
$\int_0^T (t+\Lambda)^{-\frac{3}{2}} \norm{\Phi,\Zv}_{L^2(\R)}^2 dt \lesssim \Lambda^{-\frac{1}{4}} \sup\limits_{ t\in(0,T)}\norm{\Phi,\Zv}_{L^2(\R)}^2. $

\vspace{.3cm} 

\begin{Lem}\label{Lem-est4}
	Under the assumptions of Proposition 3.2, there exist a large constant $ \Lambda_0\geq 1 $ and small positive constants $ \e_0 $ and $ \nu_0 $ such that, if $\Lambda\geq \Lambda_0, \varepsilon \leqslant \varepsilon_0$ and $\nu \leqslant \nu_0$, then it holds that
	\begin{equation}\label{est-4}
		\begin{aligned}
			\frac{d}{dt} \Ac_{1,\flat} + c_{1,\flat} \Bc_{1,\flat} & \lesssim (t+\Lambda)^{-1} \norm{\p_3 \Phi, \Phi \kappa}_{L^2(\R)}^2 + \nu \norm{\p_3^2 \Phi, \p_3^3 \Zv}_{L^2(\R)}^2 \\
			& \quad + \nu \norm{\nabla \phi^\md,\nabla \zeta^\md}_{L^2(\Omega)}^2 + \e\nu \Lambda^{\frac{1}{4}} e^{-\alpha t}.
		\end{aligned}
	\end{equation}
	where
	\begin{equation}\label{AB-1}
		\begin{aligned}
			\Ac_{1,\flat} = p'(\rhob) \norm{\p_3 \Phi}_{L^2(\R)}^2 + \norm{\p_3 \Zv}_{L^2(\R)}^2 \andd
			\Bc_{1,\flat} =  \norm{\p_3^2 \Zv}_{L^2(\R)}^2.
		\end{aligned}
	\end{equation}

\end{Lem}

\begin{proof}

Multiplying $-p'(\rhob)\p^2_3\Phi$ and  $-\p_3^2 \mathbf{Z}$ on \cref{equ-Zv}\textsubscript{1} and \cref{equ-Zv}\textsubscript{2}, respectively and adding up the resulting two equations, one can get that
\begin{equation}\label{eq4}
	\begin{aligned}
	& \frac{d}{dt} \left(\frac{p'(\rhob)}{2}\norm{\p_3\Phi}_{L^2(\R)}^2 + \frac{1}{2}\norm{\p_3\mathbf{Z}}_{L^2(\R)}^2\right) + \frac{\mu}{\rhob}\norm{\p_3^2\mathbf{Z}}_{L^2(\R)}^2 + \frac{\mu+\lambda}{\rhob}\norm{\p_3^2 Z_{3}}_{L^2(\R)}^2 \\
	& \qquad = \underbrace{-\frac{\mu}{\rhob} \int_{\R} \p_3\theta \p_3 \Phi \bar{\uv}\cdot\p_3^2\mathbf{Z} dx_3}_{I_4} + \underbrace{\int_{\R} \left(\Gv -\Nv \right)\cdot\p_3^2\Zv dx_3}_{I_5}.
\end{aligned}
\end{equation}
One can show that
\begin{equation}\label{I4}
	\abs{I_4} \leq C (t+\Lambda)^{-1} \norm{\p_3 \Phi}_{L^2(\R)}^2 + \frac{\mu}{4\rhob}  \norm{\p_3^2\Zv}_{L^2(\R)}^2.
\end{equation}
It follows from Lemmas \ref{Lem-F} and \ref{Lem-Q} and \cref{small-Zv} that
\begin{align}
	I_5 & = \int_{\R} \left(\Gv + \Qv_1^\od + \Qv_2^\od\right)\cdot\p_3^2\Zv dx_3 + \int_\R \big(\Qv_3^\od + \Qv_4^\od \big) \p_3^3 \Zv dx_3 \notag \\
	& \lesssim \e \nu \Lambda^{\frac{1}{4}} e^{-\alpha t} + \norm{\phi,\psi}_{L^4(\Omega)}^2 \norm{\p_3^2 \Zv}_{H^1(\R)}. \label{I5}
\end{align} 
%This, together with \cref{poincare}, yields that
%\begin{align*}
%	\norm{\phi^\md,\psi^\md}_{L^4(\Omega)} \lesssim \norm{\phi^\md,\psi^\md}_{H^1(\Omega)} \lesssim \norm{\nabla \phi^\md, \nabla \psi^\md}_{L^2(\Omega)}.
%\end{align*}
Using \cref{Lem-GN}, it holds that
\begin{equation}\label{GN-2}
	\norm{\phi^\od}_{L^4(\R)} = \norm{\p_3 \Phi}_{L^4(\R)} \lesssim \norm{\p_3^2 \Phi}_{L^2(\R)}^{\frac{5}{8}} \norm{\Phi}_{L^2(\R)}^{\frac{3}{8}} \lesssim \nu^{\frac{1}{2}} \norm{\p_3^2 \Phi}_{L^2(\R)}^{\frac{1}{2}},
\end{equation}
and similarly,
\begin{equation}\label{GN-3}
	\norm{\psi^\od}_{L^4(\R)} \lesssim \nu^{\frac{1}{2}} \norm{\p_3^2 \Psi}_{L^2(\R)}^{\frac{1}{2}}.
\end{equation}
On the other hand, it follows from \cref{Lem-GN} that
\begin{equation}\label{Sobolev-1}
	\begin{aligned}
		\norm{h}_{L^4(\Omega)} & \lesssim \sum_{k=1}^{3} \norm{\nabla h}_{L^2(\Omega)}^{\frac{k}{4}} \norm{h}_{L^2(\Omega)}^{1-\frac{k}{4}} \lesssim \norm{h}_{H^1(\Omega)} \qquad \forall h\in H^1(\Omega).
	\end{aligned}
\end{equation}
Then using \cref{Sobolev-1,GN-2,GN-3}, one can get that
\begin{align}
	\norm{\phi,\psi}_{L^4(\Omega)}^2 & \lesssim \norm{\phi^\od,\psi^\od}_{L^4(\R)}^2 + \norm{\phi^\md,\psi^\md}_{L^4(\Omega)}^2\notag \\
	& \lesssim \nu \norm{\p_3^2 \Phi, \p_3^2\Psi}_{L^2(\R)} + \nu \norm{\phi^\md,\psi^\md}_{H^1(\Omega)}, \notag
\end{align}
which, together with \cref{poincare,Zv,rel-3}, yields that
\begin{equation}\label{L4}
	\begin{aligned}
		\norm{\phi,\psi}_{L^4(\Omega)}^2 & \lesssim \nu \norm{\p_3^2 \Phi, \p_3^2 \Zv}_{L^2(\R)} + \nu (t+\Lambda)^{-\frac{1}{2}} \norm{\p_3 \Phi, \Phi \kappa}_{L^2(\R)} \\
		& \quad + \nu \norm{\nabla \phi^\md, \nabla \zeta^\md}_{L^2(\Omega)} + \e \nu e^{-\alpha t}.
	\end{aligned}
\end{equation}
Collecting \cref{L4,I4,I5}, one can finish the proof of \cref{est-4}.

\end{proof}

\vspace{.3cm}

\begin{Lem}\label{Lem-est5}
	Under the assumptions of Proposition 3.2, there exist a large constant $ \Lambda_0\geq 1 $ and small positive constants $ \e_0 $ and $ \nu_0 $ such that, if $\Lambda \geq \Lambda_0, \e \leq \varepsilon_0$ and $\nu \leq \nu_0$, then it holds that
	\begin{equation}\label{est-5}
		\begin{aligned}
			\frac{d}{dt} \Ac_{2,\flat} + c_{2,\flat} \Bc_{2,\flat} & \lesssim \Bc_{1,\flat}  + \nu (t+\Lambda)^{-1} \norm{\p_3 \Phi, \Phi \kappa}_{L^2(\R)}^2 +  \nu \norm{\nabla \zeta^\md}_{L^2(\Omega)}^2 \\
			& \quad + \nu \norm{\nabla^2 \phi}_{H^1(\Omega)}^2 + \e \nu \Lambda^{\frac{1}{4}} e^{-\alpha t}.
		\end{aligned}
	\end{equation}
	where $ \Bc_{0,0} $ and $\Bc_{1,\flat}$ are given by \cref{AB-1} and
	\begin{equation}\label{AB-2}
		\Ac_{2,\flat} = \frac{\mut}{2\rhob}\norm{\p_3^2\Phi}_{L^2(\R)}^2 + \int_{\R} \p_3Z_3\p_3^2\Phi dx_3 \andd
		\Bc_{2,\flat} =  \norm{\p_3^2 \Phi}_{L^2(\R)}^2.
	\end{equation}
\end{Lem}

\begin{proof}
	Multiplying $ \frac{\mut}{\rhob} \p_3^2 \Phi $ and $ \p_3^2\Phi $ on $ \p_3^2 $\cref{equ-Zv}$ _1 $ and $\p_3$\cref{equ-Z3}, respectively, and adding up the resulting two equations, one can get that
	\begin{align*}
		&\frac{d}{dt}\left(\frac{\mut}{2\rhob}\norm{\p_3^2\Phi}_{L^2(\R)}^2 + \int_{\R}\p_3Z_3\p_3^2\Phi dx_3\right) + p'(\rhob) \norm{\p_3^2\Phi}_{L^2(\R)}^2 \\
		& \qquad = \norm{\p_3^2 Z_3}_{L^2(\R)}^2 + \underbrace{\int_\R  \left(G_3 - N_3\right) \p_3^3 \Phi dx_3}_{I_6}.
	\end{align*}
	Similar to the estimate of $I_5$ in \cref{eq4}, one can use \cref{Lem-Q,L4} to prove that
	\begin{align*}
		I_6 & = \int_{\R} \left(G_3 + Q_{1,3}^\od + Q_{2,3}^\od\right)\cdot\p_3^3 \Phi dx_3 + \int_\R \big(Q_{3,3}^\od + Q_{4,3}^\od\big) \p_3^4 \Phi dx_3 \\
		& \lesssim  \e \nu \Lambda^{\frac{1}{4}} e^{-\alpha t} + \norm{\phi,\psi}_{L^4(\Omega)}^2 \norm{\p_3^3 \Phi}_{H^1(\R)} \\
		& \lesssim \e \nu \Lambda^{\frac{1}{4}} e^{-\alpha t} + \nu \norm{\p_3^2 \Phi}_{H^2(\R)}^2 + \nu \norm{\p_3^2 \Zv}_{L^2(\R)}^2 \\
		& \quad + \nu (t+\Lambda)^{-1} \norm{\p_3 \Phi, \Phi \kappa}_{L^2(\R)}^2 + \nu \norm{\nabla\phi^\md, \nabla\zeta^\md}_{L^2(\Omega)}^2.
	\end{align*}
	Note that $\norm{\p_3^3 \Phi}_{H^1(\R)} \lesssim \norm{\nabla^2 \phi}_{H^1(\Omega)} $ and $ \norm{\nabla\phi^\md}_{L^2(\Omega)} \lesssim \norm{\nabla^2\phi^\md}_{L^2(\Omega)} \lesssim\norm{\nabla^2\phi}_{L^2(\Omega)}. $ Then the proof is finished.
	
\end{proof}

%\bibliographystyle{amsplain}
%\bibliography{bibli}
\begin{Lem}\label{Lem-est6}
	Under the assumptions of Proposition 3.2, there exist a large constant $ \Lambda_0\geq 1 $ and small positive constants $ \e_0 $ and $ \nu_0 $ such that, if $\Lambda\geq \Lambda_0, \varepsilon \leqslant \varepsilon_0$ and $\nu \leqslant \nu_0$, then it holds that
	\begin{equation}\label{est-6}
		\begin{aligned}
			\frac{d}{dt} \Ac_{3,\flat} + c_{3,\flat} \Bc_{3,\flat} & \lesssim  (t+\Lambda)^{-2} \norm{\p_3 \Phi, \Phi \kappa}_{L^2(\R)}^2 + (t+\Lambda)^{-1} \norm{\p_3^2 \Phi,\p_3^2\Zv}_{L^2(\R)}^2 \\
			& \quad + \nu \norm{\nabla^2 \phi}_{H^1(\Omega)}^2 + \nu \norm{\nabla^3 \zeta}_{L^2(\Omega)}^2 + \e \nu \Lambda^{\frac{1}{4}} e^{-\alpha t}.
		\end{aligned}
	\end{equation}
	where $\Bc_{i,\flat}$ for $i=0,1,2$ are given by \cref{AB-1,AB-2} and
		\begin{equation}\label{AB-3}
			\Ac_{3,\flat} = p'(\rhob) \norm{\p_3^2 \Phi}_{L^2(\R)}^2 + \norm{\p_3^2 \Zv}_{L^2(\R)}^2 \andd
		\Bc_{3,\flat} = \norm{\p_3^3 \Zv}_{L^2(\R)}^2.
		\end{equation}
\end{Lem}

\begin{proof}
Multiplying $ -p'(\rhob) \p_3^3 \Phi $ and $ -\p_3^3 \Zv $ on $ \p_3 $\cref{equ-Zv}$ _1 $ and $ \p_3 $\cref{equ-Zv}$ _2 $, respectively and adding up the resulting two equations, one can get that
\begin{align*}
	& \frac{d}{dt} \left(\frac{p'(\rhob)}{2}\norm{\p^2_3\Phi}_{L^2(\R)}^2+\frac{1}{2}\norm{\p^2_3\mathbf{Z}}_{L^2(\R)}^2\right) + \frac{\mu}{\rhob}\norm{\p_3^3\mathbf{Z}}_{L^2(\R)}^2 + \frac{\mu+\lambda}{\rhob}\norm{\p_3^3\mathbf{Z}_{3}}_{L^2(\R)}^2 \\
	& \qquad = \underbrace{-\frac{\mu}{\rhob} \int_{\R} \p_3 \big(\p_3\theta \p_3 \Phi\big) \bar{\uv} \cdot\p_3^3\mathbf{Z} dx_3 }_{I_7} + \underbrace{\int_\R \left( \p_3\Gv - \p_3 \mathbf{N}\right)\cdot \p_3^3\Zv dx_3}_{I_8}.
\end{align*}
Using \cref{est-theta}, one has that
\begin{align}
	I_7 & \leq \frac{\mu}{4\rhob} \norm{\p_3^3 \Zv}_{L^2(\R)}^2 +  C \sum_{j=1}^{2} (t+\Lambda)^{j-3} \norm{\p_3^j \Phi}_{L^2(\R)}^2. \label{I7}
\end{align}

To achieve the optimal rate of $(\phi^\od,\zeta^\od)$, we need to deal with the term $I_8$ carefully. 
Indeed, it follows from Lemmas \ref{Lem-F} and \ref{Lem-Q} that
\begin{align}
	I_8 & = \int_{\R} \left(\p_3 \Gv + \p_3 \Qv_1^\od + \p_3 \Qv_2^\od\right) \cdot \p_3^3 \Zv dx_3 + \int_\R \big(\p_3 \Qv_3^\od + \p_3 \Qv_4^\od \big) \p_3^4 \Zv dx_3 \notag \\
%	& \lesssim \e \nu e^{-\alpha t} + \nu \norm{\p_3^3 \Zv}_{H^1(\R)}^2 + \nu^{-1} \norm{\p_3 \Qv_1, \p_3 \Qv_3}_{L^2(\Omega)}^2 \notag \\
	& \lesssim \e \nu \Lambda^{\frac{1}{4}} e^{-\alpha t} + \norm{\p_3^3 \Zv}_{H^1(\R)} (I_{8,1} + I_{8,2}), \label{I8}
\end{align}
where 
\begin{equation}\label{I8-12}
	I_{8,1} =  (t+\Lambda)^{-\frac{1}{2}} \norm{\phi, \psi}_{L^4(\Omega)}^2 \andd I_{8,2} = \norm{\nabla\phi, \nabla\psi}_{L^4(\Omega)}  \norm{\phi,\psi}_{L^4(\Omega)}.
\end{equation}
First, it follows from \cref{Zv,rel-4} that
\begin{equation}\label{ineq-4}
	\begin{aligned}
		\norm{\p_3^4 \Zv}_{L^2(\R)} & \lesssim (t+\Lambda)^{-\frac{3}{2}} \norm{\p_3 \Phi, \Phi \kappa}_{L^2(\R)} + (t+\Lambda)^{-1} \norm{\p_3^2 \Phi}_{L^2(\R)} \\
		& \quad +  \norm{\nabla^2 \phi}_{H^1(\Omega)} + \norm{\nabla^3 \zeta}_{L^2(\Omega)} + \e \nu e^{-\alpha t}.
	\end{aligned}
\end{equation}
%\begin{equation}\label{ineq-4}
%	\begin{aligned}
%		\norm{\p_3^4 \Zv}_{L^2(\R)} & \lesssim \norm{\nabla^3 \zeta}_{L^2(\Omega)} + \delta (t+1)^{-1} \Bc_{2,\flat}^{\frac{1}{2}} + \delta (t+1)^{-\frac{3}{2}} \Bc_{0,0}^{\frac{1}{2}} \\
%		& \quad + (\delta+\nu) \norm{\nabla^2 \phi}_{H^1(\Omega)} + \e \nu e^{-\alpha t}.
%	\end{aligned}
%\end{equation}
Using \cref{L4,poincare}, it holds that
\begin{equation}\label{I81}
	\begin{aligned}
		I_{8,1} & \lesssim \nu \big[ (t+\Lambda)^{-\frac{1}{2}} \norm{\p_3^2 \Phi,\p_3^2 \Zv}_{L^2(\R)} +  (t+\Lambda)^{-1} \norm{\p_3 \Phi, \Phi \kappa}_{L^2(\R)} \\
		& \qquad\ + \norm{\nabla^3 \phi^\md, \nabla^3 \zeta^\md}_{L^2(\Omega)}\big].
	\end{aligned}
\end{equation}
To estimate $I_{8,2},$ it follows from \cref{Lem-GN}, \cref{Sobolev-1,poincare} that for $ j=0,1, $
\begin{align}
	\norm{\nabla^j \phi}_{L^4(\Omega)} & \lesssim \norm{\p_3^{j+1} \Phi}_{L^4(\R)} + \norm{\nabla^j \phi^\md}_{L^4(\Omega)} \notag \\
	& \lesssim  \norm{\p_3^3 \Phi}_{L^2(\R)}^{\frac{3+4j}{10}} \norm{\Phi}_{L^\infty(\R)}^{\frac{7-4j}{10}} + \norm{\nabla^j \phi^\md}_{H^1(\Omega)} \notag \\
	& \lesssim \nu^{\frac{7-4j}{10}} \big(\norm{\p_3^3 \Phi}_{L^2(\R)}^{\frac{3+4j}{10}} + \norm{\nabla^3 \phi^\md}_{L^2(\Omega)}^{\frac{3+4j}{10}}\big),
	\label{L4*}
\end{align}
%and 
%\begin{align}
%	\norm{\phi}_{L^4(\Omega)} & \lesssim \norm{\p_3\Phi}_{L^4(\R)} + \norm{\phi^\md}_{L^4(\Omega)} \notag \\
%	& \lesssim  \norm{\p_3^2 \phi^\od}_{L^2(\R)}^{\frac{3}{10}} \norm{\Phi}_{L^\infty(\R)}^{\frac{7}{10}} + \norm{\phi^\md}_{H^1(\Omega)} \notag \\
%	& \lesssim \nu^{\frac{7}{10}} \big( \norm{\p_3^2 \phi^\od}_{L^2(\R)}^{\frac{3}{10}} + \norm{ \phi^\md}_{H^1(\Omega)}^{\frac{3}{10}} \big), \label{L4*}
%\end{align}
and similarly,
\begin{equation}\label{L4**}
	\begin{aligned}
		\norm{\nabla^j \psi}_{L^4(\Omega)} & \lesssim \nu^{\frac{7-4j}{10}} \big(\norm{\p_3^3 \Psi}_{L^2(\R)}^{\frac{3+4j}{10}} + \norm{\nabla^3 \psi^\md}_{L^2(\Omega)}^{\frac{3+4j}{10}}\big).
%		\norm{\psi}_{L^4(\Omega)} & \lesssim\nu^{\frac{3}{4}} \big( \norm{\p_3^2 \psi^\od}_{L^2(\R)}^{\frac{1}{4}} + \norm{\nabla^2 \psi^\md}_{L^2(\Omega)}^{\frac{1}{4}} \big).
	\end{aligned}
\end{equation}
These, together with \cref{L4*}, yield that
\begin{align}
	I_{8,2} & \lesssim \nu \big( \norm{\p_3^3 \Phi,\p_3^3\Psi}_{L^2(\R)} + \norm{\nabla^3 \phi, \nabla^3 \psi}_{L^2(\Omega)}\big) \notag \\
	& \lesssim \nu \big[ \norm{\p_3^3 \Zv}_{L^2(\R)} + (t+\Lambda)^{-1} \norm{\p_3 \Phi, \Phi \kappa}_{L^2(\R)} \notag \\
	& \qquad\ + (t+\Lambda)^{-\frac{1}{2}} \norm{\p_3^2 \Phi}_{L^2(\R)} + \norm{\nabla^3 \phi, \nabla^3 \psi}_{L^2(\Omega)} \big]. \label{I82}
\end{align}
Thus, it follows from \cref{I82,I81,rel-4} that 
\begin{align*}
	I_{8,1} + I_{8,2} & \lesssim \nu \big[ \norm{\p_3^3 \Zv}_{L^2(\R)}  +  (t+\Lambda)^{-1} \norm{\p_3 \Phi, \Phi \kappa}_{L^2(\R)} \\
	& \qquad + (t+\Lambda)^{-\frac{1}{2}} \norm{\p_3^2 \Phi,\p_3^2 \Zv}_{L^2(\R)} + \norm{\nabla^2 \phi}_{H^1(\Omega)} + \norm{\nabla^3 \zeta}_{L^2(\Omega)} + \e \nu e^{-\alpha t}\big].
\end{align*}
This, together with \cref{ineq-4,I7}, yields \cref{est-6}.

\end{proof}

\vspace{.3cm}

Then collecting Lemmas \ref{Lem-est4}--\ref{Lem-est6}, there exist generic numbers $ M_{i,\flat}\geq 1 $ for $ i=1,2,3, $ such that if $ \Lambda^{-1}, \e $ and $ \nu $ are small, then it holds that
\begin{equation}\label{est-1*}
	\begin{aligned}
		& \frac{d}{dt} \Big( \sum_{i=1}^{3} M_{i,\flat} \Ac_{i,\flat} \Big) + \frac{1}{2} \Big(\sum_{i=1}^{3} c_{i,\flat} M_{i,\flat}  \Bc_{i,\flat}\Big) \\
		& \qquad \lesssim  (t+\Lambda)^{-1} \norm{\p_3\Phi, \Phi\kappa}_{L^2(\R)}^2 +  \nu \norm{\nabla^2 \phi}_{H^1(\Omega)}^2 + \nu \norm{\nabla^3 \zeta}_{L^2(\Omega)}^2 + \e \nu \Lambda^{\frac{1}{4}} e^{-\alpha t},
	\end{aligned}
\end{equation}
and 
\begin{equation}\label{est-1**}
	\begin{aligned}
		\sum_{i=1}^{3} M_{i,\flat} \Ac_{i,\flat}  & \sim  \norm{\p_3\Phi,\p_3\Zv}_{H^1(\R)}^2, \\
		\sum_{i=1}^{3} c_{i,\flat} M_{i,\flat} \Bc_{i,\flat} & \sim \norm{\p_3^2 \Phi}_{L^2(\R)}^2 + \norm{\p_3^2 \Zv}_{H^1(\R)}^2.
	\end{aligned}
\end{equation}

\vspace{.3cm}

\subsection{Estimates of Non-Zero Modes}\label{Sec-md}

The estimate of the non-zero modes is similar to \cite[Section 7]{Yuan2023n}. However, if we follow the analysis of \cite{Yuan2023n} to study the system of $(\phi^\md, \psi^\md)$, we will encounter a new difficulty in the energy estimates due to the large amplitude of the background wave \cref{profile}. Nevertheless, this difficulty can be overcome by considering the system of $(\phi^\md, \zeta^\md)$ instead of $(\phi^\md, \psi^\md)$.

To formulate the system for $(\phi^\md, \zeta^\md),$ we first subtract \cref{equ-phipsi} by \cref{equ-pert-od} to get that
\begin{equation}\label{equ-md}
	\begin{cases}
		\p_t \phi^\md + \dv \psi^\md = -f_0^\md, \\
		\p_t \psi^\md + \sum\limits_{i=1}^3 \p_i \big( u_i^\vs \psi^\md +  \mv^\vs \zeta_i^\md + \er_{1,i} \big) + \nabla \big(p'(\rhob) \phi^\md + n_2 \big) \\
		\qquad\qquad\qquad\quad = \mu \lap \zeta^\md + (\mu+\lambda) \nabla \dv \zeta^\md - \gv^\md,
	\end{cases}
\end{equation}
where 
\begin{equation}\label{r-1-2}
	\begin{aligned}
		\er_{1,i} & =  \big( u_i \mv - \ut_i \mvt \big)^\md - u_i^\vs \psi^\md - \mv^\vs \zeta_i^\md, \\
		n_2 & = \big(p(\rho)-p(\rhot)\big)^\md - p'(\rhob) \phi^\md.
	\end{aligned}
\end{equation}
Note that $ \psi = \rhot \zeta + \uvt \phi + \phi \zeta $, which implies that
\begin{align}
	\psi^\md = \rhob \zeta^\md + \uv^\vs \phi^\md +  \er_3, \label{psi-zeta}
\end{align}
where 
\begin{equation}\label{r-3}
	\er_3 = [ (\rhot-\rhob) \zeta + (\uvt-\uv^\vs) \phi ]^\md + \zeta^\od \phi^\md + \phi^\od \zeta^\md + (\phi^\md \zeta^\md)^\md.
\end{equation}
Then plugging \cref{psi-zeta} into \cref{equ-md} and using the fact that $\dv \uv^\vs = 0$, the system of $(\phi^\md, \zeta^\md)$ can be formulated as
\begin{equation}\label{equ-md*}
	\begin{cases}
		\p_t \phi^\md + \rhob \dv \zeta^\md + \uv^\vs \cdot \nabla\phi^\md = - f_0^\md - \dv \er_3, \\
		\rhob \p_t\zeta^\md + \rhob \uv^\vs \dnab \zeta^\md + p'(\rhob) \nabla\phi^\md + \p_t \uv^\vs \phi^\md - \mu \lap \zeta^\md - (\mu+\lambda) \nabla\dv \zeta^\md \\
		\qquad = \er_4 - \p_t \er_3  +\uv^\vs f_0^\md - \gv^\md.
	\end{cases}
\end{equation}
with $ \er_4:=  - \sum_{i=1}^{3} \p_i \er_{1,i} - \nabla n_2 + \uv^\vs \dv \er_3 - \uv^\vs \dnab \er_3. $

Combining \cref{poincare,Sobolev-1}, the remainder \cref{r-3} satisfies that
\begin{equation}\label{est-r-3}
	\begin{aligned}
		\norm{\nabla^j \er_3}_{L^2(\Omega)} & \lesssim \nu \norm{\nabla^j \phi^\md, \nabla^j \zeta^\md}_{L^2(\Omega)} + \e \nu e^{-\alpha t }, \qquad j=0,1,2.
	\end{aligned}
\end{equation}
This, together with \cref{psi-zeta,poincare}, yields that
\begin{align}
	\pm \norm{\nabla^j \psi^\md}_{L^2(\Omega)} & \lesssim \pm \norm{\nabla^j \zeta^\md}_{L^2(\Omega)} + \norm{ \nabla^j\phi^\md}_{L^2(\Omega)} + \e \nu e^{-\alpha t}, \quad j=0,1,2. \label{psi-est}
\end{align}
It follows from \cref{r-1-2}\textsubscript{1} that
\begin{align*}
	\er_{1,i}
	& = \big[(\ut_i -u_i^\vs) \psi + (\mvt - \mv^\vs) \zeta_i \big]^\md + \zeta_i^\md \psi^\od  + \zeta_i^\od \psi^\md + (\zeta_i^\md \psi^\md)^\md,
\end{align*}
and \cref{r-1-2}\textsubscript{2}, with the fact that $\rho = \rhob + \phi^\od + (\rhot-\rhob) + \phi^\md,$ yields that
\begin{align*}
	n_2 & = \big[ \big(p(\rho) - p(\rhob+\phi^\od)\big) - \big( p(\rhot) - p(\rhob) \big) - p'(\rhob) \phi^\md\big]^\md \\
	& = \big[ a_1 (\rhot-\rhob) \big]^\md + \big(a_2 \phi^\md\big)^\md,
\end{align*}
where
\begin{align*}
	a_1 & = \int_0^1 \big[ p'(\rhob+\phi^\od + s(\rhot-\rhob) + s \phi^\md) - p'(\rhob+s(\rhot-\rhob)) \big] ds \\
	& = O(1) \big(\abs{\phi^\od}+ \abs{\phi^\md}\big) \\
	%		& = O(1) \big(\abs{\phi^\od}+\abs{\phi^\md}\big), \\
	a_2 & = \int_0^1 \big[ p'(\rhob+\phi^\od + s(\rhot-\rhob) + s \phi^\md) - p'(\rhob) \big] ds \\
	& = O(1) \big(\abs{\rhot-\rhob} + \abs{\phi^\od}+\abs{\phi^\md}\big).
\end{align*}
Thus, it holds that
\begin{align}
	\norm{\er_{1,i}}_{H^1(\Omega)} & \lesssim \nu \norm{\nabla \zeta^\md,\nabla \phi^\md}_{L^2(\Omega)} + \e \nu e^{-\alpha t} \qquad \text{for} \ i=1,2,3, \label{r-1-est} \\
	\norm{n_2}_{H^1(\Omega)} & \lesssim \nu \norm{\nabla \phi^\md}_{L^2(\Omega)} + \e \nu e^{-\alpha t}. \label{r-2-est}
\end{align}
These, together with \cref{est-r-3}, yield that
\begin{equation}\label{est-r-4}
	\norm{\er_4}_{L^2(\Omega)} \lesssim \nu \norm{\nabla\phi^\md,\nabla\zeta^\md}_{L^2(\Omega)} + \e \nu e^{-\alpha t}.
\end{equation}
Moreover, we claim that
\begin{equation}\label{est-t-r3}
	\norm{\p_t \er_3}_{L^2(\Omega)} \lesssim  \nu \norm{\nabla \phi^\md, \nabla^2 \zeta^\md}_{L^2(\Omega)} + \e \nu \Lambda^{\frac{1}{4}} e^{-\alpha t }. 
\end{equation}
In fact, it follows from \cref{equ-md*,est-r-4} that
\begin{align*}
	\norm{\p_t\phi^\md,\p_t\zeta^\md}_{L^2(\Omega)} \lesssim \norm{\nabla\phi^\md, \nabla^2 \zeta^\md}_{L^2(\Omega)} + \norm{\p_t \er_3}_{L^2(\Omega)} + \e \Lambda^{\frac{1}{4}} e^{-\alpha t}.
\end{align*}
Using the Sobolev inequality and \cref{equ-phizeta}, it holds that
\begin{align*}
	\norm{\p_t \phi^\od, \p_t \zeta^\od}_{L^\infty(\R)} \lesssim \norm{\p_t \phi^\od, \p_t \zeta^\od}_{H^1(\R)} \lesssim \norm{\p_t \phi, \p_t \zeta}_{H^1(\Omega)},
\end{align*}
which, together with \cref{equ-phizeta,small-phizeta}, yields that
\begin{equation*}
	\norm{\p_t \phi^\od, \p_t \zeta^\od}_{L^\infty(\R)} \lesssim \norm{\phi,\zeta}_{H^3(\Omega)} + \norm{f_0,\gv}_{H^1(\Omega)} \lesssim \nu + \e \Lambda^{\frac{1}{4}} e^{-\alpha t}.
\end{equation*}
Then it follows from \cref{r-3} that
\begin{align*}
	\norm{\p_t \er_3}_{L^2(\Omega)} & \lesssim \e \nu e^{-\alpha t } + \norm{\p_t \zeta^\od,\p_t \phi^\od}_{L^\infty(\R)} \norm{ \phi^\md, \zeta^\md}_{L^2(\Omega)} + \nu \norm{\p_t \phi^\md, \p_t \zeta^\md}_{L^2(\Omega)} \\
	& \lesssim \nu \norm{\nabla \phi^\md, \nabla^2 \zeta^\md}_{L^2(\Omega)} + \nu \norm{\p_t \er_3}_{L^2(\Omega)} + \e \nu \Lambda^{\frac{1}{4}} e^{-\alpha t },
\end{align*}
which yields \cref{est-t-r3}.

\vspace{.1cm}

Now we establish the a priori estimates of $(\phi^\md,\zeta^\md).$

\begin{Lem}\label{Lem-exp-1}
	Under the assumptions of \cref{Prop-apriori}, there exist a large constant $ \Lambda_0\geq 1 $ and small positive constants $ \e_0 $ and $ \nu_0 $ such that, if $ \Lambda \geq \Lambda_0 $, $ \e \leq \e_0 $ and $ \nu \leq \nu_0, $ then
	\begin{equation}\label{est-md-1}
		\begin{aligned}
			& \frac{d}{dt} \Ac_{1,\sharp}  + c_{1,\sharp} \Bc_{1,\sharp} \lesssim (\Lambda^{-1}+\nu) \norm{\nabla\phi^\md, \nabla^2 \zeta^\md}^2_{L^2(\Omega)} + \e \nu \Lambda^{\frac{1}{4}} e^{-\alpha t},
		\end{aligned}
	\end{equation}
	where
	\begin{align*}
		\Ac_{1,\sharp} = \frac{p'(\rhob)}{\rhob^2} \norm{\phi^\md}_{L^2(\Omega)}^2 + \norm{ \zeta^\md}_{L^2(\Omega)}^2 \andd \Bc_{1,\sharp} =  \norm{\nabla \zeta^\md}_{L^2(\Omega)}^2.
	\end{align*}
\end{Lem}
\begin{proof}
	Multiplying $ \frac{p'(\rhob)}{\rhob} \phi^\md $ and $ \zeta^\md $ on \cref{equ-md*}$ _1 $ and \cref{equ-md*}$ _2 $, respectively, and using the fact that $\dv\uv^\vs = 0,$ one can obtain that
	\begin{align*}
		& \p_t \Big( \frac{p'(\rhob)}{2\rhob} \abs{\phi^\md}^2 + \frac{\rhob}{2} \abs{\zeta^\md}^2 \Big) + \mu \abs{\nabla\zeta^\md}^2 + (\mu+\lambda) \abs{\dv\zeta^\md}^2 \\
		& \quad = \dv(\cdots) + (\er_4 - \p_t \er_3) \cdot \zeta^\md - \frac{p'(\rhob)}{\rhob} \dv \er_3 \phi^\md - \p_t \uv^\vs \cdot \zeta^\md \phi^\md \\
		& \qquad - \frac{p'(\rhob)}{\rhob} f_0^\md \phi^\md + (\uv^\vs f_0^\md - \gv^\md) \cdot \zeta^\md.
	\end{align*}
	Then integrating the equality above over $\Omega$, and using \cref{est-r-4,est-t-r3,est-r-3}, one can get \cref{est-md-1}.
\end{proof}

\vspace{.1cm}

\begin{Lem}\label{Lem-exp-2}
	Under the assumptions of \cref{Prop-apriori}, there exist a large constant $ \Lambda_0\geq 1 $ and small positive constants $ \e_0 $ and $ \nu_0 $ such that, if $ \Lambda \geq \Lambda_0 $, $ \e \leq \e_0 $ and $ \nu \leq \nu_0, $ then
	\begin{equation}\label{est-md-2}
		\begin{aligned}
			& \frac{d}{dt} \Ac_{2,\sharp}  + c_{2,\sharp} \Bc_{2,\sharp}\lesssim \norm{\nabla\zeta^\md}_{L^2(\Omega)}^2 + \nu \norm{\nabla^2 \zeta^\md}_{L^2(\Omega)}^2 + \e \nu \Lambda^{\frac{1}{4}} e^{-\alpha t},
		\end{aligned}
	\end{equation}
	where
\begin{align*}
	\Ac_{2,\sharp} = \frac{\mut}{2\rhob^2} \norm{\nabla\phi^\md}_{L^2(\Omega)}^2 + \int_\Omega \zeta^\md\cdot \nabla\phi^\md dx \andd
	\Bc_{2,\sharp} =  \norm{\nabla \phi^\md}_{L^2(\Omega)}^2.
\end{align*}
\end{Lem}
\begin{proof}
	Multiplying $ \frac{\mut}{\rhob} \nabla \phi^\md $ and $ \nabla \phi^\md $ on $ \nabla $\cref{equ-md*}$ _1 $ and \cref{equ-md*}$ _2 $, respectively, leads to
	\begin{equation}\label{eq-md-1}
		\begin{aligned}
		& \p_t \Big( \frac{\mut}{2\rhob} \abs{\nabla\phi^\md}^2 + \rhob \zeta^\md \dnab \phi^\md \Big) + p'(\rhob) \abs{\nabla\phi^\md}^2 \\
		& \quad = \dv(\cdots) - \frac{\mut}{\rhob} \nabla\phi^\md \cdot \nabla \uv^\vs \nabla\phi^\md - \rhob \uv^\vs \cdot \nabla\zeta^\md \nabla\phi^\md - \phi^\md \p_t \uv^\vs \dnab \phi^\md \\
		& \qquad + \rhob \dv \zeta^\md \big(\rhob \dv\zeta^\md + \uv^\vs \dnab \phi^\md + f_0^\md + \dv \er_3\big) \\
		& \qquad + \nabla\phi^\md \cdot \Big( -\frac{\mut}{\rhob} \nabla f_0^\md + \uv^\vs f_0^\md - \gv^\md \Big) \\
		& \qquad + \nabla\phi^\md \cdot (\er_4 - \p_t \er_3) - \frac{\mut}{\rhob} \nabla \phi^\md  \dnab \dv\er_3,
	\end{aligned}
	\end{equation}
	where we have used the fact that for any $ h_0\in H^1(\Omega), \mathbf{h}=(h_1,h_2,h_3) \in H^2(\Omega), $
	\begin{equation}\label{fact}
		\nabla h_0 \cdot (\lap \mathbf{h} - \nabla\dv \mathbf{h}) = \dv (\nabla h_0 \times \text{curl} \mathbf{h}).
	\end{equation}
	To estimate the term $-\frac{\mut}{\rhob} \nabla \phi^\md  \dnab \dv\er_3$ on the right-hand side of \cref{eq-md-1}, one first notes that
	\begin{align*}
		\Big|\int_\Omega \nabla \phi^\md \cdot  \left(\zeta^\od \dnab\right) \nabla \phi^\md dx\Big| & = \frac{1}{2} \Big|\int_\Omega \dv \zeta^\od \abs{\nabla\phi^\md}^2 dx \Big| \lesssim \nu \norm{\nabla\phi^\md}_{L^2(\Omega)}^2,
	\end{align*}
	and it follows from \cref{Sob-inf,poincare} that
	\begin{align*}
		\norm{\nabla\phi^\md \cdot \big( (\zeta^\md \cdot \nabla) \nabla \phi^\md \big)^\md}_{L^1(\Omega)}
		& \lesssim \norm{\nabla\phi^\md}_{L^2(\Omega)} \norm{\zeta^\md}_{L^\infty(\Omega)} \norm{\nabla^2 \phi^\md}_{L^2(\Omega)} \\
		& \lesssim \nu \norm{\nabla\phi^\md}_{L^2(\Omega)} \norm{\zeta^\md}_{H^2(\Omega)} \\
		& \lesssim \nu \norm{\nabla\phi^\md}_{L^2(\Omega)} \norm{\nabla^2\zeta^\md}_{L^2(\Omega)}.
	\end{align*}
	These, together \cref{r-3}, yield that
	\begin{equation}\label{ineq-md-1}
		\Big|\int_\Omega \nabla \phi^\md \cdot \nabla\dv\er_3 dx\Big| \lesssim \nu \norm{\nabla \phi^\md,\nabla^2 \zeta^\md}_{L^2(\Omega)} + \e \nu e^{-\alpha t}.
	\end{equation} 
	Collecting \cref{est-t-r3,est-r-3,est-r-4,ineq-md-1}, one can get \cref{est-md-2} by integrating \cref{eq-md-1} over $\Omega$.
\end{proof}

\vspace{.1cm}

\begin{Lem}\label{Lem-exp-3}
	Under the assumptions of \cref{Prop-apriori}, there exist a large constant $ \Lambda_0\geq 1 $ and small positive constants $ \e_0 $ and $ \nu_0 $ such that, if $ \Lambda \geq \Lambda_0 $, $ \e \leq \e_0 $ and $ \nu \leq \nu_0, $ then
	\begin{equation}\label{est-md-3}
		\frac{d}{dt} \Ac_{3,\sharp} + c_{3,\sharp} \Bc_{3,\sharp} \lesssim \Bc_{1,\sharp} + \Bc_{2,\sharp}  + \e \nu \Lambda^{\frac{1}{4}} e^{-\alpha t},
	\end{equation}
where
\begin{align*}
	\Ac_{3,\sharp} = \norm{\nabla\zeta^\md}^2_{L^2(\Omega)} \andd \Bc_{3,\sharp} =  \norm{\nabla^2 \zeta^\md}_{L^2(\Omega)}^2.
\end{align*}
\end{Lem}
\begin{proof}
	Multiplying $ -\lap \zeta^\md $ on \cref{equ-md*}$ _2 $ and using the fact that $\lap \zeta^\md \cdot \nabla\dv \zeta^\md = \dv(\cdots) + \abs{\nabla\dv \zeta^\md}^2 $, one can obtain that
	\begin{align*}
		& \p_t \Big(\frac{\rhob}{2} \abs{\nabla\zeta^\md}^2 \Big) + \mu \abs{\lap \zeta^\md}^2 + (\mu+\lambda) \abs{\nabla\dv \zeta^\md}^2 \\
		& \qquad = \dv(\cdots) + \rhob (\uv^\vs \dnab) \zeta^\md \cdot \lap \zeta^\md + p'(\rhob) \nabla\phi^\md \cdot \lap \zeta^\md + \phi^\md \p_t \uv^\vs \cdot\lap \zeta^\md \\
		& \qquad\quad - (\er_4 -\p_t \er_3) \cdot \lap \zeta^\md - (\uv^\vs f_0^\md - \gv^\md) \cdot \lap \zeta^\md.
	\end{align*}
	By integration by parts, it holds that $\norm{\lap \zeta^\md}_{L^2(\Omega)}^2 = \norm{\nabla^2 \zeta^\md}_{L^2(\Omega)}^2 $.
	Then one can obtain \cref{est-md-3} with the use of \cref{est-t-r3,est-r-4}.
\end{proof}

\vspace{.1cm}

Collecting Lemmas \ref{Lem-exp-1}--\ref{Lem-exp-3}, there exist some generic numbers $ M_{i,\neq}\geq 1 $ for $ i=1,2,3, $ such that if $ \Lambda^{-1}, \e $ and $ \nu $ are small, then
\begin{equation}\label{est-2*}
	\begin{aligned}
		& \frac{d}{dt} \Big( \sum_{i=1}^{3} M_{i,\neq} \Ac_{i,\neq} \Big) + \frac{1}{2} \Big(\sum_{i=1}^{3} c_{i,\neq} M_{i,\neq} \Bc_{i,\neq} \Big) \lesssim \e \nu \Lambda^{\frac{1}{4}} e^{-\alpha t},
	\end{aligned}
\end{equation}
and 
\begin{equation}\label{est-2**}
	\begin{aligned}
		\sum_{i=1}^{3} M_{i,\neq} \Ac_{i,\neq} & \sim \norm{\phi^\md, \zeta^\md}_{H^1(\Omega)}^2, \\
%		M_1' \Ac_{1,\sharp} + M_2' \Ac_{2,\sharp} + \Ac_{3,\sharp} & \gtrsim \norm{\phi^\md, \psi^\md}_{H^1(\Omega)}^2 - \e \nu e^{-\alpha t}, \\
		\sum_{i=1}^{3} c_{i,\neq} M_{i,\neq} \Bc_{i,\neq} & \sim \norm{\nabla \phi^\md}_{L^2(\Omega)}^2 + \norm{\nabla \zeta^\md}_{H^1(\Omega)}^2.
	\end{aligned}
\end{equation}
%By the Poincar\'{e} inequality, it holds that $ \norm{\phi^\md, \psi^\md}_{L^2(\Omega)} \lesssim \norm{\nabla \phi^\md,\nabla\psi^\md}_{L^2(\Omega)}. $
It follows from \cref{poincare} that
\begin{align*}
	\sum_{i=1}^{3} M_{i,\neq} \Ac_{i,\neq} \lesssim \norm{\nabla\phi^\md, \nabla\zeta^\md}_{L^2(\Omega)}^2 \lesssim \sum_{i=1}^{3} c_{i,\neq} M_{i,\neq} \Bc_{i,\neq}.
\end{align*}
Then there exists a generic constant $ \alpha_1 \in (0,\alpha) $ such that
\begin{equation}\label{exp-md}
	\norm{\phi^\md, \zeta^\md}_{H^1(\Omega)}^2 \lesssim \big( \norm{ \phi_0,\psi_0 }_{H^1(\Omega)}^2 + \e \nu \Lambda^{\frac{1}{4}} \big) e^{-\alpha_1 t}.
\end{equation}

%Moreover, it follows from \cref{est-2*} that
%\begin{equation}\label{est-md}
%	\begin{aligned}
%		& \sup_{t \in(0, T)} \norm{\phi^\md,\psi^\md}_{H^1(\Omega)}^2 + \int_0^T \big(\norm{\phi^\md, \psi^\md}_{H^1(\Omega)}^2 + \norm{\nabla^2 \zeta^\md}_{L^2(\Omega)}^2 \big) dt \\
%		& \qquad \lesssim \norm{\phi_0,\psi_0}_{H^2(\Omega)}^2 + \e.
%	\end{aligned}
%\end{equation}

\vspace{.3cm}

\subsection{Estimates of Higher-Order Derivatives.}\label{Sec-orn}

%Combining \cref{est-od,est-md} and noting that
%\begin{align*}
%	\norm{\psi^\od}_{H^2(\R)}^2 \lesssim \norm{\p_3 \Zv}_{H^2(\R)}^2 + \delta \norm{\p_3 \Phi}_{H^2(\R)}^2 + \delta \int_\R \Phi^2 \kappa^2 dx_3,
%\end{align*}
%one has that
%\begin{equation}\label{key}
%	\begin{aligned}
%		& \sup_{t \in(0, T)} \big( \norm{\Phi, \Psi}_{L^2(\R)}^2 + \norm{\phi,\psi}_{H^1(\Omega)}^2 \big)  + \int_0^T \big(\norm{\phi, \psi}_{H^1(\Omega)}^2 + \norm{\nabla^2 \zeta^\md}_{L^2(\Omega)}^2 \big) dt \\
%		& \qquad \lesssim \norm{\phi_0,\psi_0}_{H^2(\Omega)}^2 + \e + \nu \int_0^T \norm{\nabla^2\phi, \nabla^2 \psi}_{H^1(\Omega)}^2 dt.
%	\end{aligned}
%\end{equation}

Now we return to the original perturbed system \cref{equ-phizeta} to estimate the higher-order derivatives of the perturbation. It is noted that in the proof of Lemmas \ref{Lem-est7} and \ref{Lem-est8}, we use only the a priori assumption that $ \sup_{t \in(0, T)}\norm{(\phi,\zeta)}_{H^2(\Omega)} $ is small. Thus, with a higher-order assumption \cref{apriori}, one can use similar arguments to estimate the third-order derivatives, i.e. Lemmas \ref{Lem-est9} and \ref{Lem-est10}, whose proofs are omitted for convenience.

\begin{Lem}\label{Lem-est7}
	Under the assumptions of \cref{Prop-apriori}, there exist a large constant $ \Lambda_0\geq 1 $ and small positive constants $ \e_0 $ and $ \nu_0 $ such that, if $ \Lambda\geq \Lambda_0 $, $ \e \leq \e_0 $ and $ \nu \leq \nu_0, $ then
	\begin{equation}\label{est7}
		\begin{aligned}
			\frac{d}{dt} \Ac_{1} + c_{1} \Bc_{1} & \lesssim  \sum_{j=0}^{1} (t+\Lambda)^{-2+j} \norm{\nabla^j (\phi, \zeta) }_{L^2(\Omega)}^2 \\
			& \quad + \norm{ \nabla^2 \zeta}_{L^2(\Omega)}^2 + \nu \norm{ \nabla^3 \zeta}_{L^2(\Omega)}^2 + \e \nu \Lambda^{\frac{1}{4}} e^{-\alpha t},
		\end{aligned}
	\end{equation}
	where
	\begin{equation}\label{AB1}
		\Ac_{1} =  \sum_{i=1}^{3} \int_\Omega \Big(\frac{\mut}{2\rho^2} \abs{\nabla \p_i   \phi}^2 + \nabla\p_i\phi  \cdot \p_i \zeta \Big) dx \andd
		\Bc_{1} =  \norm{\nabla^2 \phi}_{L^2(\Omega)}^2.
	\end{equation}
\end{Lem}

\begin{proof}
	For $ i \in \{1,2,3\} $, taking the derivative $ \nabla \p_i  $ on \cref{equ-phizeta}\textsubscript{1} and multiplying the resulting equation by $ \cdot \frac{\nabla \p_i \phi}{\rho^2}, $
	one can get that
	\begin{equation}\label{eq7}
		\p_t \Big(\frac{\abs{\nabla \p_i   \phi}^2}{2\rho^2}\Big) + \frac{1}{\rho} \nabla \p_i  \dv\zeta \cdot \nabla \p_i\phi  = \dv (\cdots) + I_9,
	\end{equation}
	where
	\begin{equation*}
		\begin{aligned}
			I_9 & =  - \frac{1}{\rho^2} \nabla \p_i \phi \cdot \Big\{ \big[ \nabla \p_i (\rho \dv \zeta ) - \rho \nabla \p_i \dv \zeta\big] + \big[ \nabla\p_i (\uv \cdot\nabla\phi) - (\uv \cdot \nabla) \nabla \p_i \phi\big] \\
			& \qquad + \nabla \p_i \big[ \dv \uvt \phi + \nabla\rhot \cdot \zeta \big] - \frac{3}{2} \big(\dv \uvt + \dv \zeta\big) \nabla\p_i \phi + \nabla\p_i f_0 \Big\}.
		\end{aligned}
	\end{equation*}
%	Here we have used the fact that $\dv \uv^\vs = 0.$
	Using \cref{exp-diff,exp-diff-u,exp-f-g}, one can get that
	\begin{equation}\label{I9a}
		\begin{aligned}
		\norm{I_9}_{L^1(\Omega)} & \lesssim \norm{\nabla^2 \phi}_{L^2(\Omega)} \cdot \big[ \norm{\nabla\phi}_{L^4(\Omega)} \norm{\nabla^2 \zeta}_{L^4(\Omega)} + \norm{\nabla^2 \phi }_{L^2(\Omega)} \norm{\nabla \zeta}_{L^\infty(\Omega)} \\
		& \qquad + \sum_{j=0}^{2} (t+\Lambda)^{-\frac{3-j}{2}} \norm{\nabla^j \phi}_{L^2(\Omega)} + \e e^{-\alpha t} \norm{\phi, \zeta}_{H^2(\Omega)} + \e \Lambda^{\frac{1}{4}} e^{-\alpha t} \big]. 
	\end{aligned}
	\end{equation}
	It follows from \cref{Lem-GN} that
	\begin{equation}\label{tool3}
		\begin{aligned}
			\norm{	\nabla\phi}_{L^4(\Omega)} & \lesssim \sum_{k=1}^{3} \norm{ \nabla^2 \phi}_{L^2(\Omega)}^{\frac{1}{2} + \frac{k}{8}} \norm{\phi}_{L^2(\Omega)}^{\frac{1}{2} - \frac{k}{8}} \lesssim \nu^{\frac{1}{2}} \norm{\nabla^2 \phi}_{L^2(\Omega)}^{\frac{1}{2}}, \\
			\norm{\nabla^2 \zeta}_{L^4(\Omega)} & \lesssim \sum_{k=1}^{3} \norm{ \nabla^3 \zeta}_{L^2(\Omega)}^{\frac{k}{4}} \norm{\nabla^2 \zeta}_{L^2(\Omega)}^{1 - \frac{k}{4}},
		\end{aligned}
	\end{equation}
	which yields that
	\begin{align}
		\norm{\nabla\phi}_{L^4(\Omega)} \norm{\nabla^2 \zeta}_{L^4(\Omega)} \lesssim \nu \norm{\nabla^2 \phi}_{L^2(\Omega)} + \nu \norm{\nabla^2 \zeta}_{H^1(\Omega)}. \label{tool1}
	\end{align}
	Using \cref{Sob-inf} with $h=\nabla \zeta$ therein, it holds that
	\begin{align}
		\norm{\nabla^2 \phi }_{L^2(\Omega)} \norm{\nabla \zeta}_{L^\infty(\Omega)}
		& \lesssim \nu \norm{\nabla^2 \phi}_{L^2(\Omega)} + \nu \norm{\nabla^2 \zeta}_{H^1(\Omega)}. \label{tool2}
	\end{align}
	Then combining \cref{tool1,tool2,I9a}, one has that
	\begin{equation}\label{I9b}
		\begin{aligned}
			\norm{I_9}_{L^1(\Omega)} & \lesssim (\Lambda^{-\frac{1}{2}}+\nu) \norm{\nabla^2 \phi}_{L^2(\Omega)}^2 + \sum_{j=0}^{1} (t+\Lambda)^{-2+j} \norm{\nabla^j \phi}_{L^2(\Omega)}^2 \\
			& \quad + \nu \norm{\nabla^2 \zeta}_{H^1(\Omega)}^2 + \e \nu \Lambda^{\frac{1}{4}} e^{-\alpha t}.
		\end{aligned}
	\end{equation}
	
	On the other hand, multiplying $ \cdot \nabla \p_i \phi $ on $ \p_i \big(  \frac{1}{\rho}\cdot$\cref{equ-phizeta}$ _2 \big) $, 
	one can get that
	\begin{equation}\label{eq7a}
		\begin{aligned}
			& \p_t\big( \nabla\p_i\phi  \cdot \p_i \zeta \big) + \frac{p'(\rho)}{\rho} \abs{\nabla\p_i \phi}^2 - \frac{\mut}{\rho} \nabla \p_i\phi  \cdot \nabla \dv \p_i \zeta  = \dv (\cdots) + I_{10},
		\end{aligned}
	\end{equation}
	where 
	\begin{align*}
		I_{10} & =  - \nabla \p_i \phi \cdot \Big[ \p_i (\uv \cdot \nabla\zeta) + \p_i \Big( \frac{p'(\rho)}{\rho} \Big)  \nabla\phi  +\p_i \Big( \frac{p'(\rho)-p'(\rhot)}{\rho} \nabla \rhot \Big) + \p_i (\zeta \dnab \uvt)  \\
		& \qquad\qquad + \p_i \Big( \frac{\phi}{\rho}  \left(\p_t \uvt + \uvt \dnab \uvt\right) \Big) + \frac{\p_i \rho}{\rho^2} \big( \mu \lap \zeta + (\mu+\lambda) \nabla \dv \zeta \big)  + \p_i \Big( \frac{\gv - f_0 \uvt}{\rho}\Big) \Big] \\
		&  \qquad + \frac{\mu }{\rho^2} \sum_{j=1}^{3}  \p_j \rho \nabla \p_i \phi \cdot \big( \p_j \p_i \zeta - \nabla\p_i \zeta_j \big) - \dv\p_i \zeta \p_i \p_t \phi.
	\end{align*} 
	Here we have used the fact that for any $(h_0,\mathbf{h}) \in H^1(\Omega) \times H^2(\Omega), $
	\begin{align*}
		& \frac{1}{\rho} \nabla h_0 \cdot \left( \lap \mathbf{h} - \nabla\dv \mathbf{h} \right) \notag \\
		& \qquad = \sum_{j=1}^{3} \p_j \Big[ \frac{1}{\rho} \nabla h_0 \cdot (\p_j \mathbf{h}- \nabla h_j) \Big] +  \frac{1}{\rho^2} \sum_{j=1}^{3}  \p_j \rho  \nabla h_0 \cdot \big( \p_j \mathbf{h} - \nabla h_j \big).
	\end{align*}
	Note that $ \norm{\phi,\zeta}_{H^2(\Omega)} \lesssim \nu $, then it holds that
	\begin{equation}\label{I10a}
		\begin{aligned}
		\norm{I_{10}}_{L^1(\Omega)} & \lesssim \norm{\nabla^2 \phi}_{L^2(\Omega)} \Big[  \norm{\nabla^2 \zeta}_{L^2(\Omega)} + \norm{\nabla\zeta}_{L^4(\Omega)}^2 + \norm{\nabla\phi}_{L^4(\Omega)}^2  \\
		& \qquad + (t+\Lambda)^{-\frac{1}{2}} \norm{\nabla\phi, \nabla\zeta}_{L^2(\Omega)}  + (t+\Lambda)^{-1} \norm{\phi, \zeta}_{L^2(\Omega)}  \\
		& \qquad + \e e^{-\alpha t} \norm{\phi,\zeta}_{H^2(\Omega)}  + \norm{\nabla \phi}_{L^4(\Omega)} \norm{\nabla^2 \zeta}_{L^4(\Omega)} + \e \Lambda^{\frac{1}{4}} e^{-\alpha t} \Big] \\
		& \quad + \norm{\nabla^2 \zeta}_{L^2(\Omega)} \norm{\nabla\p_t \phi}_{L^2(\Omega)}.
	\end{aligned}
	\end{equation}
	Using \cref{tool3}$ _1 $, one has that
	\begin{equation}\label{tool4}
		\norm{\nabla \phi}_{L^4(\Omega)}^2 \lesssim \nu \norm{\nabla^2 \phi}_{L^2(\Omega)}  \andd
		\norm{\nabla\zeta}_{L^4(\Omega)}^2 \lesssim \nu \norm{\nabla^2 \zeta}_{L^2(\Omega)}.
	\end{equation}
	It follows from \cref{equ-phizeta}$ _1 $ that 
	\begin{equation}\label{ineq7}
		\begin{aligned}
		\norm{\nabla\p_t \phi}_{L^2(\Omega)} & \lesssim \norm{\nabla^2 \zeta}_{L^2(\Omega)} + \norm{\nabla\phi}_{L^4(\Omega)} \norm{\nabla\zeta}_{L^4(\Omega)} + \norm{\nabla^2 \phi}_{L^2(\Omega)} \\
		& \quad + \sum_{j=0}^{1} (t+\Lambda)^{-\frac{2-j}{2}} \norm{\nabla^j \phi}_{L^2(\Omega)} + \e \Lambda^{\frac{1}{4}} e^{-\alpha t}.
	\end{aligned}
	\end{equation}
	Then using \cref{I10a,ineq7,tool1,tool4}, one can get that
	\begin{equation}\label{I10b}
		\begin{aligned}
			\norm{I_{10}}_{L^1(\Omega)} & \leq \Big(\frac{p'(\rho)}{16\rho}+ \nu\Big) \norm{\nabla^2 \phi }_{L^2(\Omega)}^2 + C \norm{ \nabla^2 \zeta}_{L^2(\Omega)}^2 + C\nu \norm{ \nabla^3 \zeta}_{L^2(\Omega)}^2 \\
			& \quad  + C\sum_{j=0}^{1} (t+\Lambda)^{-2+j} \norm{\nabla^j (\phi, \zeta) }_{L^2(\Omega)}^2 + C\e \nu \Lambda^{\frac{1}{4}} e^{-\alpha t}.
		\end{aligned}
	\end{equation}
	Thus, using \cref{I9b,I10b}, one can sum $\frac{1}{\mut}\cdot$\cref{eq7} and \cref{eq7a} together to obtain \cref{est7}.

\end{proof}

\vspace{.2cm}

\begin{Lem}\label{Lem-est8}
	Under the assumptions of \cref{Prop-apriori}, there exist a large constant $ \Lambda_0\geq 1 $ and small positive constants $ \e_0 $ and $ \nu_0 $ such that, if $ \Lambda\geq \Lambda_0 $, $ \e \leq \e_0 $ and $ \nu \leq \nu_0, $ then
	\begin{equation}\label{est8}
		\begin{aligned}
			\frac{d}{dt} \Ac_2 + c_2 \Bc_2
			& \lesssim \sum_{j=0}^{1} (t+\Lambda)^{-2+j} \norm{\nabla^j (\phi, \zeta) }_{L^2(\Omega)}^2 + \norm{\nabla^2 \phi, \nabla^2 \zeta}_{L^2(\Omega)}^2 + \e \nu \Lambda^{\frac{1}{4}} e^{-\alpha t},
		\end{aligned}
	\end{equation}
	where
	\begin{align*}
		\Ac_2 = \norm{\nabla^2 \zeta}_{L^2(\Omega)}^2 \andd
		\Bc_2 = \norm{\nabla^3 \zeta}_{L^2(\Omega)}^2.
	\end{align*}
\end{Lem}

\begin{proof}
	For fixed $ i \in \{1,2,3\}, $ taking the derivative $ \p_i  $ on $ \frac{1}{\rho}\cdot $\cref{equ-phizeta}$ _2 $ and multiplying the resulting equation by $ \cdot(- \lap \p_i \zeta), $ yield that
	\begin{equation}\label{eq8}
		\begin{aligned}
			& \frac{1}{2} \p_t \big( \abs{\nabla\p_i  \zeta}^2 \big) + \frac{\mu}{\rho} \abs{ \lap\p_i\zeta}^2 + \frac{\mu+\lambda}{\rho} \abs{\nabla\dv \p_i \zeta}^2 = \dv (\cdots)  +  I_{11},
		\end{aligned}
	\end{equation}
	where 
	\begin{align*}
		I_{11} & = \lap \p_i \zeta \cdot \p_i \Big[  \uv\cdot \nabla\zeta +\frac{p'(\rho)}{\rho} \nabla \phi + \frac{p'(\rho)-p'(\rhot)}{\rho} \nabla\rhot  + \zeta \cdot \nabla\uvt + \frac{\phi}{\rho} \left(\p_t \uvt + \uvt \dnab \uvt\right) \Big]  \\
		& \quad + \lap \p_i \zeta \cdot \frac{\p_i \rho}{\rho^2} \big[ \mu \lap\zeta + (\mu+\lambda) \nabla \dv \zeta \big] + \lap \p_i \zeta \cdot \p_i \Big( \frac{\gv - f_0 \uvt}{\rho} \Big) \\
		& \quad -\frac{\mu+\lambda}{\rho^2} \dv\p_i \zeta \nabla \rho \cdot (\lap \p_i\zeta - \nabla \dv \p_i\zeta).
	\end{align*}
	Here we have used the fact that for any $\mathbf{h} \in H^2(\Omega),$
	\begin{align*}
		& \frac{1}{\rho} \nabla\dv \mathbf{h} \cdot \big( \nabla \dv \mathbf{h} - \lap \mathbf{h} \big) = \dv \Big[\frac{1}{\rho} \dv \mathbf{h} \big( \nabla \dv \mathbf{h} - \lap \mathbf{h} \big)\Big] + \frac{\dv \mathbf{h}}{\rho^2}  \nabla\rho \cdot  \big( \nabla \dv \mathbf{h} - \lap \mathbf{h} \big).
	\end{align*}
%	Note that $ (\rho^\vs, m_3^\vs) = (\rhob, 0). $ 
	Then it holds that
	\begin{equation}\label{I11}
		\begin{aligned}
		\norm{I_{11}}_{L^1(\Omega)} & \lesssim \norm{\nabla^3 \zeta}_{L^2(\Omega)} \Big[ \norm{\nabla^2 \zeta}_{L^2(\Omega)} + \norm{\nabla\zeta}_{L^4(\Omega)}^2 + \norm{\nabla^2 \phi}_{L^2(\Omega)} + \norm{\nabla\phi}_{L^4(\Omega)}^2 \\
		& \qquad + \norm{\nabla\phi}_{L^4(\Omega)} \norm{\nabla^2 \zeta}_{L^4(\Omega)} + \sum_{j=0}^{1}  (t+\Lambda)^{-\frac{2-j}{2}} \norm{\nabla^j (\phi, \zeta)}_{L^2(\Omega)} \\
		& \qquad  + \e e^{-\alpha t} \sup_{ t\in(0,T)} \norm{\phi,\zeta}_{H^2(\Omega)} + \e \Lambda^{\frac{1}{4}} e^{-\alpha t}  \Big],
		\end{aligned}
	\end{equation}
	which, together with \cref{tool1,tool4}, yields that
	\begin{align*}
		\norm{I_{11}}_{L^1(\Omega)} & \leq \Big(\frac{\mu}{16\rho}+\nu\Big) \norm{\nabla^3 \zeta}_{L^2(\Omega)}^2 + C \norm{\nabla^2 \phi, \nabla^2\zeta}_{L^2(\Omega)}^2 \\
		& \quad + C \sum_{j=0}^{1} (t+\Lambda)^{-2+j} \norm{\nabla^j (\phi, \zeta) }_{L^2(\Omega)}^2 + C \e \nu \Lambda^{\frac{1}{4}} e^{-\alpha t}.
	\end{align*}
	Note that $\norm{\lap \p_i\zeta}_{L^2(\Omega)} = \norm{\nabla^2 \p_i \zeta}_{L^2(\Omega)}. $ Then one can integrate \cref{eq8} on $ \Omega $ to obtain \cref{est8}.

\end{proof}

With a higher-order a priori assumption that $ \sup_{t \in(0, T)}\norm{(\phi,\zeta)}_{H^3(\Omega)} $ is small, one can prove the estimate of the third-order derivatives of $(\phi,\zeta)$ similarly.

\begin{Lem}\label{Lem-est9}
	Under the assumptions of \cref{Prop-apriori}, there exist a large constant $ \Lambda_0\geq 1 $ and small positive constants $ \e_0 $ and $ \nu_0 $ such that, if $ \Lambda\geq \Lambda_0 $, $ \e \leq \e_0 $ and $ \nu \leq \nu_0, $ then
	\begin{equation}\label{est9}
		\begin{aligned}
			\frac{d}{dt} \Ac_3 + c_3 \Bc_3
			& \lesssim  \sum_{j=0}^{1} (t+\Lambda)^{-3+j} \norm{\nabla^j (\phi, \zeta) }_{L^2(\Omega)}^2  \\
			& \quad + \norm{\nabla^2 \phi}_{L^2(\Omega)}^2 + \norm{\nabla^2 \zeta}_{H^1(\Omega)}^2 + \e \nu \Lambda^{\frac{1}{4}} e^{-\alpha t},
		\end{aligned}
	\end{equation}
	where
	\begin{align*}
		\Ac_3 = \sum_{i,j=1}^{3} \int_\Omega \Big(\frac{\mut}{2\rho^2} \abs{\nabla \p_{ij} \phi}^2 + \nabla \p_{ij} \phi  \cdot \p_{ij} \zeta \Big) dx \andd
		\Bc_3 = \norm{\nabla^3 \phi}_{L^2(\Omega)}^2.
	\end{align*}
\end{Lem}

\vspace{.2cm}

\begin{Lem}\label{Lem-est10}
	Under the assumptions of \cref{Prop-apriori}, there exist a large constant $ \Lambda_0\geq 1 $ and small positive constants $ \e_0 $ and $ \nu_0 $ such that, if $ \Lambda\geq \Lambda_0 $, $ \e \leq \e_0 $ and $ \nu \leq \nu_0, $ then
	\begin{equation}\label{est10}
		\begin{aligned}
			\frac{d}{dt} \Ac_4 + c_4 \Bc_4
			& \lesssim \sum_{j=0}^{1} (t+\Lambda)^{-3+j} \norm{\nabla^j (\phi, \zeta) }_{L^2(\Omega)}^2  \\
			& \quad +  \norm{\nabla^2 \phi,\nabla^2 \zeta}_{H^1(\Omega)}^2 +\e \nu \Lambda^{\frac{1}{4}} e^{-\alpha t},
		\end{aligned}
	\end{equation}
	where
	\begin{align*}
		\Ac_4 = \norm{\nabla^3 \zeta}_{L^2(\Omega)}^2 \andd \Bc_4 = \norm{\nabla^4 \zeta}_{L^2(\Omega)}^2.
	\end{align*}
\end{Lem}

\vspace{.1cm}

Collecting Lemmas \ref{Lem-est7}--\ref{Lem-est10}, there exist some generic numbers $ M_{i}\geq 1 $ for $ i=1,2,3,4 $ such that if $ \Lambda^{-1}, \e $ and $ \nu $ are small, then
\begin{equation}\label{est-3}
	\begin{aligned}
		& \frac{d}{dt} \Big( \sum_{i=1}^{4} M_{i} \Ac_{i} \Big) + \frac{1}{2} \Big(\sum_{i=1}^{4} c_{i} M_{i} \Bc_{i} \Big) \\
		& \qquad  \lesssim \sum_{j=0}^{1} (t+\Lambda)^{-2+j} \norm{\nabla^j (\phi, \zeta) }_{L^2(\Omega)}^2 + \norm{ \nabla^2 \zeta}_{L^2(\Omega)}^2 + \e \nu\Lambda^{\frac{1}{4}} e^{-\alpha t},
	\end{aligned}
\end{equation}
and 
\begin{equation}\label{est-3a}
	\begin{aligned}
		\pm \sum_{i=1}^{4} M_{i} \Ac_{i} & \lesssim \pm \norm{\nabla^2 \phi, \nabla^2 \zeta}_{H^1(\Omega)}^2 + \norm{\nabla\zeta}_{L^2(\Omega)}^2, \\
		%		M_1' \Ac_{1,\sharp} + M_2' \Ac_{2,\sharp} + \Ac_{3,\sharp} & \gtrsim \norm{\phi^\md, \psi^\md}_{H^1(\Omega)}^2 - \e \nu e^{-\alpha t}, \\
		\sum_{i=1}^{4} c_{i} M_{i} \Bc_{i} & \sim \norm{\nabla^2 \phi, \nabla^3 \zeta}_{H^1(\Omega)}^2.
	\end{aligned}
\end{equation}
Using \cref{poincare,rel-2,rel-1,rel-0}, one has that
\begin{equation}\label{est-3b}
	\begin{aligned}
		& \frac{d}{dt} \Big( \sum_{i=1}^{4} M_{i} \Ac_{i} \Big) + \frac{1}{4} \Big(\sum_{i=1}^{4} c_{i} M_{i} \Bc_{i} \Big) \\
		& \qquad \lesssim (t+\Lambda)^{-1} \norm{\p_3^2 \Phi, \p_3^2 \Zv }_{L^2(\R)}^2  + (t+\Lambda)^{-2} \norm{\p_3 \Phi, \p_3 \Zv, \Phi \kappa }_{L^2(\R)}^2 \\
		& \qquad \quad +  \norm{\p_3^3 \Zv}_{L^2(\R)}^2 + \e \nu \Lambda^{\frac{1}{4}} e^{-\alpha t}.
	\end{aligned}
\end{equation}

\vspace{.2cm}

\begin{proof}[Proof of \cref{Prop-apriori}]
	Collecting \cref{est-1*,est-2*,est-3b}, one can choose a linear combination of  $ \Ac_{i,\flat}, \Ac_{i,\neq}$ and $\Ac_{j}$ (resp. $ \Bc_{i,\flat}, \Bc_{i,\neq} $ and $ \Bc_{j} $) for $i=1,2,3 $ and $ j=1,2,3,4$, denoted by $ \Ec_1 $ (resp. $ \Dc_1 $), such that
	\begin{equation}\label{est-3*}
		\begin{aligned}
			\frac{d}{dt} \Ec_1 + \Dc_1 & \lesssim  (t+\Lambda)^{-1} \norm{\p_3 \Phi, \p_3 \Zv, \Phi \kappa}_{L^2(\R)}^2 + \e \nu\Lambda^{\frac{1}{4}} e^{-\alpha t},
		\end{aligned}
	\end{equation}
	and
	\begin{equation}\label{est-3**}
		\begin{aligned}
			\pm \Ec_1 & \lesssim \pm \left( \norm{\phi}_{H^3(\Omega)}^2 + \norm{\p_3 \Zv}_{H^1(\R)}^2 +  \norm{\zeta^\md}_{H^1(\Omega)}^2 + \norm{\nabla^2 \zeta}_{H^1(\Omega)}^2 \right) \\
			& \quad + (t+\Lambda)^{-1} \norm{\Phi\kappa}_{L^2(\R)}^2 + \e^2 \nu^2 e^{-\alpha t}, \\
			%	\Ac_1 & \gtrsim \norm{\phi, \psi}_{H^1(\Omega)}^2 - \e e^{-\alpha t} - \delta^2 (t+1)^{-1} \norm{\Phi}_{L^2(\R)}^2, 
			\Dc_1 & \sim \norm{\nabla \phi}_{H^2(\Omega)}^2 + \norm{\p_3^2 \Zv}_{H^1(\R)}^2 + \norm{\nabla \zeta^\md}_{H^1(\Omega)}^2 + \norm{\nabla^3 \zeta}_{H^1(\Omega)}^2.
		\end{aligned}
	\end{equation}
	Here we have used \cref{rel-1} in the proof of \cref{est-3**}\textsubscript{1}.
	
	\vspace{.1cm}

	With \cref{est2a,exp-md}, integrating \cref{est-3*} with respect to $t\in(0,T)$ yields that
	\begin{equation}\label{est-E1}
		\sup_{t\in(0,T)} \Ec_1(t) + \int_0^T \Dc_1(t) dt \lesssim \norm{\Phi_0,\Psi_0}_{L^2(\R)}^2 + \norm{\phi_0,\psi_0}_{H^3(\Omega)}^2 + \e \nu\Lambda^{\frac{1}{4}}.
	\end{equation}
	It follows from  \cref{rel-1,rel-2,est-3**} that
	\begin{equation}\label{ED1-eqv}
		\begin{aligned}
			\pm \Ec_1 &  \lesssim \pm \norm{\phi,\zeta}_{H^3(\Omega)}^2 +  \norm{\Phi\kappa}_{L^2(\R)}^2 + \e^2 \nu^2  e^{-\alpha t}, \\
			\pm \Dc_1 &  \lesssim \pm \big(\norm{\nabla\phi}_{H^2(\Omega)}^2 + \norm{\nabla \zeta}_{H^3(\Omega)}^2\big) + (t+\Lambda)^{-1} \norm{\p_3\Phi, \Phi\kappa}_{L^2(\R)}^2 + \e^2 \nu^2 e^{-\alpha t}.
		\end{aligned}
	\end{equation}
	This, together with \cref{est2a,est-E1,exp-md}, yields that
	\begin{equation}\label{est}
		\begin{aligned}
			& \sup_{ t\in(0,T)} \big( \norm{\Phi, \Psi}_{L^2(\R)}^2 + \norm{\phi, \zeta}_{H^3(\Omega)}^2 \big) \\
			& \qquad \qquad + \int_0^T \big(\norm{\p_3 \Phi, \p_3 \Zv, \Phi \kappa, \Zv \kappa}_{L^2(\R)}^2  + \norm{\nabla \phi}_{H^2(\Omega)}^2 + \norm{\nabla \zeta}_{H^3(\Omega)}^2\big) dt  \\
%			& \qquad \lesssim \sup_{ t\in(0,T)} \norm{\Phi,\Zv}_{L^2(\R)}^2 + \sup_{ t\in(0,T)} \Ec_1 + \int_0^T \Dc_1(t) dt + \int_0^T \norm{\p_3 \Phi, \Phi\kappa}_{L^2(\R)}^2 dt + \e^2 \nu^2 \\
			& \qquad \lesssim \norm{\Phi_0,\Psi_0}_{L^2(\R)}^2 + \norm{\phi_0,\psi_0}_{H^3(\Omega)}^2 + \e \Lambda^{\frac{1}{4}}.
		\end{aligned}
	\end{equation}
	Here we have used the fact that $\e \nu \Lambda^{\frac{1}{4}} \lesssim \e \Lambda^{\frac{1}{4}} \nu^2 + \e \Lambda^{\frac{1}{4}}$ with \cref{nu} and the smallness of $\e \Lambda^{\frac{1}{4}}.$
	Then the proof of \cref{Prop-apriori} is finished.
	
\end{proof}

\vspace{.3cm}

\section{Decay Rates}\label{Sec-rate}

Combining \cref{Thm-local} and \cref{Prop-apriori}, for the problem \cref{equ-phizeta}, \cref{ic-pertur}, there exist $\Lambda_0\geq 1$ and $\e_0>0$ such that, if $\Lambda\geq \Lambda_0$ and $\e \Lambda^{\frac{1}{4}} \leq \e_0,$ then the global existence of the solution $(\phi, \zeta) \in \mathbb{B}(0,+\infty) $ can be proved by a standard argument. 
To finish the proof of \cref{Thm-pert}, it remains to prove \cref{behavior-inf}.

\vspace{.1cm}

Since the estimate \cref{est} holds for $T=+\infty$, then using \cref{ic-phipsi,ic-anti}, one has that
\begin{equation}\label{est*}
	\begin{aligned}
		& \sup_{ t>0} \big( \norm{\Phi, \Psi}_{L^2(\R)}^2 + \norm{\phi, \zeta}_{H^3(\Omega)}^2 \big) \\
		& \qquad + \int_0^{+\infty} \left(\norm{\p_3 \Phi, \p_3 \Zv, \Phi \kappa, \Zv \kappa}_{L^2(\R)}^2  + \norm{\nabla \phi}_{H^2(\Omega)}^2 + \norm{\nabla \zeta}_{H^3(\Omega)}^2\right) dt \lesssim \e \Lambda^{\frac{1}{4}}.
	\end{aligned}
\end{equation}
Then it holds that $\nu^2 \leq \sup_{t>0} \big( \norm{\Phi, \Psi}_{L^2(\R)}^2 + \norm{\phi, \zeta}_{H^3(\Omega)}^2 \big) \lesssim \e \Lambda^{\frac{1}{4}} $ and 
\begin{equation}\label{est**}
	\sup_{t>0} \Ec_1(t) + \int_0^{+\infty} \Dc_1(t) dt \lesssim \e \Lambda^{\frac{1}{4}}.
\end{equation}
Thus, it follows from \cref{exp-md} that
\begin{align*}
	\norm{\phi^\md, \zeta^\md}_{H^1(\Omega)}^2 \lesssim \big( \e \Lambda^{\frac{1}{4}} \big)^{\frac{3}{2}} e^{-\alpha_1 t}.
\end{align*}
This, together with \cref{Sob-inf}, yields that 
\begin{align*}
	\norm{\phi^\md, \zeta^\md}_{L^\infty(\Omega)} \lesssim \nu^{\frac{3}{4}} \norm{\phi^\md, \zeta^\md}_{L^2(\Omega)}^{\frac{1}{4}} \lesssim \big( \e \Lambda^{\frac{1}{4}} \big)^{\frac{1}{2}} e^{-\frac{\alpha_1}{8} t},
\end{align*}
which proves \cref{behavior-inf}\textsubscript{2}.

Now we show \cref{behavior-inf}\textsubscript{1}. It follows from \cref{rel-2,est-3**,est*} that
\begin{equation}\label{ED-1}
	\Ec_1 \lesssim \Dc_1 + \norm{\p_3 \Phi, \p_3 \Zv}_{L^2(\Omega)}^2 +  (t+\Lambda)^{-1} \norm{\Phi \kappa}_{L^2(\R)}^2 + \e^3 \Lambda^{\frac{1}{4}} e^{-\alpha t}.
\end{equation}
With \cref{ED-1,est*,est**}, multiplying $ (t+1) $ on \cref{est-3*} and integrating the resulting inequality with respect to $ t $, one can get that
\begin{align}
	& (t+1) \Ec_1 + \int_0^t (\tau+1) \Dc_1 d\tau \notag \\
	& \qquad \lesssim \Ec_1(0) + \int_0^t \Ec_1 d\tau +  \int_0^t \norm{\p_3\Phi, \p_3 \Zv, \Phi\kappa}_{L^2(\R)}^2 d\tau + (\e \Lambda^{\frac{1}{4}})^{\frac{3}{2}} \notag \\
	& \qquad \lesssim \e \Lambda^{\frac{1}{4}}. \label{decay1}
\end{align}
On the other hand, collecting \cref{est-6,est-3b}, one can choose a linear combination of  $ \Ac_{3,\flat} $ and $\Ac_{i}$ (resp. $ \Bc_{3,\flat} $ and $ \Bc_{i} $) for $ i=1,2,3,4$, denoted by $ \Ec_2 $ (resp. $ \Dc_2 $), such that
\begin{equation}\label{est-6*}
	\begin{aligned}
		\frac{d}{dt} \Ec_2 + \Dc_2 & \lesssim (t+\Lambda)^{-2} \norm{\p_3 \Phi, \p_3 \Zv, \Phi \kappa}_{L^2(\R)}^2 \\
		& \quad + (t+\Lambda)^{-1} \norm{\p_3^2 \Phi, \p_3^2 \Zv}_{L^2(\R)}^2 + (\e \Lambda^{\frac{1}{4}})^{\frac{3}{2}} e^{-\alpha t},
%		& \lesssim \delta (t+1)^{-2} \norm{\p_3 \Phi, \p_3 \Zv, \Phi \kappa}_{L^2(\R)}^2 + \Dc_1 + \e \nu e^{-\alpha t},
	\end{aligned}
\end{equation}
and
\begin{equation}\label{est-6**}
	\begin{aligned}
		\pm\Ec_2 & \lesssim \pm \Big( \norm{\p_3^2 \Phi, \p_3^2 \Zv}_{L^2(\R)}^2 + \norm{\nabla^2 \phi, \nabla^2 \zeta}_{H^1(\Omega)}^2 \Big) \\
		& \quad  + (t+\Lambda)^{-1} \norm{\p_3 \Phi, \Phi\kappa}_{L^2(\R)}^2 + \big( \e \Lambda^{\frac{1}{4}} \big)^{\frac{3}{2}} e^{-\alpha t}, \\
%		& \lesssim \Dc_3 + \delta (t+1)^{-1}  \int_\R \Phi^2 \kappa^2 dx_3 + \delta (t+1)^{-2} \norm{\p_3 \Phi}_{L^2(\R)}^2 + \e e^{-\alpha t}, \\
 		\Dc_2 & \sim  \norm{\p_3^3 \Zv}_{L^2(\R)}^2  + \norm{\nabla^2 \phi,\nabla^3 \zeta}_{H^1(\Omega)}^2.
	\end{aligned}
\end{equation}
It follows from \cref{rel-2} and \cref{est-3**}\textsubscript{2} that
\begin{align*}
	\Ec_2 & \lesssim \Dc_1 + (t+\Lambda)^{-1} \norm{\p_3 \Phi, \Phi\kappa}_{L^2(\R)}^2 + \big( \e \Lambda^{\frac{1}{4}} \big)^{\frac{3}{2}} e^{-\alpha t}.
\end{align*}

Then multiplying $ (t+1)^2 $ on \cref{est-6*} and integrating the resulting inequality with respect to $ t $, one has that
\begin{align}
	& (t+1)^2 \Ec_2 + \int_0^t (\tau+1)^2 \Dc_2 d\tau \notag \\
	& \quad \lesssim \Ec_2(0) + \int_0^t (\tau+1) \Dc_1 d\tau  + \int_0^t \norm{\p_3 \Phi, \p_3 \Zv, \Phi\kappa}_{L^2(\R)}^2 d\tau + (\e \Lambda^{\frac{1}{4}})^{\frac{3}{2}} \notag \\
	&\quad \lesssim \e\Lambda^{\frac{1}{4}}. \label{decay2}
\end{align}
Note that $ \norm{\kappa}_{L^\infty(\R)}^2 \lesssim (t+\Lambda)^{-1}. $
Then it follows from \cref{est*,ED1-eqv,decay1} that
\begin{align}
	\norm{\phi^\od,\zeta^\od}_{L^2(\R)}^2 & \lesssim \Ec_1 + (t+\Lambda)^{-1} \norm{\Phi}_{L^2(\R)}^2 + \e^3 \Lambda^{\frac{1}{4}} e^{-\alpha t} \notag \\
	& \lesssim \e \Lambda^{\frac{1}{4}} (t+1)^{-1}. \label{decay3}
\end{align}
In addition, it follows from \cref{rel-1,decay3,est-6**,decay2} that
\begin{align}
	\norm{\p_3 \phi^\od,\p_3 \zeta^\od}_{L^2(\R)}^2 & \lesssim \Ec_2+ (t+\Lambda)^{-1} \norm{\phi^\od}_{L^2(\R)}^2 + (t+\Lambda)^{-2} \norm{\Phi}_{L^2(\R)}^2 + \big( \e \Lambda^{\frac{1}{4}} \big)^{\frac{3}{2}} e^{-\alpha t} \notag \\
	& \lesssim \e \Lambda^{\frac{1}{4}} (t+1)^{-2}. \label{decay4}
\end{align}
Combining \cref{decay3,decay4}, one can get that
%\begin{equation}\label{decay-phizeta}
%	\begin{aligned}
%		\norm{\phi,\zeta}_{L^2(\Omega)}^2 \lesssim \e^2 (t+\Lambda)^{-1} \andd
%		\norm{\nabla\phi,\nabla\zeta}_{L^2(\Omega)}^2  \lesssim \e^2 (t+\Lambda)^{-2}.
%	\end{aligned}
%\end{equation}
%Then one has that 
\begin{align*}
	\norm{\phi^\od, \zeta^\od}_{L^\infty(\R)} & \lesssim \norm{\p_3\big(\phi^\od, \zeta^\od\big)}_{L^2(\R)}^{\frac{1}{2}} \norm{\phi^\od, \zeta^\od}_{L^2(\R)}^{\frac{1}{2}} \lesssim (\e \Lambda^{\frac{1}{4}})^\frac{1}{2} (t+1)^{-\frac{3}{4}},
\end{align*}
which yields \cref{behavior-inf}\textsubscript{1}.
The proof of \cref{Thm-pert} is finished.

\vspace{.3cm}

	\textbf{Proof of \cref{Thm}.}
	Once \cref{Thm-pert} is proved, it remains to prove  \cref{behavior} to complete the proof of \cref{Thm}. 
	In fact, using \cref{exp-diff,exp-diff-u,behavior-inf}, one has that
	\begin{align*}
		\norm{(\rho^\od, \uv^\od) - (\rho^\vs, \uv^\vs)}_{L^\infty(\R)} & \lesssim \norm{(\rhot^\od, \uvt^\od) - (\rho^\vs, \uv^\vs)}_{L^\infty(\R)} + \norm{(\phi^\od, \zeta^\od)}_{L^\infty(\R)} \\
		& \lesssim (\e \Lambda^{\frac{1}{4}})^\frac{1}{2} (t+\Lambda)^{-\frac{3}{4}}.
	\end{align*}
	To prove \cref{behavior}\textsubscript{2}, note that $ (\rho^\md, \uv^\md) = (\rhot^\md, \uvt^\md) +  (\phi^\md, \zeta^\md). $ 
	Note that if $ h(x) \in L^\infty(\R^3) $ is periodic on $ \Torus^3 $ with the average $ \bar{h} = \int_{\Torus^3} h(x) dx, $ then it holds that $ h^\md = (h-\bar{h})^\md $.
	Then it follows from \cref{ansatz,uvt-1} that
	\begin{align*}
		\rhot^\md 
%		& = \frac{1}{2 } \big[\rho^\md_- (1-\Theta_\sigma) + \rho_+^\md (1+\Theta_\sigma)\big] \\
		& = \frac{1}{2 } \big[v_-^\md  (1-\Theta_\sigma) + v_+^\md (1+\Theta_\sigma)\big], \\
		\uvt^\md & = \frac{1}{2 } \big[ \zv_-^\md (1-\Theta_\sigma) + \zv_+^\md (1+\Theta_\sigma)\big] \\
		& \quad + \Big[ \frac{1}{\rhot} (v_+ - v_-) (\uv_+ - \uv_-) \Big]^\md (1-\Theta^2_\sigma),
	\end{align*}
	where $ v_\pm $ and $ \zv_\pm $ defined in \cref{vw-pm}.
	Thus, it follows from \cref{exp-per} that
	\begin{equation}\label{exp-rho}
		\norm{\rhot^\md,\uvt^\md}_{L^\infty(\Omega)} \lesssim \norm{v_-, v_+,\zv_-,\zv_+}_{L^\infty(\Omega)} \lesssim \e e^{-\alpha t}.
	\end{equation}
	This, together with \cref{behavior-inf}$ _2 $, yields that
	\begin{align*}
		\norm{(\rho^\md, \uv^\md)}_{L^\infty(\Omega)} & \lesssim \norm{(\rhot^\md, \uvt^\md)}_{L^\infty(\Omega)} + \norm{(\phi^\md, \zeta^\md)}_{L^\infty(\Omega)} \lesssim \big( \e \Lambda^{\frac{1}{4}} \big)^{\frac{1}{2}} e^{-\frac{\alpha_1}{8} t}.
	\end{align*}
	The proof of \cref{Thm} is finished.

\vspace{.3cm}

\section{Localized Perturbations}\label{Sec-local}

In this section, we outline the main ideas to solve the problem \cref{NS} with the initial data \cref{ic-loc}.
In order to use the anti-derivative technique, we first follow \cite{Liu-Xin,HXY} to introduce some \textit{diffusion waves} propagating along the transverse characteristics to carry the excessive mass of the perturbation $ (\phib_0,\psib_0). $

As the viscosity in \cref{system-1d} vanishes, the resulting hyperbolic system takes the form,
\begin{equation}
	\p_t \Big(\begin{array}{c}
			\rho\\
			\mv\end{array}
	\Big) + A(\rho,\uv) \p_3 \Big(\begin{array}{c}
			\rho\\
			\mv\end{array}\Big) = 0,
\end{equation}
where
\begin{equation}
	A(\rho,\uv) = \left(\begin{array}{cccc}
			0 & 0 & 0 & 1 \\
			-u_1 u_3 & u_3 & 0 & u_1 \\
			-u_2 u_3 & 0 & u_3 & u_2 \\
			p'(\rho)-u_3^2 & 0 & 0 & 2u_3
		\end{array} \right).
\end{equation}
This matrix has four real eigenvalues,
\begin{equation}
	\lambda_0 = u_3 - \sqrt{p'(\rho)}, \quad \lambda_1 = \lambda_2 = u_3, \quad \lambda_3 = u_3 + \sqrt{p'(\rho)},
\end{equation}
with the associated right eigenvectors,
\begin{equation}
	\rv_0(\rho,\uv) = \left(\begin{array}{c}
			1 \\
			u_1 \\
			u_2 \\
			\lambda_0
		\end{array}\right), \quad \rv_1 = \left( \begin{array}{c}
			0 \\
			1 \\
			0\\
			0
		\end{array}\right), \quad \rv_2 = \left( \begin{array}{c}
			0 \\
			0 \\
			1 \\
			0
		\end{array}\right), \quad \rv_3(\rho,\uv) = \left(\begin{array}{c}
			1 \\
			u_1 \\
			u_2 \\
			\lambda_3
		\end{array}\right).
\end{equation}
For the constant states $ (\rhob, \pm\uvb) $ in \cref{vs}, we denote
\begin{align*}
	\lambda_i^\pm := \lambda_i(\rhob, \pm\uvb) \andd \rv_i^\pm := \rv_i(\rhob, \pm\uvb) \quad \text{for } i=0,3.
\end{align*}
\textit{Diffusion waves}. Define $ \vartheta(x_3,t) = \frac{1}{2\sqrt{\pi(t+\Lambda)}} \exp\big\{-\frac{x_3^2}{4(t+\Lambda)}\big\} $ and
\begin{equation}\label{ker-0-3}
	\begin{aligned}
		\vartheta_-(x_3,t) & := \vartheta(x_3 -\lambda_0^- (t+\Lambda), t), \\
		\vartheta_+(x_3,t) & := \vartheta(x_3 -\lambda_3^+ (t+\Lambda), t).
	\end{aligned}
\end{equation}
%which satisfies that $ \p_t \vartheta = \p_3^2 \vartheta $ and $ \int_\R \vartheta(x_3,t) dx_3 \equiv 1. $ 
%And define
%\begin{equation}\label{ker-0-3}
%	\vartheta_-(x_3,t) := \vartheta(x_3-\lambda_0^-(t+1), t) \andd \vartheta_+(x_3,t) := \vartheta(x_3-\lambda_3^+(t+1), t),
%\end{equation}
Then it holds that
\begin{align*}
	\p_t \vartheta_- + \lambda_0^- \p_3 \vartheta_- = \p_3^2 \vartheta_-, \qquad \p_t \vartheta_+ + \lambda_3^+ \p_3 \vartheta_+ = \p_3^2 \vartheta_+,
\end{align*}
and
\begin{align*}
	\int_\R \vartheta_-(x_3,t) dx_3 = \int_\R \vartheta_+(x_3,t) dx_3 = 1 \qquad \forall t\geq 0.
\end{align*}

Denote
\begin{equation}
	\epsilon := \norm{\phib_0, \psib_0}_{L^2_{3}(\Omega)} + \norm{\nabla \phib_0, \nabla \psib_0}_{H^2(\Omega)}.
\end{equation}
Decompose the mass, $ \int_\Omega (\phib_0, \psib_0) dx, $ into
\begin{equation}\label{alpha-i}
	\int_\Omega (\phib_0, \psib_0) dx =  \alpha_0 \rv_0^- + \alpha_1 \rv_1 + \alpha_2 \rv_2 + \alpha_3 \rv_3^+,
\end{equation}
where each $ \alpha_i \in \R $ is a constant, satisfying that
\begin{align*}
	\max_{i=0,\cdots,3} \abs{\alpha_i} \lesssim \norm{\phib_0, \psib_0}_{L^1(\Omega)} \lesssim \norm{\phib_0, \psib_0}_{L^2_{3}(\Omega)} \lesssim \epsilon.
\end{align*}

\textit{Ansatz.} In the case for the $H^3(\Omega)$-perturbations, we set the ansatz as
\begin{equation}\label{ansatz-loc}
	\begin{aligned}
		(\rhot,\mvt)(x_3,t) & := (\rho^\vs, \mv^\vs)(x_3,t) + \alpha_0  \rv_0^- \vartheta_-(x_3,t) + \alpha_3\rv_3^+  \vartheta_+(x_3,t) \\
		& \qquad +  (\alpha_1 \rv_1 + \alpha_2 \rv_2) \vartheta(x_3,t),
	\end{aligned}
\end{equation}
and 
\begin{equation*}
	\uvt(x_3,t) := \frac{\mvt(x_3,t)}{\rhot(x_3,t)}.
\end{equation*}
Equivalently, \cref{ansatz-loc} reduces to
\begin{equation}\label{ansatz-loc-1}
	\begin{aligned}
		\rhot & = \rhob + \alpha_0 \vartheta_- + \alpha_3 \vartheta_+, \\
		\mt_i & = \rhob u_i^\vs + \ub_i (\alpha_3 \vartheta_+ - \alpha_0 \vartheta_-) + \alpha_i \vartheta, \qquad i=1,2, \\
		\mt_3 & = \lambda_0^- \alpha_0 \vartheta_- + \lambda_3^+ \alpha_3 \vartheta_+.
	\end{aligned}
\end{equation}
\begin{Rem}
	The construction \cref{ansatz-loc} is different from \cite{HXY}, since we do not shift the background vortex sheet to cancel the excessive mass in the linearly degenerate characteristic fields. 
	In fact, we find that the shift is not essential for the contact discontinuities, since after shifting, the difference between two contact discontinuities decays just same as the diffusion waves.
	However, it should be mentioned that the shifts are an essential point in the stability theory of shocks.
\end{Rem}
It follows from \cref{ansatz-loc-1} that
\begin{equation}\label{equ-ansatz}
	\begin{cases}
		\p_t \rhot + \dv \mvt  = \p_3 F, \\
		\p_t \mvt + \dv \big(\frac{\mvt\otimes\mvt}{\rhot}\big) + \nabla p(\rhot) = \mu \lap \uvt + (\mu+\lambda) \nabla\dv \uvt + \p_3 \Gv,
	\end{cases}
\end{equation}
where $ F=F(x_3,t)\in\R $ and $ \Gv =\Gv(x_3,t) = (G_1,G_2,G_3)(x_3,t) \in \R^3 $ are given by
\begin{equation}\label{F-G}
	\begin{aligned}
		F & =  \alpha_0 \p_3 \vartheta_- + \alpha_3 \p_3 \vartheta_+, \\
		G_i & = \alpha_i \p_3 \vartheta + \Big[ \frac{\mt_3 \mt_i}{\rhot} + \ub_i \big( \alpha_0 \lambda_0^- \vartheta_- - \alpha_3 \lambda_3^+ \vartheta_+ \big) \Big] \\
		& \quad - \ub_i \big( \alpha_0 \p_3 \vartheta_- - \alpha_3 \p_3 \vartheta_+ \big) - \mu \p_3 (\ut_i - u_i^\vs), \qquad i=1,2, \\
		G_3 & = \frac{\mt_3^2}{\rhot} + \big[p(\rhot) - p(\rhob) - p'(\rhob) (\rhot - \rhob)\big] - \mut \p_3 \ut_3 \\
		& \quad + \lambda_0^- \alpha_0 \p_3\vartheta_- + \lambda_3^+ \alpha_3 \p_3 \vartheta_+.
	\end{aligned}
\end{equation}
Here we have used the fact that $p'(\rhob) = \abs{\lambda_0^-}^2 = \abs{\lambda_3^+}^2.$

In the following part of this section, we set 
\begin{equation}\label{Re}
	\Rc = e^{- \frac{c_0 \abs{x_3}^2}{t+\Lambda} } + e^{- \frac{c_0 \abs{x_3-\lambda_0^-(t+\Lambda)}^2}{t+\Lambda}} + e^{ - \frac{c_0 \abs{x_3-\lambda_3^+(t+\Lambda)}^2}{t+\Lambda}}
\end{equation} 
for some generic constant $ c_0>0 $ .

\begin{Lem}
	For $ j=0,1,2, \cdots, $ it holds that
	\begin{equation}\label{bdd-F-G}
		\abs{\p_3^j F} + \abs{\p_3^j \Gv} \lesssim \epsilon(t+\Lambda)^{-1-\frac{j}{2}} \Rc.
	\end{equation}
\end{Lem}
\begin{proof}
	We prove \cref{bdd-F-G}  for $ j=0 $ only, since the case for higher orders is similar.
	For convenience, we write $ A \cong B $ if $ \abs{A-B} \lesssim \epsilon(t+\Lambda)^{-1} \Rc $.
	First, it is direct to show that $ \abs{\p_3 \vartheta} + \abs{\vartheta}^2 \cong 0, $
	which yields that $ \abs{F} + \abs{G_3} \cong 0. $ To prove that $ G_i \cong 0 $ for $ i=1,2, $ it suffices to show that
	\begin{align*}
		I := \frac{\mt_3 \mt_i}{\rhot} - \ub_i \big( \alpha_3 \lambda_3^+ \vartheta_+ - \alpha_0 \lambda_0^- \vartheta_- \big) \cong 0.
	\end{align*}
	In fact, it follows from \cref{ansatz-loc-1} that
	\begin{align*}
		I & \cong \ub_i \theta \big(\lambda_0^- \alpha_0 \vartheta_- + \lambda_3^+ \alpha_3 \vartheta_+\big) + \ub_i \big( \alpha_0 \lambda_0^- \vartheta_- - \alpha_3 \lambda_3^+ \vartheta_+ \big) \\
		& = \ub_i \big[ \alpha_0 \lambda_0^- \vartheta_- (\theta+1) + \alpha_3 \lambda_3^+ \vartheta_+ (\theta-1) \big].
	\end{align*}
	Note that $ \lambda_0^- = -\sqrt{p'(\rhob)}<0. $ Thus, if $ x_3 \geq 0, $ then
	\begin{align*}
		\vartheta_- (\theta+1) & \lesssim \vartheta_- \lesssim e^{-\frac{\abs{\lambda_0^-}^2}{8}(t+\Lambda) } e^{- \frac{ \abs{x_3-\lambda_0^-(t+\Lambda)}^2}{8(t+\Lambda)}}.
	\end{align*}
	If $ x_3 <0, $ then using \cref{bdd-theta} and the fact that $ \lambda_0^-<0, $ it holds that
	\begin{align*}
		0<\vartheta_- (\theta+1) & = \vartheta_- \int_{-\infty}^{x_3} \p_{y_3} \theta(y_3,t) dy_3
%		& \lesssim  \vartheta_- \int_{-\infty}^{x_3} \kappa(y_3,t) dy_3 \\
		\lesssim e^{- \frac{ \abs{x_3-\lambda_0^-(t+\Lambda)}^2}{4(t+\Lambda)} -\frac{\rhob x_3^2 }{16\mu(t+\Lambda)}} \lesssim e^{- c_1 (t+\Lambda)} e^{ -\frac{\rhob x_3^2 }{32\mu(t+\Lambda)}},
	\end{align*}
	for some $ c_1>0 $ which depends only on $ \abs{\lambda_0^-} $ and $ \frac{\rhob}{\mu}. $ 
%	Here we have used the fact that
%	for any constants $c_1>0, c_2>0$ and $a_1, a_2$ with $a_1\neq a_2$, there exist 
%	\begin{align*}
%		e^{-\frac{c_1\abs{x_3-a_1(t+\Lambda)}}{t+\Lambda} - \frac{c_2\abs{x_3-a_2(t+\Lambda)}}{{t+\Lambda}}} \lesssim e^{-c_3 (t+\Lambda) - c_3 \abs{x_3} }
%	\end{align*}
	Similarly, there exists a constant $ c_2>0 $, depending on $ \lambda_3^+ $ and $ \frac{\rhob}{\mu}, $ such that
	\begin{align*}
		0<\vartheta_+ (1-\theta) \lesssim e^{-c_2 (t+\Lambda)} e^{-\frac{\rhob x_3^2 }{32\mu(t+\Lambda)} }, \qquad x_3 \in\R, t\geq 0. 
	\end{align*}
\end{proof}
It follows from \cref{equ-ansatz,NS} that $ \frac{d}{dt} \big\{\int_\Omega [ (\rho,\mv) - (\rhot, \mvt)] dx \big\} = 0. $ This, together with \cref{alpha-i}, yields that
\begin{equation}\label{zero-mass*}
	\int_\Omega [ (\rho,\mv) - (\rhot, \mvt)] dx \equiv \int_\Omega [ (\rho,\mv) - (\rhot, \mvt)]|_{t=0} dx = 0.
\end{equation}
Using the same notations \cref{phipsi,Phi-Psi,Zv,Q}, the systems for $ (\phi,\zeta) $  and $ (\Phi,\Psi) $ are almost same as \cref{equ-PhiPsi,equ-phizeta}, respectively, except the source terms on the right-hand sides, in which the new ones satisfy \cref{bdd-F-G} instead of \cref{exp-f-g}. 

Denote 
\begin{align*}
	(\phi_0,\psi_0)(x)  & := (\phi,\psi)(x,0)  \\
	 &  = (\phib_0,\psib_0)(x)+ \alpha_0 \rv_0^- \vartheta_-(x_3,0) + \alpha_3\rv_3^+  \vartheta_+(x_3,0) +  (\alpha_1 \rv_1 + \alpha_2 \rv_2) \vartheta(x_3,0),\\
	 (\Phi_0,\Psi_0)(x_3) & := (\Phi,\Psi)(x_3,0) \\
	  & = \int_{-\infty}^{x_3} (\phi_0^\od,\psi_0^\od)(y_3) dy_3 = - \int_{x_3}^{+\infty} (\phi_0^\od,\psi_0^\od)(y_3) dy_3.
\end{align*}
Then it follows from \cref{small*} that $(\phi_0,\psi_0)\in H^3(\Omega)$ and $(\Phi_0,\Psi_0) \in L^2(\R).$ In fact,
\begin{align}
	\norm{\Phi_0}_{L^2(\R)}^2 & = \int_{-\infty}^0 \Big(\int_{-\infty}^{x_3} \phi_0^\od(y_3) dy_3\Big)^2 dx_3 + \int_{0}^{+\infty} \Big(\int_{x_3}^{+\infty} \phi_0^\od(y_3) dy_3\Big)^2 dx_3 \notag \\
	& \lesssim \norm{\phib_0}_{L^2_{3}(\Omega)}^2 \int_0^{+\infty} \int_{x_3}^{+\infty} (y_3^2+1)^{-\frac{3}{2}} dy_3  dx_3  + \epsilon \notag \\
	& \lesssim \epsilon. \label{L2-Anti}
\end{align}
The $L^2$-estimate of $\Psi_0$ is similar.

\vspace{.2cm}

Now we establish the a priori estimates.
First, we assume a priorily that for a fixed $ T>0, $
\begin{align*}
	(\Phi,\Psi) \in C(0,T; L^2(\R)) \andd (\phi,\zeta) \in C(0,T; H^3(\Omega)).
\end{align*}
Denote
\begin{equation}\label{nu*}
	\nu:= \sup\limits_{t\in(0,T)} \big(\norm{ \Phi,\Psi}_{L^\infty(\R)} + \norm{ \phi,\psi}_{H^3(\Omega)} \big).
\end{equation}
\begin{Rem}
	The definition of $ \nu $ here is different from \cref{nu}. In fact, same as in \cite{HXY}, due to the diffusion waves, the $ L^2 $-norm of the anti-derivatives, $\norm{\Phi,\Psi}_{L^2(\R)},$ grows actually at the rate $ (t+1)^{\frac{1}{4}} $ . 
	Nevertheless, we will show that the $L^2$-norm of their derivatives, $ \norm{\p_3 \Phi,\p_3 \Psi}_{L^2(\R)}, $ decays at the rate $ (t+1)^{-\frac{1}{4}}, $ which ensures the uniform boundedness of the $L^\infty$-norm, $\norm{\Phi,\Psi}_{L^\infty(\R)}.$
\end{Rem}
It is noted that in the anti-derivative system, the estimates of $ \Qv_i $ for $ i=1,2,3,4 $ are slightly different from \cref{est-Q}, which read
\begin{equation}\label{Q*}
	\begin{aligned}
		\abs{\Qv_1 } + \abs{\Qv_3 } & \lesssim \abs{\phi}^2 + \abs{ \psi}^2, \\
		\abs{\p_3 \Qv_1 } + \abs{\p_3 \Qv_3 } & \lesssim  (t+\Lambda)^{-\frac{1}{2}} \big(\abs{\phi}^2 + \abs{\psi_3}^2 \big) + \big(\abs{\p_3\phi} + \abs{\p_3\psi} \big) \left(\abs{\phi} + \abs{\psi}\right), \\
		\abs{\nabla^k \Qv_2 } +\abs{\nabla^k \Qv_4 }  & \lesssim \sum_{j=0}^{k}  \epsilon(t+\Lambda)^{-\frac{k-j+1}{2}} \big(\abs{\nabla^j \phi} + \abs{\nabla^j \psi} \big), \qquad k=0,1.
	\end{aligned}
\end{equation}
Besides, the estimates \cref{rel-0,rel-1,rel-2,rel-3,rel-4} still hold true  with $\e=0 $.

\vspace{.1cm} 

\begin{Lem}\label{Lem-1*}
	If $ \Lambda^{-1}, \epsilon $ and $ \nu $ are small enough, then 
	\begin{align}
		& \frac{d}{dt} \Ac_{0,\od} + c_{0,\od} \Bc_{0,\od} \lesssim \nu \norm{\nabla\phi^\md, \nabla\zeta^\md}_{L^2(\Omega)}^2 + \epsilon (t+1)^{-1} \norm{\Phi,\Zv}_{L^2(\R)}^2 + \epsilon (t+1)^{-\frac{1}{2}}, \label{A0} \\
		& \frac{d}{dt} \big(\Ac_{1,\flat} + \frac{\mut}{4\rhob} \Ac_{2,\flat}\big) + c_{1,\flat} \Bc_{1,\flat} + \frac{\mut}{4\rhob} c_{2,\flat} \Bc_{2,\flat} \notag \\
		& \qquad \lesssim (\epsilon+\nu) \big( \norm{\p_3^3 \Zv}_{L^2(\R)}^2 +  \norm{\nabla \zeta^\md}_{L^2(\Omega)}^2 +  \norm{\nabla^2 \phi}_{H^1(\Omega)}^2 \big) \notag \\
		& \qquad \quad  + (t+1)^{-1} \norm{\p_3 \Phi, \p_3 \Zv, \Phi \kappa}_{L^2(\R)}^2 + \epsilon (t+1)^{-\frac{3}{2}}, \label{A1&2} \\
		&  \frac{d}{dt} \Ac_{3,\flat} + c_{3,\flat} \Bc_{3,\flat} \notag \\
		& \qquad \lesssim (\epsilon+\nu) \big(\norm{\nabla^2 \phi}_{H^1(\Omega)}^2 +  \norm{\nabla^3 \zeta}_{L^2(\Omega)}^2\big) + (t+1)^{-1} \norm{\p_3^2 \Phi, \p_3^2 \Zv}_{L^2(\R)}^2 \notag \\
		& \qquad \quad + (t+1)^{-2} \norm{\p_3 \Phi, \p_3 \Zv, \Phi \kappa}_{L^2(\R)}^2 + \epsilon (1+t)^{-\frac{5}{2}}, \label{A3}
	\end{align}
	where $ \Ac_{i,\od} $ and $ \Bc_{i,\od} $ for $ i=0,\cdots, 3, $ denote the same notations as in \cref{Sec-od}, satisfying that
	\begin{align*}
		& \Ac_{0,\od} \sim \norm{\Phi}_{H^1(\R)}^2 + \norm{\Zv}_{L^2(\R)}^2, \quad \Bc_{0,\od} \sim \norm{\p_3\Phi,\p_3\Zv, \Phi\kappa,\Zv\kappa}_{L^2(\R)}^2, \\
		& \Ac_{1,\od}+ \frac{\mut}{4\rhob}\Ac_{2,\od} \sim \norm{\p_3 \Phi}_{H^1(\R)}^2 + \norm{\p_3\Zv}_{L^2(\R)}^2, \quad \Bc_{1,\od} + \frac{\mut}{4\rhob}\Bc_{2,\od} \sim \norm{\p_3^2\Phi,\p_3^2\Zv}_{L^2(\R)}^2, \\
		& \Ac_{3,\od} \sim \norm{\p_3^2 \Phi, \p_3^2\Zv}_{L^2(\R)}^2, \quad \Bc_{3,\od} \sim \norm{\p_3^3 \Zv}_{L^2(\R)}^2.
	\end{align*}
\end{Lem}

\begin{proof}
The proof is almost same as those of Lemmas \ref{Lem-est1}--\ref{Lem-est6}.
Indeed, the only different ingredients are the estimates related to the terms $ (F,\Gv) $ and $ (\Qv_2, \Qv_4). $ 

For instance, in \cref{tool-G,tool-Q}, we use the different estimates, 
	\begin{align*}
		\norm{F,\Gv}_{L^2(\R)} \norm{\Phi, \Zv}_{L^2(\R)} & \lesssim \epsilon (t+1)^{-\frac{3}{4}} \norm{\Phi, \Zv}_{L^2(\R)} \\
		& \lesssim \epsilon (t+1)^{-1} \norm{\Phi, \Zv}_{L^2(\R)}^2 + \epsilon (t+1)^{-\frac{1}{2}},
	\end{align*}
	and
	\begin{align*}
		\norm{\Qv_2, \Qv_4}_{L^2(\Omega)} \norm{\Zv}_{H^1(\R)} &  \lesssim \epsilon (t+1)^{-1} \norm{\Zv}_{H^1(\R)}^2 + \epsilon \norm{\phi,\psi}_{L^2(\Omega)}^2  \\
		& \lesssim \epsilon (t+1)^{-1} \norm{\Zv}_{L^2(\R)}^2 + \epsilon\norm{\p_3\Phi,\p_3 \Zv,\Phi\kappa}_{L^2(\R)}^2 + \epsilon\norm{\nabla\phi^\md,\nabla\zeta^\md}_{L^2(\Omega)}^2;
	\end{align*}
	and in the estimate of \cref{I8}, we use the different estimates,
	\begin{align*}
		\norm{\p_3^2 F, \p_3^2 \Gv}_{L^2(\R)} \norm{\p_3^2 \Phi,\p_3^2\Zv}_{L^2(\R)} & \lesssim \epsilon (t+1)^{-1} \norm{\p_3^2 \Phi,\p_3^2\Zv}_{L^2(\R)}^2 + \epsilon (t+1)^{-\frac{5}{2}}, \\
		\norm{\p_3 \Qv_2, \p_3 \Qv_4}_{L^2(\Omega)} \norm{\p_3^3 \Zv}_{H^1(\R)} &  \lesssim \epsilon \norm{\p_3^3 \Zv}_{H^1(\R)}^2 + \epsilon (t+1)^{-2} \norm{\phi,\psi}_{L^2(\Omega)}^2 \\
		&   \quad + \epsilon (t+1)^{-1} \norm{\nabla \phi, \nabla \psi}_{L^2(\Omega)}^2.
	\end{align*}
%	where $ I_{8,1} $ and $ I_{8,2} $ are the same as in \cref{I8-12}, 
%The remaining proofs are exactly the same as in \cref{Sec-od}.
\end{proof}

\begin{Lem}\label{Lem-md*}
	If $ \Lambda^{-1}, \epsilon $ and $ \nu $ are small enough, then 
	\begin{equation}\label{est-2}
		\begin{aligned}
			& \frac{d}{dt} \Ac_{\neq} + \Bc_{\neq} \leq 0,
		\end{aligned}
	\end{equation}
	where  $ \Ac_\md $ and $ \Bc_\md $ are two energy functionals, satisfying that
	\begin{align*}
		\Ac_{\neq} \sim \norm{\phi^\md, \zeta^\md}_{H^1(\Omega)}^2, 
		%		M_1' \Ac_{1,\sharp} + M_2' \Ac_{2,\sharp} + \Ac_{3,\sharp} & \gtrsim \norm{\phi^\md, \psi^\md}_{H^1(\Omega)}^2 - \e \nu e^{-\alpha t}, 
		\qquad
		\Bc_{\neq} \sim \norm{\nabla \phi^\md}_{L^2(\Omega)}^2 + \norm{\nabla \zeta^\md}_{H^1(\Omega)}^2.
	\end{align*}
\end{Lem}

\begin{proof}
	Note that the ansatz \cref{ansatz-loc} depends only on $x_3$, and so are the error terms $F$ and $\Gv$, which implies that $ (\p_3 F)^\md = 0 $ and $ (\p_3 \Gv)^\md =0 $. Then the proof is similar as in \cref{Sec-md}.
\end{proof}

\begin{Lem}
	If $ \Lambda^{-1}, \epsilon $ and $ \nu $ are small enough, then 
	\begin{equation}\label{est-tor}
		\begin{aligned}
			\frac{d}{dt} \Ac + c \Bc & \lesssim  \norm{\p_3^3 \Zv}_{L^2(\R)}^2 + \norm{ \nabla^2 \zeta^\md}_{L^2(\Omega)}^2  + (t+1)^{-1} \norm{\p_3^2\Phi,\p_3^2 \Zv}_{L^2(\R)}^2 \\
			& \quad + (t+1)^{-2} \norm{\p_3 \Phi, \p_3 \Zv, \Phi \kappa}_{L^2(\R)}^2 + \epsilon (t+1)^{-\frac{7}{2}},
		\end{aligned}
	\end{equation}
	where $ \Ac $ and $ \Bc $ are two energy functionals, satisfying that
	\begin{equation}
		\pm \Ac \lesssim \pm \norm{\nabla^2 \phi, \nabla^2 \zeta}_{H^1(\Omega)}^2 + \norm{\nabla\zeta}_{L^2(\Omega)}^2, \qquad \Bc \sim \norm{\nabla^2 \phi, \nabla^3 \zeta}_{H^1(\Omega)}^2.
	\end{equation}
\end{Lem}

\begin{proof}
	Note that
	\begin{align*}
		\abs{\p_3^j \rhot} + \abs{\p_3^j \mt_3} + \abs{\p_3^j (\mvt_\perp - \mv_\perp^\vs)} \lesssim \epsilon (t+\Lambda)^{-\frac{j+1}{2}}, \qquad j=1,2,\cdots.
	\end{align*}
	The proof is similar to \cref{Sec-orn}.
\end{proof}

%More precisely, the following inequality holds,
%\begin{align}
%	&\begin{aligned}\label{est-tor}
%		& \frac{d}{dt} \Big( \sum_{i=1}^{4} M_{i} \Ac_{i} \Big) + \frac{1}{4} \Big(\sum_{i=1}^{4} c_{i} M_{i} \Bc_{i} \Big) \\
%		& \qquad \lesssim (t+\Lambda)^{-1} \norm{\p_3^2 \Phi, \p_3^2 \Zv }_{L^2(\R)}^2  + (t+\Lambda)^{-2} \norm{\p_3 \Phi, \p_3 \Zv, \Phi \kappa }_{L^2(\R)}^2 \\
%		& \qquad \quad +  \norm{\p_3^3 \Zv}_{L^2(\R)}^2 + \norm{ \nabla^2 \zeta^\md}_{L^2(\Omega)}^2 + \varepsilon(1+t)^{-\frac{5}{2}}.
%	\end{aligned}
%\end{align}

%Multiplying \cref{est-od} by $(1+t)^{-C_0\varepsilon}$ and using the Gronwall's inequality yield,
%\begin{align}
%	\begin{aligned}
%		&\mathcal{A}_{4,\flat}\lesssim \left(\mathcal{A}_{4,\flat}(0)+\varepsilon+ \norm{ \phi_0,\psi_0 }_{H^1(\Omega)}^2  \right)(1+t)^{\frac{1}{2}},\\
%		& \int_0^T\mathcal{B}_{4,\flat}\lesssim \left(\mathcal{A}_{4,\flat}(0)+\varepsilon+ \norm{ \phi_0,\psi_0 }_{H^1(\Omega)}^2  \right)(1+t)^{\frac{1}{2}}.
%	\end{aligned}
%\end{align} 

Now we show the decay rate.
First, it follows from \cref{Lem-md*} that
\begin{equation}\label{est-md}
	\begin{aligned}
		\norm{\phi^\md,\zeta^\md}_{H^1(\Omega)} \lesssim \epsilon e^{-\alpha_1 t}.
	\end{aligned}
\end{equation}
This, together with \cref{A0}, yields that
\begin{equation}\label{est-od}
	\begin{aligned}
		\frac{d}{dt}  \Ac_{0,\od} + c_{\flat}\Bc_{0,\flat} & \leq C_0 \epsilon (t+1)^{-1} \Ac_{0,\od} + C_0 \epsilon (t+1)^{-\frac{1}{2}}.
	\end{aligned}
\end{equation}
By Gronwall's inequality, multiplying $ (t+1)^{-C_0\epsilon} $ on \cref{est-od} yields that if $\epsilon$ is suitably small, then
\begin{equation}\label{AB-od}
	\begin{aligned}
		& \Ac_{0,\od} + \int_0^t \Bc_{0,\od} d\tau \lesssim \epsilon (t+1)^{\frac{1}{2}}.
	\end{aligned}
\end{equation}

Similar to \cref{est-3*}, one can choose a linear combination of $ \Ac_{i,\od} $ for $ i=1,2,3, $ $ \Ac_\md $ and $ \Ac $ (resp. $ \Bc_{i,\od} $ for $ i=1,2,3, $ $ \Bc_\md $ and $ \Bc $), denoted by $ \Ec_1 $ (resp. $ \Dc_1 $), such that
\begin{equation}
	\frac{d}{dt} \Ec_1 + \Dc_1 \lesssim (t+1)^{-1} \norm{\p_3 \Phi, \p_3 \Zv, \Phi\kappa}_{L^2(\R)}^2 + \epsilon (t+1)^{-\frac{3}{2}}.
\end{equation}
and 
\begin{equation}
	\begin{aligned}
		\pm \Ec_1 & \lesssim \pm \left( \norm{\phi}_{H^3(\Omega)}^2 + \norm{\p_3 \Zv}_{H^1(\R)}^2 +  \norm{\zeta^\md}_{H^1(\Omega)}^2 + \norm{\nabla^2 \zeta}_{H^1(\Omega)}^2 \right) \\
		& \quad + (t+1)^{-1} \norm{\Phi\kappa}_{L^2(\R)}^2, \\
		%	\Ac_1 & \gtrsim \norm{\phi, \psi}_{H^1(\Omega)}^2 - \e e^{-\alpha t} - \delta^2 (t+1)^{-1} \norm{\Phi}_{L^2(\R)}^2, 
		\Dc_1 & \sim \norm{\nabla \phi}_{H^2(\Omega)}^2 + \norm{\p_3^2 \Zv}_{H^1(\R)}^2 + \norm{\nabla \zeta^\md}_{H^1(\Omega)}^2 + \norm{\nabla^3 \zeta}_{H^1(\Omega)}^2.
	\end{aligned}
\end{equation}
Then it holds that
\begin{equation}
	(t+1) \Ec_1 + \int_0^t (\tau+1) \Dc_1 d\tau \lesssim \Ec_1(0) +  \int_0^t \Bc_{0,\od} d\tau + \epsilon (t+1)^{\frac{1}{2}},
\end{equation}
which, together with \cref{AB-od}, yields that
\begin{align}\label{1}
	\mathcal{E}_1 \lesssim \epsilon (t+1)^{-\frac{1}{2}},\qquad \int_0^t(\tau+1)\mathcal{D}_1d\tau\lesssim \epsilon (t+1)^{\frac{1}{2}}.
\end{align}

Also, similar to \cref{est-6*,est-6**}, one can choose a linear combination of $ \Ac_{3,\od} $ $ \Ac_\md $ and $ \Ac $ (resp. $ \Bc_{3,\od} $, $ \Bc_\md $ and $ \Bc $), denoted by $ \Ec_2 $ (resp. $ \Dc_2 $), such that
\begin{equation}
	\begin{aligned}
		\frac{d}{dt} \Ec_2 + \Dc_2 & \lesssim  (t+1)^{-1} \norm{\p_3^2 \Phi, \p_3^2\Zv}_{L^2(\R)}^2  \\
		& \quad +  (t+1)^{-2} \norm{\p_3 \Phi, \p_3 \Zv, \Phi \kappa}_{L^2(\R)}^2 + \epsilon (t+1)^{-\frac{5}{2}},
	\end{aligned}
\end{equation}
and
\begin{equation}
	\begin{aligned}
		\pm\Ec_2 & \lesssim \pm \Big( \norm{\p_3^2 \Phi, \p_3^2 \Zv}_{L^2(\R)}^2 + \norm{\nabla^2 \phi, \nabla^2 \zeta}_{H^1(\Omega)}^2 \Big)  + (t+1)^{-1} \norm{\p_3 \Phi, \Phi\kappa}_{L^2(\R)}^2, \\
		%		& \lesssim \Dc_3 + \delta (t+1)^{-1}  \int_\R \Phi^2 \kappa^2 dx_3 + \delta (t+1)^{-2} \norm{\p_3 \Phi}_{L^2(\R)}^2 + \e e^{-\alpha t}, \\
		\Dc_2 & \sim  \norm{\p_3^3 \Zv}_{L^2(\R)}^2  + \norm{\nabla^2 \phi,\nabla^3 \zeta}_{H^1(\Omega)}^2.
	\end{aligned}
\end{equation}
Note that $\Ec_2 \lesssim \Dc_1 + (t+1)^{-2}  \norm{\p_3\Phi, \Phi\kappa}_{L^2(\R)}^2. $
Then one has
\begin{align}\label{2}
	\begin{aligned}
		& (t+1)^2 \mathcal{E}_2 + \int_0^t(\tau+1)^2\mathcal{D}_2 d \tau \\ 
		& \qquad \lesssim \int_0^t (\tau+1) \Dc_1 d\tau + \int_0^t \Bc_{0,\od} d\tau + \epsilon (t+1)^{\frac{1}{2}} \\
		& \qquad \lesssim \epsilon (t+1)^{\frac{1}{2}},
	\end{aligned}
\end{align} 
which yields that
\begin{align}
	\mathcal{E}_2\lesssim \epsilon(t+1 )^{-\frac{3}{2}}.
\end{align}
Then combining \cref{AB-od,1}, one has that
\begin{align*}
	\norm{\Phi,\Psi}_{L^\infty(\R)} \lesssim \norm{\p_3\Phi,\p_3\Psi}_{L^2(\R)}^{\frac{1}{2}} \norm{\Phi,\Psi}_{L^2(\R)}^{\frac{1}{2}} \lesssim \Ec_1^{\frac{1}{4}} \Ac_{0,\od}^{\frac{1}{4}} \lesssim 1.
\end{align*}
Using \cref{rel-0,rel-1} with $\e=0$ yields that
\begin{align*}
	\norm{\phi^\od, \zeta^\od}_{L^2(\R)}^2 
%	& \lesssim \norm{\p_3 \Phi,\p_3 \Zv, \Phi \kappa}_{L^2(\R)}^2 + \norm{\nabla \phi^\md}_{L^2(\Omega)}^2  \\
	\lesssim \Ec_1 + (t+\Lambda)^{-1} \Ac_{0,\od} \lesssim \epsilon (t+1)^{-\frac{1}{2}},
\end{align*}
and
\begin{align*}
	\norm{\p_3 \phi^\od, \p_3 \zeta^\od}_{L^2(\R)}^2 
%	& \lesssim \norm{\p_3^2\Phi, \p_3^2 \Zv}_{L^2(\R)}^2 + (t+1)^{-1} \norm{\p_3 \Phi, \p_3 \Zv, \Phi\kappa}_{L^2(\R)}^2  + \norm{\nabla\phi^\md, \nabla\zeta^\md}_{L^2(\Omega)}^2 \\
	\lesssim \Ec_2 + (t+1)^{-1} \Ec_1 + (t+1)^{-2} \Ac_{0,\od} \lesssim \epsilon (t+1)^{-\frac{3}{2}}.
\end{align*}
Finally, one has that
\begin{align}\label{3}
	\norm{\phi^\od, \zeta^\od}_{L^\infty(\R)} & \lesssim \norm{\p_3\big(\phi^\od, \zeta^\od\big)}_{L^2(\R)}^{\frac{1}{2}} \norm{\phi^\od, \zeta^\od}_{L^2(\R)}^{\frac{1}{2}} \lesssim \epsilon^{\frac{1}{2}} (t+1)^{-\frac{1}{2}}.
\end{align}
Note that the diffusion waves propagating along the transverse characteristics, i.e. $ \vartheta, \vartheta_{\pm} $ in \cref{ansatz-loc}, also decay at the rate $ (t+1)^{-\frac{1}{2}}. $ Hence, one can obtain \cref{behavior-loc}$ _1 $ immediately. The proof of \cref{Thm-loc} is complete.

\vspace{.3cm}

\begin{appendices}

\section{Existence and Decay of Shift Curve} \label{Sec-app1}
\begin{proof}[Proof of \cref{Lem-shift}]
	If $\e$ is small, then it follows from \cref{Lem-per} that the periodic solutions $ (\rho_\pm, \mv_\pm) $ to the problem \cref{NS}, \cref{ic-per} belong to the $ C(0,+\infty; W^{4,\infty}(\Torus^3)) $ space.
	Then the existence and uniqueness of \cref{ode-shift} can be derived from the Cauchy-Lipschitz theorem.
	It remains to show \cref{exp-sigma}. In fact, using \cref{exp-per} and Cauchy's inequality, one has that $ \rho_-^\od = \rhob + O(1) \e e^{-\alpha t} $ and $ m_{3-}^\od = O(1) \e e^{-\alpha t}. $
%	Note that the function $ \theta(x_3,t)= \Theta\big(\frac{x_3}{\srt}\big) $, which solves \cref{equ-theta}, is odd with respect to $ x_3, $ then so is the function $ x_3 \Theta'\big(\frac{x_3}{\sqrt{1+t}}\big) $.
%	Then it holds that
	Then it follows from \cref{DN} that
	\begin{equation}\label{D}
		\mathfrak{D} = 2 \rhob + O(1) \e e^{-\alpha t},
	\end{equation}
	and
	\begin{equation}\label{N}
		\mathfrak{N} =  -\frac{\rhob}{2\srt} \int_\R \frac{x_3-\sigma}{ t+\Lambda} \Theta'\Big(\frac{x_3-\sigma}{\srt}\Big) dx_3  + O(1) \e e^{-\alpha t} = O(1) \e e^{-\alpha t}.
	\end{equation}
	Here in \cref{N} we have used the fact that $ \xi \Theta'(\xi) $ is odd with respect to $ \xi $ (see \cref{Lem-theta}).
	Thus, it holds that $\abs{\sigma'(t)} \lesssim \e e^{-\alpha t}$.
	
	Now we compute the limit of $ \sigma(t) $ as $ t\to +\infty. $ For $ i=1 $ or $ 2, $ it follows from \cref{equ-t}\textsubscript{2} that
	\begin{equation}\label{equ-mt-i}
		\p_t \mt_i^\od + \p_3 \big[ (\ut_3 \mt_i)^\od \big] = \mu \p_3^2 \ut_i^\od + g_i^\od.
	\end{equation}
	For any fixed $ t>0, $ integer $ N>0 $ and constant $ a\in \Torus=(0,1) $, define the bounded domain
	\begin{align*}
		\Omega_{a,N}(t) := \{ (x_3,\tau): \sigma(\tau) - N +a < x_3 < \sigma(\tau) +N +a, \ 0<\tau<t \}.
	\end{align*}
%	Recall that $\sigma(0)=\sigma_0.$
	Integrating \cref{equ-mt-i} over $ B_{a,N}(t) $ yields that
	\begin{equation}\label{eq-1}
		\begin{aligned}
			\int_{\Omega_{a,N}(t)} g_{i}^\od(x_3,\tau) \ dx_3d\tau & = \int_{\sigma(t)-N+a}^{\sigma(t)+N+a} \mt_i^\od(x_3,t) d x_3 - \int_{-N+a}^{N+a} \mt_i^\od(x_3,0) d x_3 \\
			& \quad - \int_0^t \big[ (\ut_3 \mt_i)^\od - \mu \p_3 \ut_i^\od - \sigma'(\tau) \mt_i^\od \big](\sigma(\tau)+N+a, \tau) d\tau \\
			& \quad + \int_0^t \big[ (\ut_3 \mt_i)^\od - \mu \p_3 \ut_i^\od - \sigma'(\tau) \mt_i^\od \big](\sigma(\tau)-N+a, \tau) d\tau \\
			& := I_1 + I_2 + I_3.
		\end{aligned}
	\end{equation}
	Denote $ \xi = \frac{x_3}{\srt}. $ Then it holds that
	\begin{equation}\label{I-1}
		\begin{aligned}
			I_1 & = \frac{1}{2} \int_{-N+a}^{N+a} \big[ w_{i-}^\od(x_3+\sigma,t) (1-\Theta(\xi)) + w_{i+}^\od(x_3+\sigma,t) (1+\Theta(\xi)) \big] dx_3 \\
			& \quad + \mb_i \Big[ \int_{-N+a}^{N+a} \Theta(\xi) dx_3 - \int_{-N+a}^{N+a} \Theta\Big(\frac{x_3}{\sqrt{\Lambda}}\Big) dx_3\Big] \\
			& := I_{1,1} + I_{1,2}.
		\end{aligned}
	\end{equation}
	Since $ w_{i\pm}(\cdot,t) $ have zero averages on $ \Torus^3 $, then
	\begin{align*}
		I_{1,1} & = \frac{1}{2} \Big[ \int_a^{N+a} w_{i-}^\od(x_3+\sigma,t) (1-\Theta(\xi)) dx_3 - \int_{-N+a}^a w_{i-}^\od(x_3+\sigma,t) (1+\Theta(\xi)) dx_3 \\
		& \qquad + \int_{-N+a}^a w_{i+}^\od(x_3+\sigma,t) (1+\Theta(\xi)) dx_3 - \int_a^{N+a} w_{i+}^\od(x_3+\sigma,t) (1-\Theta(\xi)) dx_3 \Big].
	\end{align*}
	By \cref{coin-pm}\textsubscript{2}, one has that $ w_{i+}^\od \equiv w_{i-}^\od + 2 \ub_i v_-^\od. $
	Thus, it holds that
	\begin{align*}
		I_{1,1} = \ub_i \Big[ \int_{-N+a}^a v_-^\od(x_3+\sigma,t) (1+\Theta(\xi)) dx_3 - \int_a^{N+a} v_-^\od(x_3+\sigma,t) (1-\Theta(\xi)) dx_3 \Big],
	\end{align*}
	which yields that
	\begin{align}
		\lim_{N\to +\infty} \int_0^1 I_{1,1} da & = \ub_i \int_0^1 \Big[ \int_{-\infty}^a v_-^\od(x_3+\sigma,t) (1+\Theta(\xi)) dx_3 \notag \\
		& \qquad\qquad\ - \int_a^{+\infty} v_-^\od(x_3+\sigma,t) (1-\Theta(\xi)) dx_3 \Big] da. \label{int-I-1}
	\end{align}
	On the other hand, since $ \Theta(\cdot) $ is an odd function, then the first integral in $I_{1,2}$ satisfies that
	\begin{align*}
		\int_{-N+a}^{N+a} \Theta(\xi) dx_3 & = \int_{N}^{N+a} \Theta(\xi) dx_3 + \int_{-N}^{N} \Theta(\xi) dx_3 - \int_{-N}^{-N+a} \Theta(\xi) dx_3 \\
		& = \int_{N}^{N+a} \Theta(\xi) dx_3 - \int_{-N}^{-N+a} \Theta(\xi) dx_3.
	\end{align*}
	This, together with the dominated convergence theorem, yields that
	\begin{align*}
		\lim_{N\to +\infty} \int_0^1 \int_{-N+a}^{N+a} \Theta(\xi) dx_3 da & = \int_0^1 \Big[  \lim_{N\to +\infty} \int_{N}^{N+a} \Theta(\xi) dx_3 - \lim_{N\to +\infty} \int_{-N}^{-N+a} \Theta(\xi) dx_3 \Big] da \\
		& = 1.
	\end{align*}
	Similarly, one can show that
	\begin{align*}
		\lim_{N\to +\infty} \int_0^1 \int_{-N+a}^{N+a} \Theta\Big(\frac{x_3}{\sqrt{\Lambda}}\Big) dx_3 da = 1.
	\end{align*}
	Thus, it holds that 
	\begin{equation}\label{int-I-12}
		\lim_{N\to +\infty} \int_0^1 I_{1,2} da = 0.
	\end{equation}

	For $ I_2, $ it holds that
	\begin{align*}
		I_2 & = -\int_0^t \big[ (u_{3+} m_{i+})^\od - \mu \p_3 u_{i+}^\od - \sigma'(t) m_{i+}^\od \big](\sigma(\tau)+a, \tau) d\tau \\
		& \quad\ - \int_0^t I_{2,1}^\od(\sigma(\tau)+N+a, \tau) d\tau,
	\end{align*}
	where 
	$I_{2,1} = \big[\ut_3 \mt_i-u_{3+}m_{i+} - \mu \p_3 (\ut_i - u_{i+}) - \sigma'(\tau) (\mt_i-m_{i+})\big], $ which satisfies that $\abs{I_{2,1}^\od(\sigma(\tau)+N+a, \tau)} \lesssim 1-\Theta(\frac{N+a}{\sqrt{\tau+\Lambda}})$. Then one has that
	\begin{align*}
		\lim_{N\to +\infty} \int_0^1 I_2 da & = -\int_0^t \int_0^1 (u_{3+} m_{i+})^\od (\sigma(\tau)+a, \tau) da d\tau + \mb_i(\sigma(t)-\sigma(0)).
	\end{align*}
	Similarly, one can prove that
	\begin{align*}
		\lim_{N\to +\infty} \int_0^1 I_3 da & = \int_0^t \int_0^1 (u_{3-} m_{i-})^\od (\sigma(\tau)+a, \tau) da d\tau + \mb_i(\sigma(t)-\sigma(0)).
	\end{align*}
	Using \cref{coin-pm}\textsubscript{3} and \cref{ic-sigma}, one has that
	\begin{equation}\label{sum-I2-3}
		\lim_{N\to +\infty} \int_0^1 (I_2+I_3) da = -2 \ub_i \int_0^t \int_0^1 m_{3-}^\od(\sigma(\tau)+a, \tau) da d\tau + 2 \mb_i \sigma(t) =  2 \mb_i \sigma(t)  .
	\end{equation}
	By \cref{F1,f2}, one can prove that $g_i^\od \in L^\infty(0,t; L^1(\R)).$ Then it follows from the dominated convergence theorem and \cref{zero-mass-f} that
	\begin{align*}
		\lim_{N\to +\infty} \int_0^1 \int_{B_{a,N}(t)} g_i^\od(x_3,t) dx_3 dt da = \int_0^t \int_\R g_i^\od(x_3,t) dx_3 dt = \int_0^t \int_\R f_{2,i}^\od(x_3,t) dx_3 dt = 0. 
	\end{align*}
	This, together with \cref{sum-I2-3,int-I-1,int-I-12}, yields that
	\begin{align*}
		\sigma(t) & = \frac{1}{2\rhob} \int_0^1 \Big[ \int_a^{+\infty} v_-^\od(x_3+\sigma,t) (1-\Theta(\xi)) dx_3 - \int_{-\infty}^a v_-^\od(x_3+\sigma,t) (1+\Theta(\xi)) dx_3 \Big] da \\
		& = O(1) \e \Lambda^{\frac{1}{2}} e^{-\alpha t}.
	\end{align*}
	The proof is finished.
\end{proof}

\section{Estimates of Error Terms}\label{Sec-app2}
\begin{proof}[Proof of \cref{Lem-F}]
	For convenience, we use the convention $ A \approx B $ to denote that $ \norm{A-B}_{L^2(\Omega)} \lesssim \e \Lambda^{\frac{1}{4}} e^{-\alpha t}. $
	It follows from \cref{uvt-1} that
	\begin{equation}\label{approx-1}
		\begin{aligned}
			\uvt & \approx \frac{1}{2} \big[ \uv_- (1-\Theta_\sigma) + \uv_+ (1+\Theta_\sigma) \big],
		\end{aligned}
	\end{equation}
	which, together with \cref{ansatz}, yields that
	\begin{align}
		\ut_3 \mvt & \approx \frac{1}{4} \big[ u_{3-} \mv_- (1-\Theta_\sigma)^2 + u_{3+} \mv_+ (1+\Theta_\sigma)^2 \big] \notag \\
		& \approx \frac{1}{2} \big[ u_{3-} \mv_- (1-\Theta_\sigma) + u_{3+} \mv_+ (1+\Theta_\sigma) \big]. \label{approx-2}
	\end{align}
	Similarly, one can verify that
	\begin{align}
		& p(\rhot) - \frac{1}{2} p(\rho_-) (1-\Theta_\sigma) - \frac{1}{2} p(\rho_+) (1+\Theta_\sigma) \notag \\ 
		& \qquad = \frac{1}{2} \int_0^1 p'(\rho_- + r (\rhot-\rho_-)) dr (\rhot-\rho_-) (1-\Theta_\sigma) \notag \\
		& \qquad \quad + \frac{1}{2} \int_0^1 p'(\rho_+ + r (\rhot-\rho_+)) dr (\rhot-\rho_+) (1+\Theta_\sigma) \notag \\
		&  \qquad \approx 0. \label{approx-3}
	\end{align}

	1) 	We now prove the estimate of $ \Gv $. By \cref{zero-mass-f,f=0}, one has that for all $ x_3\in\R $ and $ t\geq 0, $
	\begin{align*}
		G_i(x_3,t) & = F_{1,3i}^\od(x_3,t) + \int_{-\infty}^{x_3} f_{2,i}^\od(y_3,t) dy_3 \\
		& = -F_{1,3i}^\od(x_3,t) - \int_{x_3}^{+\infty} f_{2,i}^\od(y_3,t) dy_3 \quad \text{for }\ i=1,2, \\
		G_3(x_3,t) & = F_{1,33}^\od(x_3,t).
	\end{align*}	
	With the aid of \cref{approx-1,approx-2,approx-3} and the fact that 
	\begin{align*}
		\p_3 \uvt & \approx \frac{1}{2} \big[ \p_3 \uv_-  (1-\Theta_\sigma) + \p_3\uv_+ (1+\Theta_\sigma) \big] + \uvb \p_3 \Theta_\sigma, \\
		\dv \uvt &\approx \frac{1}{2} \big[ \dv \uv_-  (1-\Theta_\sigma) + \dv \uv_+ (1+\Theta_\sigma) \big],
	\end{align*} 
	one can use \cref{F1} to get that 
	\begin{equation}\label{F-13o}
		\Fv_{1,3}^\od \approx -\mu\uvb \p_3 \Theta_\sigma = \frac{-\mu\uvb}{\srt} \Theta'\Big(\frac{x_3-\sigma}{\srt}\Big).
	\end{equation}
	Then one has that $ G_3 = F_{1,33}^\od \approx 0. $
	For $ i=1 $ or $ 2, $ it follows from \cref{f-2-i} that
	\begin{equation}\label{f-2i}
		\int_{\pm\infty}^{x_3} f_{2,i}^\od(y_3,t) dy_3 =  \frac{-\rhob\ub_i}{2\srt} \int_{\pm\infty}^{\frac{x_3-\sigma}{\srt}} \xi \Theta'(\xi) d\xi + R_{i\pm},
	\end{equation}
	where
	\begin{align*}
		R_{i\pm} & = \frac{\ub_i}{\srt} \int_{\pm\infty}^{x_3} \Big[-\rho_-^\od \sigma'(t) -(\rho_-^\od-\rhob) \frac{x_3-\sigma}{2(t+\Lambda)}  + m_{3-}^\od \Big] \Theta'_\sigma dy_3.
		%		R_{i+} & = \frac{\ub_i}{\srt} \int^{+\infty}_{x_3} \big[ -\rho_-^\od \sigma'(t) + m_{3-}^\od \big] \Theta'_\sigma dy_3,
	\end{align*}
	Since $\Theta'>0$, then it holds that $ \abs{R_{i\pm}} \lesssim \e e^{-\alpha t} (1\mp\Theta_\sigma). $
	It follows from \cref{equ-Theta} that
	\begin{align*}
		\mu \Theta'(\xi) = -\frac{\rhob}{2} \int_{-\infty}^{\xi} \eta \Theta'(\eta) d\eta =\frac{\rhob}{2} \int_{\xi}^{+\infty} \eta \Theta'(\eta) d\eta \qquad \forall \xi\in\R.
	\end{align*}
	This, together with \cref{F-13o,f-2i}, yields that
	\begin{align*}
		F_{1,3i}^\od + \int_{-\infty}^{x_3} f_{2,i} dy_3 \approx R_{i-} \andd -F_{1,3i}^\od - \int_{x_3}^{+\infty} f_{2,i}^\od dy_3 \approx R_{i+},
	\end{align*}
	which implies that $ G_i \approx R_{i-} $ and meanwhile, $ G_i \approx R_{i+}. $ Thus, it holds that
	\begin{align*}
		\norm{G_i}_{L^2(\R)} \lesssim \e \Lambda^{\frac{1}{4}} e^{-\alpha t} + \Big(\int_{-\infty}^{\sigma} \abs{R_{i-}}^2 dx_3\Big)^{\frac{1}{2}} + \Big(\int^{+\infty}_{\sigma} \abs{R_{i+}}^2 dx_3\Big)^{\frac{1}{2}} \lesssim \e \Lambda^{\frac{1}{4}} e^{-\alpha t}.
	\end{align*}
	
	2) Then we prove the estimates of $f_0$ and $\gv.$
	Using \cref{f0} and the fact that $ (\rho_+-\rho_-, m_{3+}-m_{3-}) = O(1) \e e^{-\alpha t} $, one can easily obtain that $ \norm{f_0}_{L^2(\Omega)} \lesssim \e \Lambda^{-\frac{1}{4}} e^{-\alpha t}. $
	It follows from \cref{ansatz,uvt-1} that for $ i=1,2,3 $,
	\begin{align*}
		\p_i (\ut_i \mvt) & \approx \frac{1}{2} \big[\p_i (u_{i-} \mv_-) (1-\Theta_\sigma) + \p_i (u_{i+} \mv_+) (1+\Theta_\sigma)\big], \\
		\p_i p(\rhot) & \approx \frac{1}{2} \p_j p(\rho_-) (1-\Theta_\sigma) + \frac{1}{2} \p_i p(\rho_+) (1+\Theta_\sigma), \\
		\p_i^2 \uvt & \approx \frac{1}{2} \p_i^2 \uv_- (1-\Theta_\sigma) + \frac{1}{2} \p_i^2 \uv_+ (1+\Theta_\sigma) + \frac{ \delta_{i3} \uvb}{t+\Lambda}  \Theta''_\sigma, \\
		\p_i \dv \uvt & \approx \frac{1}{2} \p_i \dv \uv_- (1-\Theta_\sigma) + \frac{1}{2} \p_i \dv \uv_+ (1+\Theta_\sigma).
	\end{align*}
	Using the relations above in \cref{F1}, one can get that
	\begin{equation}\label{eq-F}
		\begin{aligned}
			\sum_{i=1}^{3} \p_i \Fv_{1,i} \approx -\frac{\mu \uvb}{t+\Lambda} \Theta''\Big(\frac{x_3-\sigma}{\srt}\Big).
		\end{aligned}
	\end{equation}
	On the other hand, it follows from \cref{f2} that
	\begin{align*}
		\fv_2 \approx -\frac{ \mvb}{2(t+\Lambda)} \cdot \frac{x_3-\sigma}{\srt} \Theta'\Big(\frac{x_3-\sigma}{\srt}\Big),
	\end{align*}
	which, together with \cref{equ-Theta}, yields that $ \gv = \sum_{i=1}^{3} \p_i \Fv_{1,i} + \fv_2 \approx 0. $
	The estimates of the derivatives of $f_0$ and $ \gv $ can be proved similarly. We omit the details.
\end{proof}

\vspace{.2cm}

\section{Perturbations of Velocity and Momentum}\label{Sec-app4}

\begin{proof}[Proof of \cref{Lem-rel}]
	1) We first prove \cref{rel-0,rel-1,rel-2}. Note that 
	\begin{equation}\label{zeta-1}
		\zeta^\od = \frac{1}{\rhob} \big( \p_3 \Zv + \p_3 \uv^\vs \Phi\big) + \Big[\Big(\frac{1}{\rhot}-\frac{1}{\rhob}\Big) \psi \Big]^\od - \Big[\Big(\frac{\uvt}{\rhot}-\frac{\uv^\vs}{\rhob}\Big) \phi \Big]^\od - \Big(\frac{1}{\rhot} \phi\zeta\Big)^\od.
	\end{equation}
	Then it holds that
	\begin{align}
		 \pm\norm{\zeta^\od}_{L^2(\R)} & \lesssim \pm \norm{\p_3 \Zv}_{L^2(\R)} +  \norm{\Phi \kappa}_{L^2(\R)} + \e \nu e^{-\alpha t} + \nu \norm{\zeta}_{L^2(\Omega)}, \notag
	\end{align}
	which yields \cref{rel-0} directly.
	Similarly, one has that
	\begin{align*}
	 \pm \norm{\p_3 \zeta^\od}_{L^2(\Omega)} & \lesssim \pm \norm{\p_3^2 \Zv}_{L^2(\Omega)} + (t+\Lambda)^{-\frac{1}{2}} \norm{\p_3 \Phi, \Phi \kappa}_{L^2(\Omega)}  \\
		& \quad  + \e \nu e^{-\alpha t} + \nu \norm{\p_3\phi,\p_3 \zeta}_{L^2(\Omega)},
	\end{align*}
	which gives \cref{rel-1}. Moreover, it holds that
	\begin{align*}
		\pm \norm{\p_3^2 \zeta^\od}_{L^2(\R)}
		& \lesssim \pm \norm{\p_3^3 \Zv}_{L^2(\R)} + (t+\Lambda)^{-\frac{1}{2}} \norm{\p_3^2 \Phi}_{L^2(\R)} + (t+\Lambda)^{-1} \norm{\p_3 \Phi, \Phi \kappa}_{L^2(\R)} \\
		& \quad  + \e \nu e^{-\alpha t} + \norm{\nabla \phi}_{L^4(\Omega)} \norm{\nabla\zeta}_{L^4(\Omega)} + \nu \norm{\nabla^2 \phi, \nabla^2 \zeta}_{L^2(\Omega)},
	\end{align*}
	which, together with \cref{tool4}, yields \cref{rel-2}.
	
	\vspace{.1cm}
	
	2) Then we prove \cref{rel-3,rel-4}. It follows from \cref{poincare}, \cref{psi-zeta} and \cref{est-r-3} that
	\begin{align*}
		\pm \norm{\nabla^2 \psi^\md}_{L^2(\Omega)} & \lesssim \pm \norm{\nabla^2 \zeta^\md}_{L^2(\Omega)} + \norm{\phi^\md}_{H^2(\Omega)} + \norm{\nabla^2 \mathbf{n}_3}_{L^2(\Omega)} \\
		& \lesssim \pm \norm{\nabla^2 \zeta^\md}_{L^2(\Omega)} + \nu \norm{\nabla^2 \phi^\md, \nabla^2 \zeta^\md}_{L^2(\Omega)} + \e \nu e^{-\alpha t},
	\end{align*}
	which yields \cref{rel-3}.	
	Using the identity that 
	$$\psi = \rhob \zeta + \uv^\vs \phi + \phi \zeta + \big[ (\rhot-\rhob) \zeta + (\uvt-\uv^\vs) \phi \big], $$
	one can get that
	\begin{equation}\label{app-1}
		\begin{aligned}
			\pm \norm{\nabla^3 \psi}_{L^2(\Omega)} & \lesssim \pm \norm{\nabla^3 \zeta}_{L^2(\Omega)} + \sum_{j=0}^{1} (t+\Lambda)^{\frac{j-3}{2}} \norm{\p_3^j \phi^\od}_{L^2(\R)} + \Lambda^{-1} \norm{\phi^\md}_{H^1(\Omega)} \\
			& \quad + \Lambda^{-\frac{1}{2}} \norm{\nabla^2 \phi}_{L^2(\Omega)} + \norm{\nabla^3 \phi}_{L^2(\Omega)} + \nu \norm{\nabla^3 \zeta}_{L^2(\Omega)} \\
			& \quad + \norm{\nabla\phi}_{L^4(\Omega)} \norm{\nabla^2 \zeta}_{L^4(\Omega)} + \e e^{-\alpha t} \norm{\phi,\zeta}_{H^3(\Omega)}. 
		\end{aligned}
	\end{equation}
	It follows from \cref{tool3}\textsubscript{1} that $$ \norm{\nabla\phi}_{L^4(\Omega)}  \lesssim \nu^{\frac{1}{2}} \norm{\nabla^2 \phi}_{L^2(\Omega)}^{\frac{1}{2}}, \qquad \norm{\nabla^2 \zeta}_{L^4(\Omega)} \lesssim \nu^{\frac{1}{2}} \norm{\nabla^3 \zeta}_{L^2(\Omega)}^{\frac{1}{2}}. $$
	Applying these two inequalities to \cref{app-1} and using \cref{poincare}, one can obtain \cref{rel-4}.
	
\end{proof}

\end{appendices}

\vspace{.2cm}

\textbf{Acknowledgments.} F. Huang is partially supported by the National Natural Science Foundation of China No. 12288201 and the National Key R\&D Program of China No. 2021YFA1000800.
Z. Xin is supported in part by the Zheng Ge Ru Foundation, Hong Kong RGC Earmarked
Research Grants CUHK-14301421, CUHK-14300917, CUHK-14302819 and CUHK-14300819, the
key project of NSFC (Grant No. 12131010) and by Guangdong Basic and Applied Basic Research
Foundation 2020B1515310002. 
Q. Yuan is partially supported by the National Natural Science Foundation of China 12201614, Youth Innovation Promotion Association of CAS 2022003 and CAS Project for Young Scientists in Basic Research YSBR-031.

\vspace{.3cm}

%%%%%%%%%%%%%%%%%%%%%%%%%%%%%%%%%%%%%%%%%%%
%%%%% Bib
%%%%%%%%%%%%%%%%%%%%%%%%%%%%%%%%%%%%%%%%%%%

\providecommand{\bysame}{\leavevmode\hbox to3em{\hrulefill}\thinspace}
\providecommand{\MR}{\relax\ifhmode\unskip\space\fi MR }
\providecommand{\MRhref}[2]{%
	\href{http://www.ams.org/mathscinet-getitem?mr=#1}{#2}
}
\providecommand{\href}[2]{#2}

%\bibliographystyle{amsplain}
%\bibliography{bibli}

\end{sloppypar}

%\vspace{1.5cm}

\end{document}